\documentclass[a4paper,reqno,10pt]{amsart}

\usepackage{amsmath}
\usepackage{amsthm}
\usepackage{amssymb}
\usepackage{amscd} %diagrams

\usepackage{bbm} %bold numbers
\usepackage{mathrsfs} %nice calligraphy
\usepackage{mathabx}
\usepackage{txfonts}

\usepackage[T1]{fontenc} %allows for hyphenation of words with accents
\usepackage{lmodern}

\usepackage{graphicx} %uncomment for inclusion of graphic images

\usepackage{verbatim} %DO NOT comment this line (comment environment)

\usepackage{sqrcaps} % font

\usepackage{url}
\usepackage{comment}
\usepackage[svgnames]{xcolor}
  \usepackage{pdfcolmk}

\usepackage{tikz}
\usepackage{pifont}
\swapnumbers %theorems numbering appears in front of 'Theorem'

\newtheoremstyle{exercise} %for books or class notes
  {3pt} %space above
  {3pt} %space below
  {\small\rmfamily} %body font
  {
} %indent amount(empty=no indent,\parindent=para indent)
  {\rmfamily\scshape} %thm head font
  {.} %punctuation after thm head
  {.5em} %space after thm head: " " = normal interword space;
  {} %thm head spec (can be left empty, meaning `normal')

\newtheoremstyle{newplain}
  {5pt}
  {5pt}
  {\itshape}
  {}
  {\rmfamily\scshape}
  {. ---}
  {.5em}
  {}

\newtheoremstyle{newremark}
  {5pt}
  {5pt}
  {\rmfamily}
  {}
  {\rmfamily\scshape}
  {. ---}
  {.5em}
  {}

\theoremstyle{newplain}
\newtheorem*{Theorem*}{Theorem} %no numbering for Theorem*
\newtheorem*{Claim*}{Claim}

\theoremstyle{newplain}
\newtheorem{Theorem}{Theorem}
\newtheorem{Lemma}[Theorem]{Lemma}

\newtheorem{Proposition}[Theorem]{Proposition}
\newtheorem{Conjecture}[Theorem]{Conjecture}
\newtheorem{Definition}[Theorem]{Definition}

\theoremstyle{newremark}
\newtheorem{Empty}[Theorem]{}
\newtheorem{Remark}[Theorem]{Remark}
\newtheorem{Example}[Theorem]{Example}
\newtheorem{Claim}[Theorem]{Claim}

\theoremstyle{exercise}

\numberwithin{Theorem}{section}
\numberwithin{Exercise}{section}

\newcommand{\ind}{\mathbbm{1}} %indicatrix function

\newcommand{\R}{\mathbf{R}} %real numbers
\newcommand{\Rm}{\R^m}
\newcommand{\Rn}{\R^n}

\renewcommand{\setminus}{\thicksim} %set theoretic difference à la Federer

\newcommand{\bbN}{\mathbb{N}}

\newcommand{\calB}{\mathscr{B}}

\newcommand{\calE}{\mathscr{E}}
\newcommand{\calF}{\mathscr{F}}
\newcommand{\calG}{\mathscr{G}}
\newcommand{\calH}{\mathscr{H}}
\newcommand{\calI}{\mathscr{I}}

\newcommand{\calK}{\mathscr{K}}
\newcommand{\calL}{\mathscr{L}}
\newcommand{\calM}{\mathscr{M}}
\newcommand{\calN}{\mathscr{N}}

\newcommand{\calP}{\mathscr{P}}

\newcommand{\calR}{\mathscr{R}}

\newcommand{\caN}{\boldsymbol{\frak N}}

\renewcommand{\frak}{\mathcal}

\newcommand{\frA}{\frak A}

\newcommand{\frU}{\frak U}

\newcommand{\balpha}{\boldsymbol{\alpha}}
\newcommand{\bbeta}{\boldsymbol{\beta}}
\newcommand{\bchi}{\boldsymbol{\chi}}
\newcommand{\bdelta}{\boldsymbol{\delta}}

\newcommand{\bgamma}{\boldsymbol{\gamma}}

\newcommand{\bmu}{\boldsymbol{\mu}}

\newcommand{\btheta}{\boldsymbol{\theta}^*}

\newcommand{\bB}{\mathbf{B}}

\newcommand{\bC}{\mathbf{C}}
\newcommand{\bG}{\mathbf{G}}
\newcommand{\bN}{\mathbf{N}}
\newcommand{\bO}{\mathbf{O}}

\newcommand{\bU}{\mathbf{U}}

\newcommand{\bZ}{\mathbf{Z}}

\newcommand{\bc}{\pmb{c}}

\newcommand{\bg}{\pmb{g}}

\DeclareMathOperator{\rmap}{\mathrm{ap}} %approximate limits and derivatives
\DeclareMathOperator{\rmBdry}{\mathrm{Bdry}} %boundary
\DeclareMathOperator{\rmcard}{\mathrm{card}} %cardinal
\DeclareMathOperator{\rmconv}{\mathrm{conv}} %convex hull
\DeclareMathOperator{\rmdiam}{\mathrm{diam}} %diameter

\DeclareMathOperator{\rmdist}{\mathrm{dist}} %distance
\newcommand{\rmesssup}{\mathrm{ess\,sup}} %essential supremum
\DeclareMathOperator{\rmHom}{\mathrm{Hom}} %homomorphisms
\DeclareMathOperator{\rmid}{\mathrm{id}} %the identity map
\DeclareMathOperator{\rmim}{\mathrm{im}} %image
\DeclareMathOperator{\rmint}{\mathrm{int}} %intérieur
\DeclareMathOperator{\rmInt}{\mathrm{Int}} %interior
\DeclareMathOperator{\rminv}{\mathrm{inv}} %intérieur
\DeclareMathOperator{\rmLip}{\mathrm{Lip}} %Lipschitz constant

\DeclareMathOperator{\rmloc}{\mathrm{loc}}
\DeclareMathOperator{\rmrank}{\mathrm{rank}} %rank
\DeclareMathOperator{\rmset}{\mathrm{set}} %set
\DeclareMathOperator{\rmspan}{\mathrm{span}} %span
\DeclareMathOperator{\rmspt}{\mathrm{spt}} %support
\DeclareMathOperator{\rmsquare}{\mathrm{square}} %square
\DeclareMathOperator{\rmsupp}{\mathrm{supp}} %support, as well
\DeclareMathOperator{\rmTan}{\mathrm{Tan}} %tangent space or measure
\newcommand{\rmI}{\mathrm{I}}
\newcommand{\rmII}{\mathrm{II}}

\newcommand{\blseg}{\pmb{[}}
\newcommand{\brseg}{\pmb{]}}

\newcommand{\hel} {
\hskip2.5pt{\vrule height7pt width.5pt depth0pt}
\hskip-.2pt\vbox{\hrule height.5pt width7pt depth0pt}
\, }

\newcommand{\shel} {
\hskip2.5pt{\vrule height5pt width.4pt depth0pt}
\hskip-.2pt\vbox{\hrule height.4pt width5pt depth0pt}
\, }

\newcommand{\cone} {
\times \negthickspace \negthickspace \negthickspace \times \medspace }

\newcommand{\lseg}{\boldsymbol{[}\!\boldsymbol{[}}
\newcommand{\rseg}{\boldsymbol{]}\!\boldsymbol{]}}

\def\XXint#1#2#3{{%
\setbox0=\hbox{$#1{#2#3}{\int}$}
\vcenter{\hbox{$#2#3$}}\kern-.5\wd0}}

\newcommand{\lno}{\left\bracevert}
\newcommand{\rno}{\right\bracevert}

\renewcommand{\em}{\bf}

\newcommand{\veps}{\varepsilon}

\newcommand{\vphi}{\varphi}

\newcommand{\la}{\langle}
\newcommand{\ra}{\rangle}

\renewcommand{\leq}{\leqslant}
\renewcommand{\geq}{\geqslant}
\renewcommand{\subset}{\subseteq}
\renewcommand{\supset}{\supseteq}

\newlength{\drop}

\usepackage{trajan}

\begin{document}

%=================
% TITLE AND AUTHOR
%=================

%\thispagestyle{empty}
%\titleAT
%\clearpage

\title[Mass minimizing $G$ chains]{On the existence of mass minimizing rectifiable\\ $G$ chains in finite dimensional normed spaces}

\def\curraddrname{{\itshape On leave of absence from:}}

\author[Th. De Pauw]{Thierry De Pauw}
\address{School of Mathematical Sciences\\
Shanghai Key Laboratory of PMMP\\ 
East China Normal University\\
500 Dongchuang Road\\
Shanghai 200062\\
P.R. of China}
\curraddr{Universit\'e Paris Diderot\\ 
Sorbonne Universit\'e\\
CNRS\\ 
Institut de Math\'ematiques de Jussieu -- Paris Rive Gauche, IMJ-PRG\\
F-75013, Paris\\
France}
\email{thdepauw@math.ecnu.edu.cn}

\author[I. Vasilyev]{Ioann Vasilyev}
\address{St. Petersburg Department of Steklov Mathematical Institute\\ Russian Academy
of Sciences (PDMI RAS)\\ Russia\\
{\itshape Current address:} Universit\'e Paris-Est Marne\\ LAMA (UMR 8050)\\ 5 Boulevard Descartes\\ 77454
Champs sur Marne\\ France}
\email{ioann.vasilyev@imj-prg.fr}

\keywords{Hausdorff measure, integral geometry, rectifiable chains, Plateau problem}

\subjclass[2010]{Primary 49Q15,49Q20,52A21,28A75,52A38; Secondary 52A40,49J45}

\thanks{The first author was partially supported by the Science and Technology Commission of Shanghai (No. 18dz2271000).}

\date{December 9th, 2018}

%=========
% ABSTRACT
%=========

\begin{abstract}
We introduce the notion of density contractor of dimension $m$ in a finite dimensional normed space $X$. If $m+1=\dim X$ this includes the area contracting projectors on hyperplanes whose existence was established by H. Busemann. If $m=2$, density contractors are an ersatz for such projectors and their existence, established here, follows from work by D. Burago and S. Ivanov. Once density contractors are available, the corresponding Plateau problem admits a solution among rectifiable $G$ chains, regardless of the group of coefficients $G$. This is obtained as a consequence of the lower semicontinuity of the $m$ dimensional Hausdorff mass, of which we offer two proofs. One of these is based on a new type of integral geometric measure. 
\end{abstract}

\maketitle

%===========
% DEDICATION
%===========

\begin{comment}
\begin{flushright}
\textsf{\textit{Dedicated to: Joseph Plateau, Arno, Hooverphonic, Hercule
Poirot,\\
Kim Clijsters, Pierre Rapsat, Stromae,\\ and non Jef t'es pas tout seul.}}
\end{flushright}
\end{comment}

%==================
% TABLE OF CONTENTS
%==================

\tableofcontents
%\newpage

%==============================
% THE MEAT --- OR SO ONE THINKS
%==============================

\section{Foreword}

The classical tools of Geometric Measure Theory, developed by {\sc H. Federer} and {\sc W.H. Fleming} offer the following setting for the Plateau problem. Here the ambient space $X = \ell_2^n$ is the $n$ dimensional Euclidean space and $m$ is an integer comprised between 1 and $n-1$. Given an $m-1$ dimensional rectifiable cycle with integer multiplicity and compact support $B \in \calR_{m-1}(X,\bZ)$, we consider the variational problem
\begin{equation*}
(\calP)\begin{cases}
\text{minimize } \calM(T) \\
\text{among } T \in \calR_m(X,\bZ) \text{ with } \partial T = B \,.
\end{cases}
\end{equation*} 
As $\partial B = 0$ there are competitors indeed, one being provided by the cone construction. The mass of a competitor $T$ is
\begin{equation}
\label{intro.1}
\calM(T) = \int_A | \vec{T} | d\calH^m
\end{equation}
where $A$ is the underlying countably $(\calH^m,m)$ rectifiable set on which $T$ concentrates, $|\vec{T}|$ is the norm of its algebraic multiplicity and $\calH^m$ is the Hausdorff measure associated with the ambient Euclidean structure.
\par 
Problem $(\calP)$ admits a minimizing sequence $T_1,T_2,\ldots$ with each $T_k$ supported in the convex hull of $\rmsupp B$. This is because we may replace, if needed $T_k$ by $\pi_\# T_k$ where $\pi : X \to X$ is the nearest point projection on the convex hull of $\rmsupp B$. Since $\rmLip \pi \leq 1$ it follows that $\calM(\pi_\# T_k) \leq \calM(T_k)$. Such sequence is relatively compact with respect to {\sc H. Whitney}'s flat norm -- a consequence of the deformation theorem and rectifiability theorem, \cite{FED.FLE.60}. Its accumulation points $T$ are minimizers of problem $(\calP)$ because the mass is lower semicontinuous with respect to convergence in the flat norm:
\begin{equation}
\label{intro.2}
\calM(T) \leq \liminf_k \calM(T_k) \,.
\end{equation}
The convergence in flat norm,
\begin{multline}
\label{intro.3}
\calF(T-T_k) = \inf \big\{ \calM(R) + \calM(S) : R \in \calR_m(X,\bZ), S \in \calR_{m+1}(X,\bZ) \\ \text{ and } T-T_k = T + \partial S \big \}
\end{multline}
is in this case equivalent to the weak* convergence as currents.
\par 
The same method applies to proving the existence of a mass minimizing chain when, in $(\calP)$ the group of coefficients $\bZ$ is replaced by a cyclic group $\bZ_q$, \cite[4.2.26]{GMT}, a finite group \cite{FLE.66}, or more generally a locally compact normed Abelian group $G$ that does not contain any nontrivial curve of finite length, according to the work of {\sc B. White} \cite{WHI.99.deformation,WHI.99.rectifiability}. For instance $G=\bZ_2$ allows for considering a nonorientable minimal surface bounded by the M\"obius strip, whereas $G=\bZ_3$ allows for considering a minimal surface bounded by a triple M\"obius trip, singular along a spine where infinitesimally three half planes meet at equal angles.
\par 
We now describe two ways of understanding why mass is lower semicontinuous with respect to convergence in the flat norm.
\begin{enumerate}
\item[(1)] {\it Orthogonal projectors are contractions,} in particular they reduce mass: If $W \subset X$ is an $m$ dimensional affine subspace, $\pi_W : X \to X$ is the orthogonal projector onto $W$ and $\sigma$ is an $m$ dimensional simplex in $X$ then $\calH^m(\pi_W(\sigma)) \leq \calH^m(\sigma)$. To see how this is related to the lower semicontinuity of mass assume (as we may, according to the strong approximation theorem \cite[4.2.22]{GMT}) that $T,T_1,T_2,\ldots$ are all polyhedral chains. Thus $T - (T_k + R_k) = \partial S_k$ with $\calM(R_k)$ small. If we are able to infer from this that $\calM(T) \leq \calM(T_k + R_k)$ then we will be done. Simplifying even further we assume that $T = g \lseg \sigma \rseg$ is associated with a single simplex contained in some $m$ dimensional affine subspace $W \subset X$, and we write $T_k + R_k = \sum_j g_j \lseg \sigma_j \rseg$  with the $\sigma_j$ nonoverlapping. Since $\pi_{W\,\#}(T - (T_k+R_k))$ is an $m$ dimensional cycle with compact support in $W$ the constancy theorem implies that $T = \pi_\# T = \pi_\#(T_k + R_k)$. Therefore
\begin{multline*}
\calM(T) = \calM \left( \sum_j \pi_{W\,\#} (g_j \lseg \sigma j \rseg) \right) \\\leq \sum_j |g_j| \calH^m \left( \pi_{W\,\#}(\sigma_j) \right) \leq \sum_j |g_j| \calH^m(\sigma_j) = \calM(T_k + R_k)
\end{multline*}
\item[(2)] {\it Hausdorff measure coincides with integral geometric measure.} In other words the mass of $T$ can be recovered from the mass of its 0 dimensional slices $\la T , \pi , y \ra$ corresponding to all orthogonal map $\pi \in \bO^*(n,m)$ from $\Rn$ to $\Rm$, and $y \in \Rm$. Specifically if $\btheta_{n,m}$ denotes an $\bO(n)$ invariant probability measure on $\bO^*(n,m)$ then
\begin{equation}
\label{intro.10}
\calM(T) = \bbeta_1(n,m)^{-1} \int_{\bO^*(n,m)} d\btheta_{n,m}(\pi) \int_{\Rm} \calM(\la T , \pi , y \ra) d\calL^m(y)
\end{equation}
for some suitable constant $\bbeta_1(n,m) > 0$, see \cite[2.10.15 and 3.2.26]{GMT}. The mass $\calM(Z)$ of a 0 dimensional chain $\calR_0(X,\bZ) \ni Z = \sum_j g_j \bdelta_{x_j}$ (where the $x_j$'s are distinct) is simply the finite sum $\calM(Z) = \sum_j |g_j|$. Now $\calM : \calR_0(X,G) \to \R$ is lower semicontinuous with respect to convergence in the flat norm, and if $\calF(T-T_k) \to 0$ rapidly then for every $\pi$ one infers that $\calF(\la T,\pi,y \ra - \la T_k,\pi,y\ra) \to 0$ for almost every $y$. It therefore ensues from Fatou's Lemma that
\begin{multline*}
\calM(T) \leq   \bbeta_1(n,m)^{-1} \int_{\bO^*(n,m)} d\btheta_{n,m}(\pi) \int_{\Rm} \liminf_k\calM(\la T_k , \pi , y \ra) d\calL^m(y)\\
\leq \liminf_k \bbeta_1(n,m)^{-1} \int_{\bO^*(n,m)} d\btheta_{n,m}(\pi) \int_{\Rm} \calM(\la T_k , \pi , y \ra) d\calL^m(y)\\  = \liminf_k \calM(T_k) \,.
\end{multline*}
\end{enumerate}
\par 
Of interest to us in this paper is the case when the Euclidean norm of the ambient space is replaced with another, arbitrary norm $\|\cdot\|$. Thus we consider $G$ chains in a finite dimensional normed space $X$. The notions of $m$ dimensional rectifiable $G$ chain and their convergence in flat norm are not affected, thus the compactness tool for applying the direct method of the calculus of variations is available without modification. Hausdorff measure  $\calH^m_{\|\cdot\|}$ however has changed in the process, to the extent that the corresponding new Hausdorff mass
\begin{equation*}
\calM_H(T) = \int_A | \vec{T} | d\calH_{\|\cdot\|}^m \,,
\end{equation*}
$T \in \calR_m(X,\bZ)$, is not known in general to be lower semicontinuous. In light of the two methods evoked above we note that:
\begin{enumerate}
\item[(1)] If $W \subset X$ is an $m$ dimensional subspace, there does not need to exist a projector $\pi : X \to W$ with $\rmLip \pi \leq 1$. Perhaps the simplest case is when $X = \ell^3_\infty$ and $W = X \cap \{ (x_1,x_2,x_3) : x_1 + x_2 + x_3 = 0 \}$.
\item[(2)] In a specific sense there is no integral geometric formula available for the Hausdorff measure in a normed space, according to {\sc R. Schneider} \cite{SCH.01}.
\end{enumerate}
\par 
Notwithstanding the first observation, {\sc H. Busemann} established the existence, in codimension 1 of projectors that reduce the area. More precisely if $W \subset X$ is an affine subspace of dimension $m= \dim X -1$ then there exists a projector $\pi : X \to X$ onto $W$ such that $\calH^m_{\|\cdot\|}(\pi(\sigma)) \leq \calH^m_{\|\cdot\|}(\sigma)$ whenever $\sigma$ is an $m$ dimensional simplex in $X$. This is enough for implementing successfully method (1) and proving the existence of an $\calM_H$ minimizing rectifiable $G$ chain when $m = \dim X -1$. In fact in \ref{existence.convex} below we show that there exists at least one minimizer supported in the convex hull of its boundary, even though not all minimizers have this extra property.
\par 
If $1 < m < \dim X - 1$ the existence of projectors (onto $m$ dimensional affine subspaces) that reduce area remains conjectural. {\sc D. Burago} and {\sc S. Ivanov} were able \cite{BUR.IVA.12} to design a calibration method in case $m=2$ proving that flat 2 dimensional disks minimize their Hausdorff measure $\calH^2_{\|\cdot\|}$ among Lipschitz competing surfaces with coefficients in $G=\R$ or $G=\bZ_2$. 
\par 
Even though in that case there may not exist area reducing projectors we show here how their method leads to the notion of density contractors, i.e.\ a probability measure on the space of linear mappings $\pi : X \to X$ with rank at most 2, playing the analogous role of a single projector as far as comparing areas is concerned. The fact that we need to consider arbitrary linear mappings $\pi$, possibly with large Lipschitz constant, is a source of some technical complications.
\par 
In fact we develop in the present paper an axiomatic theory of density contractors in arbitrary dimension $m$ and establish their existence when $m=2$ or $m=\dim X -1$, see \ref{existence.d.c}. We prove how density contractors can be used to compare the Hausdorff mass of chains \ref{411}, and the Hausdorff measure of sets \ref{413}. We show that once density contractors exist in dimension $m$, the Plateau problem admits $\calM_H$ minimizers with coefficients in arbitrary locally compact groups of coefficients that satisfy {B. White}'s condition stated above, see \ref{special.existence}. The minimizers obtained here, following \cite{AMB.SCH} see \ref{general.existence}, are merely locally rectifiable and we do not know in general whether they may be further required to be compactly supported, unless $m = \dim X - 1$. 
\par 
The main concern is indeed to show that $\calM_H : \calR_m(X,G) \to  \R$ is lower semicontinuous with respect to convergence in the flat norm. In the spirit of this introduction we obtain two proofs of this fact, granted the existence of density contractors.
\begin{enumerate}
\item[(1)] In Section 4 we establish a necessary and sufficient condition for the lower semicontinuity of $\calM_H$ in terms of a triangle inequality of polyhedral cycles, see \ref{theorem.1} and \ref{theorem.2}. In Section 5, where the notion of density contractor is introduced, we give a simple proof that their existence implies this triangle inequality for cycles, see \ref{tic.d.c}.
\item[(2)] In Section 6 we define a new type of integral geometric measure associated with density contractors, \ref{53}. Rather than integrating the mass of slices as in \eqref{intro.10}, which would correspond to a measure $\calI^m_1$ in the notation of \cite{GMT}, we consider a measure obtained from Caratheodory's method II with the local supremum of the averaged measure of projected sets constituting a fine cover (averaged using a density contractor), corresponding to $\calI^m_\infty$ in the notation of \cite{GMT}. Rather than essential suprema we consider locally actual suprema, thus more in the spirit of the Gross measure, \cite[2.10.4(1)]{GMT}. We show that the corresponding newly defined Gross measure coincides with the Hausdorff measure of rectifiable sets, see \ref{57}. Furthermore, in the spirit of \cite{BOU.DEP} we show that the corresponding Gross mass defined in \ref{61} is lower semicontinuous with respect to flat convergence, \ref{64}.
\end{enumerate}
\par 
It is our pleasure to record helpful discussions with Ph. Bouafia and G. Godefroy.

\section{Preliminaries}

\begin{Empty}[Hausdorff distance]
\label{h.d}
If $X$ is a metric space and $A,B \subset X$ are compact we define their {\em Hausdorff distance} as
\begin{equation*}
\rmdist_{\calH}(A,B) = \inf \big\{ \delta > 0 : A \subset \bB(B,\delta) \text{ and } B  \subset \bB(A,\delta) \big \}
\end{equation*}
where $\bB(A,\delta) = X \cap \{ x : \rmdist(x,A) \leq \delta \}$. The following are rather obvious.
\begin{enumerate}
\item[(1)] {\it Suppose $f,f_1,f_2,\ldots$ are continuous mappings from $X$ to $Y$ such that $f_k \to f$ locally uniformly as $k \to \infty$, and $A \subset X$ is compact. It follows that $\rmdist_{\calH}(f_k(A),f(A)) \to 0$. as $k \to \infty$.}
\item[(2)] {\it Suppose $A,A_1,A_2,\ldots$ are compact subsets of $X$ such that $\rmdist_{\calH}(A_k,A) \to 0$ as $k \to \infty$, and  $f : X \to Y$ is continuous. Then $\rmdist_{\calH}(f(A_k),f(A)) \to 0$ as $k \to \infty$.}
\end{enumerate}
\end{Empty}

\begin{Empty}
\label{b.m}
Let $X$ be a finite dimensional real linear space. Given a norm $\nu$ on $X$ we let $B_\nu = X \cap \{ x : \nu(x) \leq 1\}$ denote its unit (closed) ball. If $\nu_1$ and $\nu_2$ are two norms on $X$,  we define 
$
\delta(\nu_1,\nu_2) = \inf \left\{ \lambda > 0 : B_{\nu_1} \subset \lambda B_{\nu_2} \text{ and } B_{\nu_2} \subset \lambda B_{\nu_1} \right\}
$.
One readily checks that $\log \delta$ is well defined and a distance on the set of norms on $X$. When we will consider a convergent sequence of norms on $X$ it will be relative to this distance. 

\begin{enumerate}
\item[(1)] {\it Given $\nu,\nu_1,\nu_2,\ldots$ a sequence of norms on $X$, $\delta(\nu,\nu_j) \to 1$ as $j \to \infty$ if and only if $\nu_j(x) \to \nu(x)$ as $j \to \infty$ for every $x \in X$.}
\end{enumerate}

This follows from the relation $\nu(x) = \inf \{ t > 0 : x \in t.B_{\nu} \}$ and the fact that the pointwise convergence of the $\nu_j$, $j=1,2,\ldots$, to $\nu$ implies their uniform convergence on bounded subsets of $X$, due to their convexity.

\begin{enumerate}
\item[(2)] {\it Given two norms $\nu_1$ and $\nu_2$ on $X$ and $W \subset X$ a linear subspace, one has $\delta(\nu_1|_W,\nu_2|_W) \leq \delta(\nu_1,\nu_2)$.}
\end{enumerate}
\end{Empty}

\begin{Empty}
\label{haar}
Given a norm $\nu$ on $X$ we associate with it the Hausdorff outer measures $\calH^m_\nu$, $m \in \{1,\ldots,n\}$, see for instance \cite[2.10.2]{GMT}. It is the case that (the restriction to $\calB(X)$ of) $\calH^{\dim X}_\nu$ is a Haar measure on $X$. It therefore follows from the uniqueness of Haar measure and the Borel regularity of Hausdorff measures that if $\nu_1$ and $\nu_2$ are norms on $X$ then there exists $0 < \beta(\nu_1,\nu_2) < \infty$ such that $\calH^{\dim X}_{\nu_1} = \beta(\nu_1,\nu_2) \calH^{\dim X}_{\nu_2}$. In order to estimate $\beta(\nu_1,\nu_2)$ in terms of the distance between $\nu_1$ and $\nu_2$ we recall the following central result of H. Busemann \cite{BUS.47} (see also \cite[Lemma 6]{KIR.94} or \cite[7.3.6]{THOMPSON}):
\begin{equation}
\label{eq.busemann.1}
\calH^{\dim X}_\nu(B_\nu) = \balpha(\dim X)
\end{equation}
is independent of $\nu$. The identity $\balpha(m) = \calH^{\dim X}_{\nu_1}(B_{\nu_1}) = \beta(\nu_1,\nu_2) \calH^{\dim X}_{\nu_2}(B_{\nu_1})$, the definition of $\delta(\nu_1,\nu_2)$ and the homogeneity of Hausdorff measures imply that
\begin{equation}
\label{eq.fri.1}
\delta(\nu_1,\nu_2)^{- \dim X} \leq \beta(\nu_1,\nu_2) \leq \delta(\nu_1,\nu_2)^{\dim X} \,.
\end{equation}
We will also use the following observation $\calH^{\dim X}_{\nu,\delta} = \calH^{\dim X}_{\nu}$ for all $0 < \delta \leq \infty$, where the former are the $\delta$ size approximating outer measures.
\end{Empty}

\begin{Empty}[Grassmannian]
\label{grass}
Given $X$ and $m \in \{1,\ldots,\dim X - 1 \}$ we let $\bG_m(X)$ denote the set of $m$ dimensional linear subspaces of $X$. In order to give $\bG_m(X)$ a topology we equip first $X$ with an inner product $\la \cdot , \cdot \ra$ and corresponding norm $| \cdot |$. Given $W \in \bG_m(X)$ we let $\pi_W : X \to X$ denote the orthogonal projection onto $W$. We then define $d(W_1,W_2) = \vvvert\pi_{W_1}-\pi_{W_2}\vvvert$, $W_1,W_2 \in \bG_m(X)$, where $\vvvert \cdot \vvvert$ is the operator norm. In the remaining part of this paper, an inner product structure will always be fixed on $X$.
\par 
Let $W_1,W_2,\ldots$ be a sequence in $\bG_m(X)$. Choose $e_1^k,\ldots,e_n^k$  to be an orthonormal basis of $X$ such that $\rmspan\{e_1^k,\ldots,e^k_m\}=W_k$. Possibly passing to a subsequence we may assume $e_j^k \to e_j$ as $k \to \infty$, $j=1,\ldots,n$, and clearly $e_1,\ldots,e_n$ is an orthonormal basis of $X$. Define $W = \rmspan \{e_1,\ldots,e_m\} \in \bG_m(X)$. Given $x \in X$ a simple calculation shows that $\left| (\pi_W - \pi_{W_k})(x)\right| \leq |x| \sum_{j=1}^n \left| (\pi_W - \pi_{W_k})(e_j)\right| \leq 2 n |x| \sum_{j=1}^n\left| e^k_j-e_j\right|$. Thus $d(W,W_k) \to 0$.
\end{Empty}

\begin{Empty}[Busemann-Hausdorff density]
Now given two norms $\nu_1$ and $\nu_2$ on $X$ and $W \in \bG_m(X)$ we claim that there exists $0 < \beta(\nu_1,\nu_2,W) < \infty$ such that $\calH^m_{\nu_1} \hel W = \beta(\nu_1,\nu_2,W) \calH^m_{\nu_2} \hel W$. This is because (the restriction to $\calB(W)$ of) $\calH^m_{\nu_i} \hel W$, $i=1,2$, are both Haar measures on $W$. 

In the remaining part of this paper we will consider a given fixed norm $\|\cdot\|$ on $X$. Comparing it with the underlying reference inner product norm $|\cdot|$, we define the {\bf Busemann-Hausdorff density} function $\psi : \bG_m(X) \to \R$ as $\psi(W) = \beta(\|\cdot\|,|\cdot|,W)$, corresponding to each $m \in \{1,\ldots,\dim X -1\}$ (we omit both $\|\cdot\|$ and $m$ in the notation for $\psi$). It follows that $\psi(W)$ is characterized by the identity $\calH^m_{\|\cdot\|}(W \cap E) = \psi(W) \calH^m_{|\cdot|}(W \cap E)$ for any Borel $E \subset X$ such that one of the (and therefore both) both measures appearing there are nonzero and finite. Letting respectively $E = B_{|\cdot|}$ and $E = B_{\|\cdot\|}$ and referring to \eqref{eq.busemann.1} we find that
\begin{equation}
\psi(W) = \frac{\calH^m_{\|\cdot\|}\left(W \cap B_{|\cdot|}\right)}{\balpha(m)} = \frac{\balpha(m)}{\calH^m_{|\cdot|}\left(W \cap B_{\|\cdot\|}\right)} \,.
\end{equation}
\end{Empty}

\begin{Proposition}
\label{25}
The Busemann-Hausdorff density $\psi : \bG_m(X) \to \R$ is continuous.
\end{Proposition}

\begin{proof}
Let $W,W_1,W_2,\ldots$ be members of $\bG_m(X)$ such that $W_j \to W$ as $j \to \infty$. We consider linear isometries $f_j : \ell_2^m \to (X,|\cdot|)$, $j=1,2,\ldots$, whose range is $W_j$. By compactness a subsequence of $f_1,f_2,\ldots$, still denoted $f_1,f_2,\ldots$, converges pointwise to some linear isometry $f : \ell_2^m \to (X,|\cdot|)$. Since $W_j \to W$ as $j \to \infty$ we infer that the range of $f$ is $W$. Now we define norms $\nu,\nu_1,\nu_2,\ldots$ on $\Rm$ by $\nu(\xi) = \|f(\xi)\|$ and $\nu_j(\xi) = \|f_j(\xi)\|$, $\xi \in \Rm$, so that $f : (\Rm,\nu) \to (W,\|\cdot\|)$ and $f_j : (\Rm,\nu_j) \to (W_j,\|\cdot\|)$ are also isometries. Therefore $f_* \calH^m_{|\cdot|} = \calH^m_{|\cdot|} \hel W$, $f_{j\,*} \calH^m_{|\cdot|} = \calH^m_{|\cdot|} \hel W_j$, $f_* \calH^m_{\nu} = \calH^m_{\|\cdot\|} \hel W$ and $f_{j\,*} \calH^m_{\nu_j} = \calH^m_{\|\cdot\|} \hel W_j$, $j=1,2,\ldots$. Since $\nu_j \to \nu$ pointwise it follows from \ref{b.m}(1) and \eqref{eq.fri.1} that $\beta_{\ell_2^m}(\nu_j,|\cdot|) \to \beta_{\ell_2^m}(\nu,|\cdot|)$ and hence also $\beta_X(\|\cdot\|,|\cdot|,W_j) \to \beta_X(\|\cdot\|,|\cdot|,W)$ as $j \to \infty$. As the argument can be repeated for any subsequence of the original sequence $W_1,W_2,\ldots$ the proof is complete.
\end{proof}

\section{Plateau problem}

In this section we describe the setting in which we will state the Plateau problem. Groups of polyhedral, rectifiable and flat chains have been studied in \cite{GMT,FLE.66,WHI.99.rectifiability,WHI.99.deformation,DEP.HAR.07}. We now provide a very quick overview.

\begin{Empty}[Polyhedral, rectifiable, and flat $G$ chains]
We let $(G,\lno\cdot \rno)$ be an Abelian group equipped with a norm $\lno \cdot \rno$ which turns it into a complete metric space. As before $(X,\|\cdot\|)$ is a finite dimensional normed linear space and $m \in \{1,\ldots,\dim X - 1\}$. With an $m$ dimensional oriented simplex $\sigma$ in $X$ and a group element $g \in G$ we associate an object $g \lseg \sigma \rseg$. We consider equivalence classes of formal sums of these $\sum_{k=1}^\kappa g_k \lseg \sigma_k \rseg$. The equivalence identifies $(-g) \lseg - \sigma \rseg = g \lseg \sigma \rseg$ (where $- \sigma$ has orientation opposite to that of $\sigma$) and $g\lseg \sigma \rseg =  \sum_{k=1}^\kappa g \lseg \sigma_k \rseg$ if $\sigma_1,\ldots,\sigma_k$ is a simplicial partition of $\sigma$ with the same orientation. We let $\calP_m(X,G)$ denote the group whose elements are these {\em polyhedral $G$ chains of dimension $m$}. A boundary operator $\partial : \calP_m(X,G) \to \calP_{m-1}(X,G)$ is defined as usual. Given $P \in \calP_m(X,G)$ we define its {\em Hausdorff mass} by the formula
\begin{equation*}
\calM_H(P) = \sum_{k=1}^\kappa \lno g_k \rno \calH^m_{\|\cdot\|}(\sigma_k) \,,
\end{equation*}
where $P = \sum_{k=1}^\kappa g_k \lseg \sigma_k \rseg$ and the $\sigma_1,\ldots,\sigma_k$ are chosen to be nonoverlapping.
This definition does not depend upon the choice of such a decomposition of $P$. If we want to insist that the mass is defined with respect to the norm $\|\cdot\|$ we will write $\calM_{H,\|\cdot\|}(P)$ instead of $\calM_H(P)$ to avoid confusion, as another mass $\calM_{H,|\cdot|}(P)$ is readily available as well. Both are equivalent. The {\em flat norm} of $P \in \calP_m(X,G)$ is the defined as
\begin{multline*}
\calF(P) = \inf \big\{ \calM_{H,|\cdot|}(Q) +  \calM_{H,|\cdot|}(Q) : Q \in \calP_m(X,G), R \in \calP_{m+1}(X,G)\\ \text{ and } P = Q + \partial R \big\} \,.
\end{multline*}
The completion of $\calP_m(X,G)$ with respect to $\calF$ is the group $\calF_m(X,G)$ whose members are called {\em flat $G$ chains of dimension $m$}. One important feature of this group is that with a Lipschitz map $f : X \to Y$ one can associate a push-forward morphism $f_\# : \calF_m(X,G) \to \calF_m(Y,G)$ that commutes with $\partial$. The $m$ dimensional Lipschitz $G$ chains of $X$ are then defined to be the members of $\calF_m(X,G)$ of the form $\sum_{k=1}^\kappa f_{k \, \#} P_k$ where $P_k \in \calP_m(\ell^m_\infty,G)$ and $f_k : \ell_\infty^m \to X$ is Lipschitz, $k=1,\ldots,\kappa$. One further defines the subgroup $\calR_m(X,G) \subset \calF_m(X,G)$ whose members, called {\em rectifiable $G$ chains of dimension $m$}, have compact support and are limits in the $\calM_H$ norm of sequences of Lipschitz $G$ chains of dimension $m$. With such $T \in \calR_m(X,G)$ is associated a countably $(\calH^m_{\|\cdot\|},m)$ rectifiable Borel subset $\rmset_m \|T\| \subset X$ and for almost every $x \in  \rmset_m \|T\|$ a so-called $G$ orientation $\bg(x)$ which consists of a nonzero group element and an orientation of the approximate tangent space of $ \rmset_m \|T\|$ at $x$. The {\em Hausdorff mass} of $T$ is defined as
\begin{equation*}
\calM_H(T) = \int_{\rmset_m \|T\|} \lno \bg(x) \rno d\calH^m_{\|\cdot\|}(x) \,.
\end{equation*}
The definition is consistent with the previous one in case $T$ is polyhedral. 
\end{Empty}

\begin{Empty}[Hausdorff mass and lower semicontinuity]
The Hausdorff Euclidean mass $\calM_{H,|\cdot|} : \calR_m(X,G) \to \R$ is lower semicontinuous with respect to $\calF$ convergence whereas the lower semicontinuity of  $\calM_{H,\|\cdot\|} : \calR_m(X,G) \to \R$ is unknown in general. This is the main topic of the present paper. One can of course consider the $\calF$ lower semicontinuity extension
\begin{equation*}
\calM_{H,|\cdot|} : \calF_m(X,G) \to [0,\infty] \,.
\end{equation*}
Here $\calM_{H,|\cdot|}(T)$ is the infimum of $\liminf_k \calM_{H,|\cdot|}(T_k)$ corresponding to all sequences $(T_k)_k$ in $\calR_m(X,G)$ that $\calF$ converge to $T$. Of course it may occur that $\calM_{H,|\cdot|}(T) = \infty$.
\end{Empty}

\begin{Empty}[Locally rectifiable $G$ chains]
If $T \in \calF_m(X,G)$ and $u : X \to \R$ is Lipschitz then the restriction $T \hel \{ u < r \}$ is defined for $\calL^1$ almost every $r \in \R$. We here define $\calR^{\rmloc}_m(X,G)$ to be the subgroup of $\calF_m(X,G)$ of those $T$ such that for every bounded open set $U \subset X$, letting $(x) = \rmdist(x,U)$, there exists $R \in \calR_m(X,G)$ such that $T \hel \{ u <r \} = R \hel \{ u < r \}$ for $\calL^1$ almost every $0 < r \leq 1$. These do not necessarily have compact support. We call these {\em locally rectifiable $G$ chains of dimension $m$} and we define
\begin{equation*}
\calM_H(T) = \sup \big\{ \calM_H(R \hel U) : \text{ $U$ and $R$ are as above} \big\} \,.
\end{equation*}
One checks this is consistent with the preceding number in case $\|\cdot\| = |\cdot|$.
\end{Empty}

\begin{Empty}[Compactness]
The following is a consequence of the deformation theorem proved in this context by B. White \cite{WHI.99.deformation}. If $K\subset X$ is compact, $\lambda > 0$ and $(G,\lno \cdot \rno)$ is locally compact then 
\begin{equation*}
\calF_m(X,G) \cap \big \{ T : \rmsupp T \subset K \text{ and } \calM_{H,|\cdot|}(T) + \calM_{|\cdot|}(\partial T) \leq \lambda \big\}
\end{equation*}
is $\calF$ compact.
\end{Empty}

\begin{Empty}[White groups]
\label{compactness}
We say that $(G,\lno \cdot \rno)$ is a {\em White group} if $G$ does not contain any nontrivial curve of finite length. Of course if $(G,\lno \cdot \rno)$ is totally disconnected then it is White, for instance $G=\bZ$, $G=\bZ_q$ for any $q=2,3,\ldots$, or $G=\bZ_2^{\bN}$ the Cantor group. The group $(\R,|\cdot|)$ is not White, but $(\R,|\cdot|^p)$ is whenever $0 < p < 1$. The reason for considering those groups is the following result, see \cite{WHI.99.rectifiability} : If $(G,\lno \cdot \rno)$ is a White group, $T \in \calF_m(X,G)$ and $\calM_{H,|\cdot|}(T) < \infty$ then $T \in \calR_m(X,G)$. The cases $G=\bZ$ and $G=\bZ_q$ go back to \cite[4.2.16(3) and 4.2.26]{GMT}. Together with the preceding number we obtain the following compactness result: If $K\subset X$ is compact, $\lambda > 0$ and $(G,\lno \cdot \rno)$ is locally compact and White then 
\begin{equation*}
\calR_m(X,G) \cap \big \{ T : \rmsupp T \subset K \text{ and } \calM_{H,|\cdot|}(T) + \calM_{|\cdot|}(\partial T) \leq \lambda \big\}
\end{equation*}
is $\calF$ compact.
\end{Empty}

We now state a particular Plateau problem: that of minimizing the {\em Hausdorff mass} in the context of rectifiable $G$ chains in a finite dimensional normed space.

\begin{Theorem}
\label{general.existence}
Assume that 
\begin{enumerate}
\item[(A)] $(X,\|\cdot\|)$ is a finite dimensional normed space and $(G,\lno \cdot \rno)$ is an Abelian normed locally compact White group;
\item[(B)] $1 \leq m \leq \dim X - 1$;
\item[(C)] $\calM_H : \calR_m(X,G) \to \R$ is lower semicontinuous with respect to $\calF$ convergence;
\item[(D)] $B \in \calR_{m-1}(X,G)$ and $\partial B = 0$.
\end{enumerate}
It follows that the Plateau problem
\begin{equation*}
(\calP) \begin{cases}
\text{minimize } \calM_H(T) \\
\text{among } T \in \calR_{m}^{\rmloc}(X,G) \text{ such that } \partial T = B
\end{cases}
\end{equation*}
admits a solution.
\end{Theorem}

\begin{Empty}
Two comments are in order.
\begin{enumerate}
\item[(1)] The nontrivial assumption is (C), that the Hausdorff mass be lower semicontinuous. The remaining sections of this paper are devoted to establishing hypothesis (C) in case $m=2$ or $m = \dim X -1$ and $G$ is arbitrary.
\item[(2)] There is another technical problem with applying the direct method of calculus of variations. In order to apply the compactness theorem \ref{compactness} one would need to exhibit a minimizing sequence supported in some given compact set. {\it In case $X$ is Euclidean} this is done as follows. If $T$ is such that $\partial T = B$ and $\rmsupp B \subset \bB(0,R)$, one considers the nearest point projection $f : X \to \bB(0,R)$. Since $\rmLip f \leq 1$ one has $\calM_H(f_\# T) \leq \calM_H(T)$, and also $\partial f_\# T = f_\# \partial T = B$. Pushing forward along $f$ each member of a given minimizing sequence produces a new minimizing sequence of chains all supported in $\bB(0,R)$. In case the norm is not Euclidean this process cannot be repeated as the map
\begin{equation*}
f : X \to \bB(0,1) : x \mapsto \begin{cases}
x & \text{ if }\|x\| \leq 1 \\
\frac{x}{\|x\|} & \text{ if }\|x\| > 1 
\end{cases}
\end{equation*}
has merely $\rmLip f \leq 2$ and no better in general (as can be seen by considering $\|x\|_1 = \sum_{j=1}^N |x_j|$ in $\R^N$). This leads one to show that a minimizing sequence is necessarily {\it tight} and perform a diagonal argument as in \cite{AMB.SCH}. Here we are not able to do better in general. Yet in the last section, see \ref{existence.convex}, we show that if $m = \dim X -1$ then there exists a minimizing chain with compact support (contained in the convex hull of $\rmsupp B$). 
\end{enumerate}
\end{Empty}

\begin{proof}
Let $T_1,T_2,\ldots$ be a minimizing sequence. We may readily assume each $T_k \in \calR_m(X,G)$. As in \cite[Proof of Theorem 1.1]{AMB.SCH} we start by showing that this sequence is tight:
\begin{equation}
\label{eq.thu.100}
\lim_{r \to \infty} \sup_{k=1,2,\ldots} \calM_H\left(T_k \hel \bB_{\|\cdot\|}(0,r) \right) = 0 \,.
\end{equation}
Let $\gamma = \inf (\calP)$, $\veps > 0$, $\alpha = \veps^{-1} \sup_{k=1,2,\ldots} \calM_H(T_k-T_1) < \infty$, choose $r > 0$ such that $\calM_H \left(T_1 \hel \bB_{\|\cdot\|}(0,r)^c \right) \leq \veps$ and choose $h > 0$ such that 
\begin{equation}
\label{eq.thu.101}
\alpha < \int_r^{r+h} r^{-1}d\calL^1(r) \,.
\end{equation}
Fix $k=1,2,\ldots$. Since
\begin{equation}
\label{eq.thu.102}
\int_r^{r+h} \calM_H(\la T_k - T_1 , \|\cdot\| , \rho \ra) d\calL^1(\rho) \leq 2C_{m,1}\calM_H(T_k-T_1)
\end{equation} 
according to \cite[3.7.1(9)]{DEP.HAR.07}, there exists $r \leq R_k \leq r+h$ such that $\la T_k - T_1 , \|\cdot\| , \rho \ra \in \calR_{m-1}(X,G)$ and
\begin{equation}
\calM_H(\la T_k - T_1 , \|\cdot\| , \rho \ra)  \leq R_k^{-1}  2C_{m,1}\calM_H(T_k-T_1) \alpha^{-1} \,,
\end{equation}
for if not the combination of \eqref{eq.thu.101} and \eqref{eq.thu.102} would lead to a contradiction. We let $S_k = \bdelta_0 \cone \la T_k - T_1 , \|\cdot\| , \rho \ra \in \calR_m(X,G)$ and we recall that $\calM_H(S_k) \leq 2C'_m R_k \calM_H(\la T_k - T_1 , \|\cdot\| , \rho \ra) \leq 4C'_mC_{m,1} \veps$ (see for instance \cite[\S 2.5]{DEP.13.approx}) and $\partial S_k = \la T_k - T_1 , \|\cdot\| , \rho \ra$ since $\partial \la T_k - T_1 , \|\cdot\| , \rho \ra = - \la \partial (T_k - T_1) , \|\cdot\| , \rho \ra = 0$. Since also $\la T_k - T_1 , \|\cdot\| , \rho \ra = \partial (T_k-T_1) \hel \bB_{\|\cdot\|}(0,R_k)$ we see that $\partial R_k = B$, where $R_k = (T_k-T_1) \hel \bB_{\|\cdot\|}(0,R_k) + T_1 - S_k$. Therefore $\gamma \leq \calM_H(R_k)$. Now $R_k = T_k \hel \bB_{\|\cdot\|}(0,R_k) + T_1 \hel \bB_{\|\cdot\|}(0,R_k)^c + S_k$ and accordingly
\begin{equation}
\begin{split}
\gamma \leq \calM_H(R_k)& \leq \calM_H \left( T_k \hel \bB_{\|\cdot\|}(0,r+h) \right) + \calM_H \left( T_1 \hel \bB_{\|\cdot\|}(0,r)^c\right) + \calM_H(S_k)\\
& \leq \calM_H(T_k) - \calM_H \left( T_k \hel \bB_{\|\cdot\|}(0,r+h)^c \right) + (1 + 4C'_mC_{m,1})\veps \,.
\end{split}
\end{equation}
Choosing $k_0$ such that $\calM_H(T_k) - \gamma \leq \veps$ whenever $k \geq k_0$ it follows that
\begin{equation}
\sup_{k=k_0,k_0+1,\ldots} \calM_H \left( T_k \hel \bB_{\|\cdot\|}(0,r+h)^c \right) \leq (2 + 4C'_mC_{m,1})\veps \,.
\end{equation}
As $\lim_{\rho \to \infty} \calM_H\left( T_k \hel \bB_{\|\cdot\|}(0,\rho)^c \right) = 0$ for each $k=1,\ldots,k_0-1$ it readily follows that
\begin{equation}
\limsup_{\rho \to \infty} \sup_{k=1,2,\ldots} \calM_H\left( T_k \hel \bB_{\|\cdot\|}(0,\rho)^c \right) \leq (2 + 4C'_mC_{m,1})\veps \,.
\end{equation}
Since $\veps > 0$ is arbitrary \eqref{eq.thu.100} is established.
\par 
We consider all $j=1,2,\ldots$ large enough for $\rmsupp B \subset \bU_{\|\cdot\|}(0,j)$. Applying inductively the compactness theorem to suitable subsequences of $\big(T_k \hel \bB_{\|\cdot\|}(0,j)\big)_k$ the diagonal argument yields one subsequence of $(T_k)_k$, still denoted as such, and one sequence $(\hat{T}_j)_j$ of members of $\calR_m(X,G)$ such that $\calF \left( \hat{T}_j - T_k \hel \bB_{\|\cdot\|}(0,j)\right) \to 0$ as $k \to \infty$, for every $j$. We show that $(\hat{T}_j)_j$ is $\calF$ Cauchy. Given $\veps > 0$ we choose $r > 0$ such that the supremum in \eqref{eq.thu.100} is bounded above by $\veps$. Given $j_1,j_2$ larger than $r$, there exists $k_i$, $i=1,2$, such that $\calF \left( \hat{T}_{j_i} - T_k \hel \bB_{\|\cdot\|}(0,j_i) \right) \leq \veps$ whenever $k \geq k_i$. For $k=\max\{k_1,k_2\}$ we thus have 
\begin{equation*}
\begin{split}
\calF\left(\hat{T}_{j_1} - \hat{T}_{j_2} \right) & \leq \calF \left(\hat{T}_{j_1} - T_k \hel \bB_{\|\cdot\|}(0,j_1) \right) \\&\quad\quad + \calF \left(T_k \hel \bB_{\|\cdot\|}(0,j_1) - T_k \hel \bB_{\|\cdot\|}(0,j_2) \right)\\&\quad\quad + \calF \left( T_k \hel \bB_{\|\cdot\|}(0,j_2) - \hat{T}_{j_2}\right)\\
&\leq 3 \veps \,.
\end{split}
\end{equation*}
Let $T \in \calF_m(X,G)$ be the $\calF$ limit of this sequence. Given a bounded open subset $U \subset X$, $u(x) = \rmdist_{\|\cdot\|}(U,x)$, and $r_0 > 0$ we select $j_0$ sufficiently large for $\bB_{\|\cdot\|}(U,r_0) \subset  \bU_{\|\cdot\|}(0,j_0)$. We first recall that for $\calL^1$ almost every $0< r \leq r_0$ one has $\calF \left( T \hel \{ u < r \} - \hat{T}_j \hel \{ u < r \}\right) \to 0$ as $j \to \infty$, \cite[5.2.3(2)]{DEP.HAR.07}. Furthermore if $j \geq j_0$ then $\hat{T}_j \hel \{ u < r \} = \hat{T}_{j_0} \hel \{u<r\}$ by the definition of $\hat{T}_j$. Therefore $T \hel \{ u < r \} = \hat{T}_{j_0} \hel \{ u < r \}$, and thus $T \in \calR_m^{\rmloc}(X,G)$. Now clearly $\partial \hat{T}_j = B$ for each $j$ and thus $\partial T = B$. Finally it remains to show that $\calM_H(T) \leq \gamma$. Let $U \subset X$ be open and bounded and choose $j_0$ such that $\bB_{\|\cdot\|}(U,1) \subset \bU_{\|\cdot\|}(0,j_0)$. Now if $R \in \calR_m(X,G)$ is so that $ R \hel \{ u < r \} = T \hel \{ U < r \}$ for almost every $0 < r \leq r_0$ then also $R  \hel \{ u < r \} = \hat{T}_{j_0} \hel \{ u < r \}$ for such $r$. Therefore
\begin{equation*}
\calM_H(R \hel U) \leq \calM_H\left(\hat{T}_{j_0}\right) \leq \liminf_k \calM_H \left( T_k \hel \bB_{\|\cdot\|}(0,j_0) \right) \leq \gamma \,,
\end{equation*}
according to hypothesis (C).
Since $U$ is arbitrary we conclude that $\calM_H(T) = \gamma$.
\end{proof}

\section{A general criterion for lower semicontinuity}

We will need the following technical fact. Given $T,T_1,T_2,\ldots$ in $\calF_m(X,G)$ we say that $T_1,T_2,\ldots$ {\em converges rapidly} to $T$ if $\sum_j \calF(T_j-T) < \infty$. Of course any convergent sequence admits a rapidly convergent subsequence. The following is a consequence for instance of \cite[5.2.1(3)]{DEP.HAR.07}.

\begin{Proposition}
\label{rapid} Assume that $T,T_1,T_2,\ldots$ belong to $\calF_m(X;G)$ and that $(T_j)_j$ converges rapidly to $T$. If $u : X \to \R$ is Lipschitzian then $\calF(T \hel \{u > r \} - T_j \hel \{u > r \}) \to 0$ as $j \to \infty$, for almost every $r \in \R$.
\end{Proposition}

\begin{Theorem}
\label{theorem.1}
The following are equivalent.
\begin{enumerate}
\item[(1)] The triple $\big((X,\|\cdot\|),(G,\lno \cdot \rno),m\big)$ satisfies the triangle inequality for cycles.
\item[(2)] The Hausdorff mass $\calM_H : \calP_m(X;G) \to \R_+$ is lower semicontinuous with respect to $\calF$ convergence.
\end{enumerate}
\end{Theorem}

\begin{proof}
In order to establish that $(1) \Rightarrow (2)$ we let $P,P_1,P_2,\ldots$ belong to $\calP_m(X;G)$ and be such that $\calF(P-P_j) \to 0$ as $j \to \infty$. Thus $P-P_j = Q_j + \partial R_j$, $Q_j \in \calP_m(X;G)$, $R_j \in \calP_{m+1}(X;G)$ and $\calM_H(Q_j) + \calM_H(R_j) \to 0$ as $j \to \infty$. Considering first the case when $P = g \lseg \sigma \rseg$, we notice that $\partial (P-P_j-Q_j) = \partial(\partial R_j)=0$ whence $\calM_H(P) \leq \calM_H(P_j + Q_j)$ follows from hypothesis (1) for all $j=1,2,\ldots$. Accordingly, $\calM_H(P) \leq \liminf_j \calM_H(P_j + Q_j) = \liminf_j \calM_H(P_j)$. Turning to the general case we write $P= \sum_{k=1}^\kappa g_k \lseg \sigma_k \rseg$ where the $\sigma_k$ are nonoverlapping. Letting $W_k \in \bG_m(X)$ have a translate containing $\sigma_k$, we define $U_k = \mathring{\sigma}_k + V_k$ where $\mathring{\sigma}_k$ is the relative interior of $\sigma_k$ and $V_k$ is a convex polyhedral neighborhood of zero in some complementary subspace of $W_k$. Choosing those $V_k$ small enough we can assume that $U_1,\ldots,U_\kappa$ are pairwise disjoint. We next define $u_k = \rmdist(\cdot,U_k^c)$. We observe that (A) the sets $\sigma_k \cap \{ u_k > r \}$ are simplicial (i.e. of the type $\sigma_{k,r}$) and (B) the sets $\{ u_k > r \}$ are (convex open) polyhedra. Given $\veps > 0$ we choose $\rho > 0$ such that $\calH^m(\sigma_k \cap \{ u_k \leq r \}) \leq \veps \kappa^{-1} \lno g_k \rno^{-1}$ whenever $0 < r \leq \rho$ and $k=1,\ldots,\kappa$, whence also $\calM_H(P) \leq \veps + \sum_{k=1}^\kappa \calM_H(P \hel \{u_k > r\})$. Replacing the original sequence $(P_j)_j$ by a subsequence if necessary we may assume that $\liminf_j \calM_H(P_j) = \lim_j \calM_H(P_j)$ and that $(P_j)_j$ converges rapidly to $P$. Thus there exists $0 < r \leq \rho$ such that $\calF(P \hel \{u_k > r \} - P_j \hel \{u_k > r\}) \to 0$ as $j \to \infty$ for each $k=1,\ldots,\kappa$ according to \ref{rapid}. It then follows from observations (A) and (B) above, and from the first case of this proof, that
\begin{equation*}
\calM_H(P \hel \{u_k > r \}) \leq \liminf_j \calM_H( P_j \hel \{u_k > r\})\,.
\end{equation*}
In turn,
\begin{multline*}
\calM_H(P) - \veps \leq \sum_{k=1}^\kappa \calM_H(P \hel \{u_k > r\}) \leq \sum_{k=1}^\kappa \liminf_j \calM_H (P_j \hel \{u_k > r \}) \\
\leq \liminf_j \sum_{k=1}^\kappa \calM_H (P_j \hel \{u_k > r \}) \leq \lim_j \calM_H(P_j) \,.
\end{multline*}
Since $\veps > 0$ is arbitrary the proof is complete.
\par 
We prove $(2) \Rightarrow (1)$ by contraposition. Let $\sigma_1,\ldots,\sigma_\kappa$ be simplexes in $X$ and $g_1,\ldots,g_\kappa$ be elements of $G$ such that
\begin{equation}
\label{eq.1}
\eta := \calM_H(g_1\lseg\sigma_1\rseg) - \sum_{k=2}^\kappa \calM_H(g_k \lseg \sigma_k \rseg) > 0 
\end{equation}
and $\partial (P-Q)=0$ where we have abbreviated $P = g_1 \lseg \sigma_1 \rseg$ and $Q=\sum_{k=2}^\kappa g_k \lseg \sigma_k \rseg$. There is no restriction to assume that $\sigma$ does not overlap any of the $\sigma_k$, $k=2,\ldots,\kappa$. If indeed $\sigma_1$ and some $\sigma_k$ overlap then we replace the summand $g_k \lseg \sigma_k \rseg$ in $Q$ by a sum $\begin{displaystyle}S_{k,\delta} = \sum_{i=0}^m g_k\bdelta_{p_{k,\delta}} \cone \lseg \tau_i \rseg\end{displaystyle}$ where $\tau_i$ runs over the facets of $\sigma_k$ (properly oriented so that the cycle condition is preserved) and where $p_{k,\delta}$ belongs to $X$, but not to the affine plane containing $\sigma_k$, and is a distance $\delta$ apart from (say) the barycenter of $\mathring{\sigma}_k$. It follows from \ref{25} that $\calM_H(S_{k,\delta}) \to \calM_H(g_k \lseg \sigma_k \rseg)$ as $\delta \to 0$. Thus \eqref{eq.1} is preserved upon choosing $\delta > 0$ small enough, and clearly the new simplexes replacing $\sigma_k$ do not overlap $\sigma_1$.
\par 
Let $x_0,x_1,\ldots,x_m \in X$ be such that $\lseg \sigma_1 \rseg = \lseg x_0,x_1,\ldots,x_m \rseg$. For an integer $j$ we define $I_j = \bbN^m \cap \{ (\alpha_1,\ldots,\alpha_m) : \alpha_1 + \ldots + \alpha_m \geq j-1 \}$ and corresponding to each $\alpha \in I_j$ we define an affine bijection $f_{j,\alpha} : X \to X$ by the formula
\begin{equation*}
f_{j,\alpha}(x) = \left( x_0 + \sum_{i=1}^m \alpha_i \left(\frac{x_i-x_0}{j}\right) \right) + \frac{x-x_0}{j} \,,
\end{equation*}
and we notice that $\rmLip f_{j,\alpha} = j^{-1}$ and that $\calM_H(f_{j,\alpha\,\#}g_1 \lseg \sigma_1 \rseg) = j^{-m} \calM_H( g_1 \lseg \sigma_1 \rseg)$. We introduce
\begin{equation*}
P_j = P + \sum_{\alpha \in I_j} f_{j,\alpha\,\#}(Q-P) \in \calP_m(X;G) \,.
\end{equation*}
Since $\partial (Q-P)=0$ there exists $R \in \calP_{m+1}(X;G)$ such that $Q-P = \partial R$. Therefore
\begin{multline*}
\calF(P_j-P) = \calF \left( \partial \sum_{\alpha \in I_j} f_{j,\alpha\,\#} R \right) \leq \calM_H\left( \sum_{\alpha \in I_j} f_{j,\alpha\,\#} R \right) \\
\leq (\rmcard I_j) \max_{\alpha \in I_j}(\rmLip f_{j,\alpha})^{m+1} \calM_H(R) \leq j^m \left(\frac{1}{j}\right)^{m+1} \calM_H(R)\,,
\end{multline*}
whence $\calF(P_j-P) \to 0$ as $j \to \infty$. In order to estimate the mass of $P_j$ we note that 
\begin{equation*}
P_j = P - \sum_{\alpha \in I_j} f_{j,\alpha\,\#} P + \sum_{\alpha \in I_j} f_{j,\alpha\,\#} Q
\end{equation*}
and that the $f_{j,\alpha\,\#} P$ are nonoverlapping subchains of $P$. Therefore
\begin{equation*}
\begin{split}
\calM_H(P_j) & \leq  \calM_H(P) -  \sum_{\alpha \in I_j} \calM_H \left(f_{j,\alpha\,\#}P \right) +  \sum_{\alpha \in I_j} \calM_H \left(f_{j,\alpha\,\#}Q \right)\\
& \leq \calM_H(P) \left( 1 - \left( \frac{1}{j}\right)^m (\rmcard I_j) \right) + \calM_H(Q) \left( \frac{1}{j}\right)^m (\rmcard I_j) \\
& = \calM_H(P) - \eta \left( \frac{1}{j}\right)^m (\rmcard I_j) \,.
\end{split}
\end{equation*}
Since $\lim_j (\rmcard I_j) j^{-m} = 1/2$ we infer that
\begin{equation*}
\liminf_j \calM_H(P_j) \leq \calM_H(P) - \frac{\eta}{2} < \calM_H(P) \,.
\end{equation*}
\end{proof}

We will need two versions of approximation of rectifiable $G$ chains. One is established in \cite[Theorem 4.2]{DEP.13.approx} and the other one is below.

\begin{Theorem}[Approximation Theorem]
\label{approximation}
Let $T \in \calR_m(X,G)$ and $\veps > 0$. There exist $P \in \calP_m(X,G)$ and $f : X \to X$ a diffeomorphism of class $C^1$ with the following properties:
\begin{enumerate}
\item[(1)] $\max \{ \rmLip_{\|\cdot\|}(f) , \rmLip_{\|\cdot\|}(f^{-1}) \} \leq 1 + \veps$;
\item[(2)] $\|f(x)-x\| \leq \veps$ for every $x \in X$;
\item[(3)] $f(x) = x$ whenever $\rmdist_{\|\cdot\|}(x,\rmspt T) \geq \veps$;
\item[(4)] $\calM_H(P - f_\# T) \leq \veps$.
\end{enumerate}
\end{Theorem}

There are three minor differences with \cite[4.2.20]{GMT} which require some comments: the coefficients group $G$ is not necessarily $\bZ$; we merely hypothesize that $T$ be rectifiable (thus $\partial T$ could have infinite mass and conclusion (4) does not involve the normal mass); and our conclusions involve the $\|\cdot\|$ metric in the ambient space $X$ (of course that would not seriously affect conclusions (2), (3) and (4) but there is something to say about conclusion (1)). 

\begin{proof}[Sketch of proof]
The proof uses \cite[3.1.23]{GMT} whose only metric aspects regard the Lipschitz constants of $f$ and $f^{-1}$, and the balls $\bU(b,r)$ and $\bU(b,tr)$. We fix a Euclidean structure $|\cdot|$ on $X$. A careful inspection of the proof of \cite[3.1.23]{GMT} reveals that it holds with Euclidean balls $\bU(b,r)$ and $\bU(b,tr)$ unchanged and Lipschitz constants $\rmLip_{\|\cdot\|}(f)$ and $\rmLip_{\|\cdot\|}(f^{-1})$ with respect to the ambient norm $\|\cdot\|$. Indeed bounding from above these Lipschitz constants involves only estimating the operator norm of $Df(x) - \rmid_X$; this requires an extra factor accounting for the fact $\|\cdot\|$ and the Euclidean norm are equivalent, which can be absorbed in the definition of $\veps$.
\par 
The remaining part of the proof mimics that of \cite[4.2.19]{GMT}. We let $\|T\| = \lno \bg \rno \calH^m_{\|\cdot\|} \hel A$, $A = \rmset_m \|T\|$. As $A$ is countably $(\calH^m_{|\cdot|},m)$ rectifiable there exist (at most) countably many $m$ dimensional $C^1$ submanifolds $M_j$ of $X$, $j \in J$, such that $\|T\|(\Rn \setminus \cup_{j \in J} M_j)=0$. We classically check that for $\|T\|$ almost every $x \in \Rn$ there exists $j(x) \in J$ such that $x \in M_{j(x)}$ and $\Theta^m(\|T\| \hel (\Rn \setminus M_{j(x)}),x)=0$. Given $\hat{\veps} > 0$, for all such $x$ there exists $r(x) > 0$ such that for each $0 < r < r(x)$
\begin{enumerate}
\item[(A)] the above version of \cite[3.1.23]{GMT} applies with $t=(1+\hat{\veps})^{-1}$ at scale $r$ and point $x$ to $M_{j(x)}$;
\item[(B)] $\calM_H\left(T \hel \bB(x,tr) - T \hel M_{j(x)} \cap \bB(x,tr)\right) \leq \hat{\veps} \calM_H\left(T \hel \bB(x,tr)\right)$;
\item[(C)] $\calM_H\left(T \hel \bB(x,tr) - T \hel \bB(x,r)\right) \leq 2 C (1-t^m) \calM_H\left(T \hel \bB(x,r)\right)$ (where $C > 0$ is such that $\calH^m_{\|\cdot\|} \leq C \calH^m_{|\cdot|}$).
\end{enumerate}  
According to the Besicovitch-Vitali covering Theorem there exists a disjointed family of balls $\bB(x_k,r_k)$, $k=1,2,\ldots$, whose centers are as before and $0 < r_k < r(x_k)$, and $\|T\|(\Rn \setminus \cup_{k=1}^\infty \bB(x_k,r_k)) =0$. Thus $\|T\|(\Rn \setminus \cup_{k=1}^\kappa \bB(x_k,r_k)) \leq \hat{\veps}$ for some $\kappa$. For each $k=1,\ldots,\kappa$ we associate with $t = 1- \hat{\veps}$, $M_{j(x_k)}$, $x_k$ and $r_k$ the $C^1$ diffeomorphism $f_k$ of $X$ according to (the above version of) Proposition \cite[3.1.23]{GMT}, and we infer from its last conclusion that $f_{k\,\#}(T \hel M_{j(x_k)} \cap \bB(x_k,tr_k))$ is an $m$ dimensional $G$ chain supported in an affine $m$ dimensional subspace of $X$ and thus corresponds to a $G$ valued $L_1$ function, therefore $\calM_H(f_{k\,\#}(T \hel M_{j(x_k)} \cap \bB(x_k,tr_k)) - P_k) \leq \hat{\veps}\kappa^{-1}$ for some $P_k \in \calP_m(X,G)$ as in \cite[Lemma 3.2]{DEP.HAR.14}. Letting $f : X \to X$ coincide with $f_k$ in the ball $\bB(x_k,r_k)$, $k=1,\ldots,\kappa$, and with $\rmid_X$ otherwise, one readily checks that conclusions (1), (2) and (3) hold. Finally,
\begin{equation*}
\begin{split}
\calM_H(P - f_\#T) & \leq \sum_{k=1}^\kappa \calM_H \left( P_k - f_{k\,\#} (T \hel M_{j(x_k)} \cap \bB(x_k,r_k)) \right) \\
&\quad\quad + \sum_{k=1}^\kappa \calM_H \left( T \hel \bB(x_k,tr_k) - T \hel M_{j(x_k)} \cap \bB(x_k,tr_k)\right) \\
&\quad\quad + \calM_H \left( f_\# T \hel \left( \Rn \setminus \cup_{k=1}^\kappa \bB(x_k,r_k) \right)\right) \\
& \leq \hat{\veps} + \hat{\veps} (1+\hat{\veps})^m \calM_H(T) + (1+\hat{\veps})^m \hat{\veps}\,.
\end{split}
\end{equation*}
\end{proof}

The following is inspired by the proof of \cite[5.1.5]{GMT}.

\begin{Theorem}
\label{theorem.2}
The following are equivalent.
\begin{enumerate}
\item[(1)] The triple $\big((X,\|\cdot\|),(G,\lno \cdot \rno),m\big)$ satisfies the triangle inequality for cycles.
\item[(2)] The Hausdorff mass $\calM_H : \calR_m(X;G) \to \R_+$ is lower semicontinuous with respect to $\calF$ convergence.
\end{enumerate}
\end{Theorem}

\begin{proof}
That $(2) \Rightarrow (1)$ follows from Theorem \ref{theorem.1}. We now prove the reciprocal proposition holds. Let $T,T_1,T_2,\ldots$ be members of $\calR_m(X,G)$ such that $\calF(T-T_j) \to 0$ as $j \to \infty$. We first establish the conclusion in the particular case when $T \in \calP_m(X,G)$ is polyhedral. According to \cite[Theorem 4.2]{DEP.13.approx} there are $P_j \in \calP_m(X,G)$ such that $\calF(T_j-P_j) \leq j^{-1}$ and $\calM_H(P_j) \leq j^{-1} + \calM_H(T_j)$, $j=1,2,\ldots$. Therefore $\calF(T-P_j) \to 0$ as $j \to \infty$ and it follows from Theorem \ref{theorem.1} that
\begin{equation*}
\calM_H(T) \leq \liminf_j \calM_H(P_j) \leq \liminf_j \calM_H(T_j) \,.
\end{equation*}
Turning to the general case we associate with $T$ and $\veps > 0$ a polyhedral chain $P$ and a $C^1$ diffeomorphism $f$ as in Theorem \ref{approximation}. Letting $E \in \calR_m(X,G)$ be such that $P = f_\#T + E$ we see that $\calM_H(E) \leq \veps$. Observing that $\calF(f_\# T - f_\# T_j) \to 0$ as $j \to \infty)$ we infer that also $\calF(P - (E + f_\#T_j)) \to 0$ as $j \to \infty$ and thus 
\begin{equation*}
\calM_H(P) \leq \liminf_j \calM_H(E + f_\#T_j) \leq \veps + (1+\veps)^m \liminf_j \calM_H(T_j) 
\end{equation*}
according to the particular case treated first. Now since $T = (f^{-1})_\#(P-E)$ we conclude that
\begin{equation*}
\begin{split}
\calM_H(T) & \leq (1+\veps)^m \calM_H(P - E) \\
& \leq (1+\veps)^m \left( 2\veps +  (1+\veps)^m \liminf_j \calM_H(T_j) \right) \,.
\end{split}
\end{equation*}
As $\veps > 0$ is arbitrary the proof is complete.
\end{proof}

\section{Density contractors}

\begin{Empty}[Spaces of linear homomorphisms]
\label{s.l.h}
We let $\rmHom(X,X)$ denote the linear space of linear homomorphisms $X \to X$. It is equipped with its usual norm $\vvvert \pi \vvvert = \max \{ |\pi(x)| : x \in X \text{ and } |x| \leq 1 \}$ corresponding to the Euclidean norm $|\cdot|$ of $X$. 

Given $m \in \{1,\ldots,\dim X -1\}$ we define the metric subspace $\rmHom_m(X,X) = \rmHom(X,X) \cap \{ \pi : \rmrank \pi \leq m \}$ which is closed, as well as its own subspace $\rmHom_{m,\rminv}(X,X) = \rmHom_m(X,X) \cap \{ \pi : \rmrank \pi = m \}$ which is relatively open. We further define the map
\begin{equation*}
\rmHom_{m,\rminv}(X,X) \to \bG_m(X) : \pi \mapsto W_\pi
\end{equation*}
so that $\rmim \pi = W_\pi$ and we claim it is continuous. Indeed if $\pi,\pi_1,\pi_2,\ldots$ belong to $\rmHom_{m,\rminv}(X,X)$, $\pi_k \to \pi$ and $\rmim \pi_k = W_k$ then we choose orthonormal bases $e_1^k,\ldots,e^k_n$ of $X$ such that $\rmspan \{e^k_1,\ldots,e^k_m\} = W_k$. Each subsequence of these bases admits a subsequence (still denote the same way) such that $e^k_j \to e_j$ as $k \to \infty$, $j=1,\ldots,n$, where $e_1,\ldots,e_n$ is some orthonormal basis of $X$. Given $x \in X$ write $\pi_k(x) = \sum_{j=1}^m t^k_j e^k_j$ with $\sum_{j=1}^m (t^k_j)^2 \leq |x|^2 \Gamma^2$ where $\Gamma = \sup_k \vvvert \pi_k \vvvert < \infty$. If $i=m+1,\ldots,n$ then $\la e_i , \pi(x) \ra = \lim_k \la x , \pi_k(x) \ra = \lim_k \sum_{j=1}^m t^k_j \la e_i , e^k_j \ra = 0$. Therefore $\rmim \pi \subset \rmspan \{e_{m+1},\ldots,e_n\}^\perp = W$ where $W = \rmspan \{e_1,\ldots,e_m\}$. Since $\rmrank \pi = m$ we infer $\rmim \pi = W$. As $d(W,W_k) \to 0$ according to \ref{grass}, the asserted continuity follows. 

Finally, corresponding to $W \in \bG_m(X)$ we define $\rmHom(X,W) = \rmHom(X,X) \cap \{ \pi : \rmim \pi \subset W \}$.
\end{Empty}

\begin{Proposition}
\label{41}
Let $A \subset X$ be such that $\calH^m_{\|\cdot\|}(A) < \infty$ and define $f_A : \rmHom_m(X,X) \to \R$ by the formula $f_A(\pi) = \calH^m_{\|\cdot\|}(\pi(A))$. It follows that
\begin{enumerate}
\item[(1)] If $A$ is compact then $f_A$ is upper semicontinuous;
\item[(2)] If $A$ is compact and convex then $f_A$ is continuous;
\item[(3)] If $A$ is Borel then $f_A$ is Borel.
\end{enumerate}
\end{Proposition}

\begin{proof}
(1) We start with the following remark. Given $\pi \in \rmHom_m(X,X)$ let $W \in \bG_m(X)$ be such that $\rmim \pi \subset W$. Observe that in the definition of $\calH^m_{\|\cdot\|,\delta}(\pi(A))$, $0 < \delta \leq \infty$, one can restrict to covers of $\pi(A)$ by subsets of $W$. By means of a linear homomorphism $W \to \Rm$ one transforms this number to the $\calH^m_{\nu,\delta}$ measure of a subset of $\Rm$ with respect to some norm $\nu$ in $\Rm$, therefore $\calH^m_{\|\cdot\|,\delta}(\pi(A)) = \calH^m_{\|\cdot\|}(\pi(A))$ according to \ref{haar}. We let $\delta = \infty$. Consider a sequence $\pi_1,\pi_2,\ldots$ in $\rmHom_m(X,X)$ converging to $\pi$. Letting $\veps > 0$ we choose a cover $(E_i)_{i \in I}$ of $\pi(A)$ such that $\sum_{i \in I} \balpha(m) 2^{-m} (\rmdiam_{\|\cdot\|} E_i)^m < \veps + \calH^m_{\|\cdot\|,\infty}(\pi(A))$. Notice there is no restriction to assume the $E_i$ are open in $X$. Define $E = \cup_{i \in I} E_i$ and observe the compactness of $\pi(A)$ implies there exists $r > 0$ such that $\bU(\pi(A),r) \subset E$. Now if $k$ is sufficiently large then $\pi_k(A) \subset \bU(\pi(A),r) \subset \cup_{i \in I} E_i$ -- because $\pi_k(A) \to \pi(A)$ in Hausdorff distance according to \ref{h.d}(1) -- therefore $\calH^m_{\|\cdot\|,\infty}(\pi_k(A)) \leq \sum_{i \in I} \balpha(m) 2^{-m} (\rmdiam_{\|\cdot\|} E_i)^m < \veps + \calH^m_{\|\cdot\|,\infty}(\pi(A))$. By our initial remark it follows that $\limsup_k \calH^m_{\|\cdot\|}(\pi_k(A)) \leq \veps + \calH^m_{\|\cdot\|}(\pi(A))$. Since $\veps > 0$ is arbitrary the proof of (1) is complete.
\par 
(3) Choose a nondecreasing sequence of compact subsets of $A$, say $A_1,A_2,\ldots$ such that $\calH^m_{\|\cdot\|}(A \setminus A_j) \to 0$. Given $\pi \in \rmHom_m(X,X)$ we notice that
\begin{multline*}
0 \leq \calH^m_{\|\cdot\|}(\pi(A)) - \calH^m_{\|\cdot\|}(\pi(A_j)) = \calH^m_{\|\cdot\|}(\pi(A) \setminus \pi(A_j)) \leq \calH^m_{\|\cdot\|}(\pi(A \setminus A_j)) \\
\leq (\rmLip_{\|\cdot\|} \pi)^m \calH^m_{\|\cdot\|}(A \setminus A_j) \to 0 \,.
\end{multline*}
In other words $f_{A_j} \to f_A$ pointwise. Since each $f_{A_j}$ is upper semicontinuous according to (1) the conclusion follows.
\par 
(2) Since $f_A$ is upper semicontinuous according to (1) it remains to establish it is lower semicontinuous as well. We start with the particular case when the norm $\|\cdot\| = |\cdot|$ is Euclidean. Let $\pi,\pi_1,\pi_2,\ldots$ be members of $\rmHom_m(X,X)$ such that $\pi_k \to \pi$. Choose $W \in \bG_m(X)$ with $\rmim \pi \subset W$ and let $\pi_W$ denote the orthogonal projection onto $W$. Since $\pi_k(A) \to \pi(A)$ in Hausdorff distance, \ref{h.d}(1), $\pi_W(\pi_k(A)) \to \pi_W(\pi(A))$ as well, \ref{h.d}(2). Furthermore $\pi_W(\pi(A)) = \pi(A)$. Recall that any Haar measure on $W$, restricted to the collection of compact convex subsets of $W$, is continuous with respect to Hausdorff distance, \cite[3.2.36]{GMT}. Therefore $\lim_k \calH^m_{|\cdot|}(\pi_W(\pi_k(A)) = \calH^m_{|\cdot|}(\pi(A))$. Since we are in the Euclidean setting, $\calH^m_{|\cdot|}(\pi_W(\pi_k(A)) \leq \vvvert \pi_W \vvvert^m \calH^m_{|\cdot|}(\pi_k(A)) = \calH^m_{|\cdot|}(\pi_k(A))$, and finally $\calH^m_{|\cdot|}(\pi(A)) \leq \liminf_k \calH^m_{|\cdot|}(\pi_k(A))$ which completes the proof in case $\|\cdot\| = |\cdot|$. 
\par 
We now turn to the general case. If $\rmrank \pi < m$ then clearly $\calH^m_{\|\cdot\|}(\pi(A)) = 0 \leq \liminf_k \calH^m_{\|\cdot\|}(\pi_k(A))$. We henceforth assume that $\rmrank \pi = m$ and thus also $\rmrank \pi_k = m$ if $k$ is sufficiently large. Recalling that $\calH^m_{\|\cdot\|}(\pi(A)) = \beta(\|\cdot\|,|\cdot|,W_\pi) \calH^m_{|\cdot|}(\pi(A)) = \psi(W_\pi) \calH^m_{|\cdot|}(\pi(A))$ the lower semicontinuity in the variable $\pi$ follows from the particular case and from the continuity of $\pi \mapsto W_\pi$, \ref{s.l.h}.
\end{proof}

\begin{Definition}
A {\em density contractor on $W \in \bG_m(X)$} is a Borel probability measure $\mu$ on $\rmHom_m(X,X)$ such that 
\begin{enumerate}
\item[(1)] $\mu$ is supported in $\rmHom(X,W)$, i.e. $\mu(\rmHom_m(X,X) \setminus \rmHom(X,W)) = 0$;
\item[(2)] If $V \in \bG_m(X)$ and $A \subset V$ is Borel then 
\begin{equation*}
\int_{\rmHom_m(X,X)} \calH^m_{\|\cdot\|}(\pi(A)) d\mu(\pi) \leq \calH^m_{\|\cdot\|}(A) \,,
\end{equation*}
with equality when $V=W$.
\end{enumerate}
When there will be several norms under consideration on $X$, in order to avoid confusion we will insist that $\mu$ is a density contractor on $W$ {\em with respect to $\calH^m_{\|\cdot\|}$}.
\end{Definition}

In view of the preceding section, density contractors are useful for the following reason.

\begin{Theorem}
\label{tic.d.c}
Assume that $(X,\|\cdot\|)$ and $m \in \{1,\ldots,\dim X -1 \}$ have the following property: Every $W \in \bG_m(X)$ admits a density contractor with respect to $\calH^m_{\|\cdot\|}$. It follows that for every complete normed Abelian group $(G,\lno\cdot\rno)$ the triple $\big((X,\|\cdot\|),(G,\lno\cdot\rno),m)$ satisfies the triangle inequality for cycles.
\end{Theorem}

\begin{proof}
Let $P = \sum_{k=1}^\kappa g_k \lseg \sigma_k \rseg \in \calP_m(X,G)$, with the $\sigma_1,\ldots,\sigma_\kappa$ nonoverlapping. Since the statement to check is invariant under translation of $P$ we may assume 0 belongs to the support of $g_1\lseg \sigma_1\rseg$. Let $\pi \in \rmHom(X,W)$ where $W \in \bG_m(X)$ is the $m$ dimensional subspace of $X$ containing $\sigma_1$. Now $\pi_\#P \in \calP_m(W,G)$, $\partial \pi_\# P = 0$ and $\pi_\#P$ has compact support. Applying the Constancy Theorem \cite[Theorem 6.2]{DEP.HAR.14} with a large $m$ cube $Q \subset W$ such that $\rmspt(\pi_\#P) \subset \rmint Q$ we find that $\pi_\#P=0$. Therefore
\begin{equation*}
g_1 \lseg \pi(\sigma_1) \rseg = \pi_\#(g_1 \lseg \sigma_1 \rseg ) = - \pi_\# \sum_{k=2}^\kappa g_k \lseg \sigma_k \rseg = - \sum_{k=2}^\kappa g_k \lseg \pi(\sigma_k) \rseg \,.
\end{equation*}
The triangular inequality for $\calM_H$ thus implies 
\begin{multline*}
\lno g_1 \rno \calH^m_{\|\cdot\|}(\pi(\sigma_1)) = \calM_H(g_1 \lseg \pi(\sigma_1) \rseg) \leq \sum_{k=2}^\kappa \calM_H( g_k \lseg \pi(\sigma_k) \rseg)\\ = \sum_{k=2}^\kappa \lno g_k \rno \calH^m_{\|\cdot\|} (\pi(\sigma_k)) \,.
\end{multline*}
Now let $\mu$ be a density contractor for $W$. Integrating the above inequality with respect to $\mu$ yields the sought for inequality:
\begin{multline*}
\calM_H(g_1 \lseg \sigma_1 \rseg) = \lno g_1 \rno \calH^m_{\|\cdot\|}(\sigma_1) = \lno g_1 \rno \int_{\rmHom_m(X,X)} \calH^m_{\|\cdot\|}(\pi(\sigma_1)) d\mu(\pi) \\ \leq \sum_{k=2}^\kappa \lno g_k \rno \int_{\rmHom_m(X,X)} \calH^m_{\|\cdot\|} (\pi(\sigma_k)) d\mu(\pi) \leq \sum_{k=2}^\kappa\lno g_k \rno \calH^m_{\|\cdot\|} (\sigma_k) \\
= \sum_{k=2}^\kappa \calM_h( g_k \lseg \sigma_k \rseg) \,.
\end{multline*}
\end{proof}

\begin{Remark}
\label{rem-dc}
The following are two trivial cases of existence of density contractors. We recall that if $\pi : X \to W$ is Lipschitzian then $\calH^m_{\|\cdot\|}(\pi(A)) \leq (\rmLip \pi)^m \calH^m_{\|\cdot\|}(A)$, for every $A \subset X$, where the Lipschitz constant $\rmLip \pi$ is with respect to the norm $\|\cdot\|$ of $X$ and $W$. Thus in case $\pi$ is a projector onto $W$ (i.e. $\pi |_W = \rmid_W$) and $\rmLip \pi = 1$ then $\bdelta_\pi$ is readily a density contractor on $W$.
\begin{enumerate}
\item[(1)] If the norm $\|\cdot\|=|\cdot|$ is Euclidean and $m$ is arbitrary then the orthogonal projector $\pi : X \to W$ verifies the above condition.
\item[(2)] If the norm $\|\cdot\|$ is arbitrary and $m=1$ then there exists a projector $\pi : X \to W$ with $\rmLip\pi = 1$. Indeed letting $w$ be a unit vector spanning $W$, we choose $\alpha \in X^*$ such that $\rmLip \alpha =1$ and $\alpha(w)=1$, according to Hahn's theorem, and we define $\pi(x) = \alpha(x)w$.
\end{enumerate}
There does not always exist a projector $\pi : X \to W$ with $\rmLip = 1$, even when $m+1 = \dim X = 3$. For instance when $X = \ell_\infty^3$ and $W = X \cap \{(x_1,x_2,x_3) : x_1+x_2+x_3=0\}$ any projector $\pi : X \to W$ has $\rmLip \pi \geq 1 + 1/7$. This does not rule out the possibility that there be a projector onto $W$ that decreases the area $\calH^2_{\|\cdot\|_{\infty}}$ ; such projector actually exists according to the following classical result of H. Busemann.
\end{Remark}

\begin{Theorem}[Busemann, 1949]
\label{44}
Let $m = \dim X - 1$. For every $W \in \bG_m(X)$ there exists a projector $\pi$ onto $W$ with the following property. For every $V \in \bG_m(X)$ and every Borel set $A \subset V$ one has $\calH_{\|\cdot\|}^m(\pi(A)) \leq \calH_{\|\cdot\|}^m(A)$. 
\end{Theorem}

In fact H. Busemann shows that the function
\begin{equation*}
X \to \R : u \mapsto \frac{|u| \balpha(m)}{\calH^m_{| \cdot |}\left(B_{\|\cdot\|} \cap \rmspan\{u\}^\perp\right)}
\end{equation*}
is convex (hence a norm on $X$), see \cite{BUS.49} or \cite[\S 7.1]{THOMPSON}. The existence of the projector the follows for instance as in \cite[Theorem 4.13]{ALV.THO.04}.

\begin{Theorem}[Burago-Ivanov, 2012]
\label{45}
Let $m=2$ and assume the norm $\|\cdot\|$ is crystalline\footnote{i.e.\ its unit ball $B_{\|\cdot\|}$ is a polytope}. Given $W \in \bG_2(X)$ we let $u_1,\ldots, u_{2p}$ denote a collection of distinct unit vectors in $W$, numbered in consecutive order, such that $-u_i=u_{p+i-1}$ for each $i=1,\ldots,p$, and containing all  the vertices of the polygon $W \cap B_{\|\cdot\|}$. Let $\alpha_i \in X^*$ be a supporting functional of $B_{\|\cdot\|}$ such that $\alpha_i|_{\rmconv\{u_i,u_{i+1}\}} = 1$ and $\lambda_i = 2 \calH^2_{|\cdot|}(W \cap B_{\|\cdot\|})^{-1} \calH^2_{|\cdot|}(\rmconv\{0,u_i,u_{i+1}\})$, $i=1,\ldots,p$. Define $\omega \in \bigwedge^2X$ by the formula
\begin{equation*}
\omega = \balpha(2) \sum_{1 \leq i < j \leq p} \lambda_i \lambda_j . \alpha_i \wedge \alpha_j \,.
\end{equation*}
It follows that for every $V \in \bG_2(X)$ one has
\begin{equation*}
| \omega(v_1 \wedge v_2) | \leq \balpha(2) \sum_{1 \leq i < j \leq p} \lambda_i \lambda_j  \left| \la \alpha_i \wedge \alpha_j , v_1 \wedge v_2 \ra \right| \leq \psi(V)
\end{equation*}
where $v_1,v_2$ is an orthonormal basis of $V$, with both inequalities becoming equalities when $V=W$.
\end{Theorem}

This is the main result of \cite{BUR.IVA.12} (see the discussion at the beginning of $\S2$ and Proposition 2.2 therein). We should point out that our formulation differs from that in \cite{BUR.IVA.12} in two respects: First, we stress the middle inequality in the conclusion; Second, there may be more unit vectors $u_1,\ldots,u_{2p}$ in our statement than there are vertices of $W \cap B_{\|\cdot\|}$ but same argument as in \cite{BUR.IVA.12} applies in this slightly more general situation (which is needed in the proof of \ref{48}).   We now proceed to showing how it leads to the existence of density contractors in case $m=2$. We start with two easy and useful observations.

\begin{Proposition}
\label{46}
Let $\mu$ be a Borel probability measure on $\rmHom_m(X,X)$, $W \in \bG_m(X)$, and assume $\mu$ is supported in $\rmHom(X,W)$. The following are equivalent.
\begin{enumerate}
\item[(1)] $\mu$ is a density contractor on $W$;
\item[(2)] For every $V \in \bG_m(X)$ there exists some Borel subset $A \subset V$ with $0 < \calH^m_{\|\cdot\|}(A) < \infty$ and
\begin{equation*}
\int_{\rmHom_m(X,X)} \calH^m_{\|\cdot\|}(\pi(A)) d\mu(\pi) \leq  \calH^m_{\|\cdot\|}(A)\,,
\end{equation*}
with equality when $V=W$;
\item[(3)] For every $V \in \bG_m(X)$ there exists some Borel subset $A \subset V$ with $0 < \calH^m_{|\cdot|}(A) < \infty$ and
\begin{equation*}
\int_{\rmHom_m(X,X)} \calH^m_{|\cdot|}(\pi(A)) d\mu(\pi) \leq \frac{\psi(V)}{\psi(W)} \calH^m_{|\cdot|}(A)\,,
\end{equation*}
with equality when $V=W$.
\end{enumerate}
\end{Proposition}

\begin{proof}
Recalling that $\calH^m_{\|\cdot\|}(E) = \psi(V) \calH^m_{|\cdot|}(E)$ whenever $E \subset V \in \bG_m(X)$ and $E$ is Borel, we infer at once that $\mu$ is a density contractor on $W$ if and only if the condition in (3) holds for every Borel $A \subset V$. Thus clearly (3) is a consequence of (1), and is equivalent to (2). In order to establish that (2) implies (1) we fix $V \in \bG_m(X)$ and we define
\begin{equation*}
\phi(A) = \int_{\rmHom(X,W)} \calH^m_{\|\cdot\|}(\pi(A)) d\mu(\pi) \,,
\end{equation*}
$A \in \calB(V)$. We must show that if $\phi(A) \leq  \calH^m_{\|\cdot\|}(A)$ (resp. $\phi(A) = \calH^m_{\|\cdot\|}(A)$ when $V=W$) for some $A \in \calB(V)$ such that $0 < \calH^m_{|\cdot|}(A) < \infty$ then it holds for all $A \in \calB(V)$. We define
\begin{equation*}
G_{W,V} = \rmHom(X,W) \cap \{ \pi : \pi|V \text{ is injective} \} \,.
\end{equation*}
Clearly $G_{W,V}$ is an open subset of $\rmHom(X,W)$ and $\calH^m_{\|\cdot\|}(\pi(A))=0$ whenever $\pi \not\in G_{W,V}$. When $\pi \in G_{W,V}$ we abbreviate $\phi_\pi(A) = \calH^m_{\|\cdot\|}(\pi(A))$. Since $\pi|V$ is a homeomorphism from $V$ to $W$ it follows that $\phi_\pi$ is a measure on $\calB(V)$. Since $\phi(A) = \int_{G_{W,V}} \phi_\pi(A)d\mu(\pi)$, it ensues from the monotone convergence theorem that $\phi$ is a measure as well on $\calB(V)$. If $h \in V$ and $\pi \in \rmHom(X,W)$ then $\phi_\pi(A+h)=\phi_\pi(A)$ because $\pi$ is linear and $\calH^m_{\|\cdot\|}$ is translation invariant. Therefore $\phi$ is also translation invariant. Now either $\phi=0$ and there is nothing to prove or $\phi$ is one of Haar measures on $V$ and the conclusion follows from their uniqueness up to a multiplicative factor, because (the restriction to $\calB(V)$ of) $\calH^m_{\|\cdot\|}$ is also a Haar measure on $V$.
\end{proof}

In our next observation we consider density contractors with respect to a sequence of norms on $X$. Rather than using the ambiguous notation $\|\cdot\|_j$, $j=1,2,\ldots$ for a sequence of norms, we prefer using $\nu_j$, $j=1,2,\ldots$.

\begin{Proposition}
\label{47}
Let $\nu,\nu_1,\nu_2,\ldots$ be a sequence of norms on $X$, let $W \in \bG_m(X)$ and let $\mu_1,\mu_2,\ldots$, be a sequence of probability measures on $\rmHom_m(X,X)$ all supported in $\rmHom(X,W)$. We assume that
\begin{enumerate}
\item[(1)] Each $\mu_k$ is a density contractor on $W$ with respect to $\calH^m_{\nu_k}$, $k=1,2,\ldots$;
\item[(2)] $\nu_k \to \nu$ as $k \to \infty$;
\item[(3)] The sequence $\mu_1,\mu_2,\ldots$ is uniformly tight, i.e.
\begin{equation*}
\lim_{n \to \infty} \sup_{k =1,2,\ldots} \mu_k \left( \rmHom(X,W) \cap \{ \pi : \vvvert \pi \vvvert \geq n \} \right) = 0 \,;
\end{equation*}
\item[(4)] There exists a compact convex $A \subset W$ such that $0 < \calH^m_{|\cdot|}(A)$ and
\begin{equation*}
\lim_{n \to \infty} \sup_{k =1,2,\ldots} \int_{ \rmHom(X,W) \cap \{ \pi : \vvvert \pi \vvvert \geq n \} } \calH^m_{\nu_k}(\pi(A))d\mu_k(\pi) = 0 \,.
\end{equation*}
\end{enumerate}
It follows that there exists a density contractor on $W$ with respect to $\calH^m_\nu$.
\end{Proposition}

\begin{proof}
We first notice that $\mu_1,\mu_2,\ldots$ admits a subsequence (still denoted the same way) converging tightly to some Borel probability measure $\mu$ on $\rmHom_m(X,X)$, also supported in $\rmHom(X,W)$, according to assumption (3) and Prokhorov's Theorem \cite[Chapter II Theorem 6.7]{PARTHASARATHY}. In other words
\begin{equation}
\label{eq.thu.3}
\int_{\rmHom_m(X,X)} f(\pi) d\mu_k(\pi) \to \int_{\rmHom_m(X,X)} f(\pi) d\mu(\pi) \text{ as } k \to \infty
\end{equation}
whenever $f : \rmHom_m(X,X) \to \R$ is continuous and bounded. Assumption (1) says that
\begin{equation}
\label{eq.thu.1}
\int_{\rmHom_m(X,X)} \calH^m_{\nu_k}(\pi(A)) d\mu_k(\pi) \leq \calH^m_{\nu_k}(A)
\end{equation}
for every $k=1,2,\ldots$, every $V \in \bG_m(X)$ and every Borel $A \subset V$, with equality when $V=W$. Assumption (2) is that $\delta(\nu,\nu_k) \to 1$ and $k \to \infty$. According to \ref{46} it suffices to establish that \eqref{eq.thu.1} holds with $\nu_k$ replaced by $\nu$, $\mu_k$ replaced by $\mu$, for some Borel $A \subset V$ such that $0 < \calH^m_\nu(A) < \infty$, with equality when $V=W$. We will start by proving the inequality case.
\par 
Fix $V \in \bG_m(X)$ and choose a compact convex set $A \subset X$ such that $0 < \calH^m_{|\cdot|}(A)$, for instance $A = V \cap B_{|\cdot|}$. We define on $\rmHom_m(X,X)$ the real valued functions $f,f_1,f_2,\ldots$ by the formula $f(\pi) = \calH^m_\nu(\pi(A))$ and $f_k(\pi) = \calH^m_{\nu_k}(\pi(A))$, $k=1,2,\ldots$. Letting $\calK_{\mathrm{conv}}(W)$ denote the space of nonempty compact convex subsets of $W$ endowed with its Hausdorff metric, we notice that $\rmHom(X,W) \to \calK_{\mathrm{conv}}(W) : \pi \mapsto \pi(A)$ is continuous since $\rmdist_\calH(\pi(A) , \tilde{\pi}(A)) \leq \vvvert \pi - \tilde{\pi} \vvvert \rmdiam ( A \cup \{0\})$. If $\calH$ is a Haar measure on $W$ then its restriction $\calK_{\mathrm{conv}}(W) \to \R : C \mapsto \calH(C)$ is continuous, see for instance \cite[3.2.36]{GMT}. It therefore follows that the $f,f_1,f_2,\ldots$ are all continuous. As 
\begin{equation}
\label{eq.thu.2}
f_k(\pi) = \calH^m_{\nu_k}(\pi(A)) \leq (\rmLip_{\nu_k} \pi)^m \calH^m_{\nu_k}(A)
\end{equation}
we note that they do not need to be bounded. Letting $\Gamma = \sup_{k=1,2,\ldots} \delta(\nu_k,|\cdot|)$ we note that $\Gamma < \infty$ and that $\rmLip_{\nu_k}(\pi) \leq \Gamma^2 \vvvert \pi \vvvert$, $k=1,2,\ldots$, as well as $\rmLip_{\nu}(\pi) \leq \Gamma^2 \vvvert \pi \vvvert$ whenever $\pi \in \rmHom_m(X,X)$. Given $n=1,2,\ldots$ we choose a continuous (cut-off function) $\chi_n : \rmHom_m(X,X) \to \R$ such that
\begin{equation}
\label{eq.cutoff}
\ind_{\rmHom_m(X,X) \cap \{ \pi : \vvvert \pi \vvvert \leq n \}} \leq \chi_n \leq
\ind_{\rmHom_m(X,X) \cap \{ \pi : \vvvert \pi \vvvert \leq n+1 \}} \,.
\end{equation}
Since are $\chi_n f_k$ compactly supported, they are bounded. In fact 
\begin{equation*}
\|\chi_n f_k \|_\infty \leq \left( \Gamma^2 (n+1) \right)^m \Gamma^m \calH^m_{|\cdot|}(A)
\end{equation*}
according to \eqref{eq.thu.2}, \eqref{eq.fri.1} and \ref{b.m}(2). We now show that $\chi_n f_1, \chi_n f_2,\ldots$ converge uniformly to $\chi_n f$. Observe that
\begin{multline*}
| f_k(\pi) - f(\pi) | = \left( \beta(\nu_k,\nu,W) - 1 \right) \calH^m_\nu(\pi(A)) \\ \leq \left( \beta(\nu_k,\nu,W) - 1 \right) \Gamma^{2m} \vvvert \pi \vvvert^m \calH^m_\nu(A) \,,
\end{multline*}
for every $\pi \in \rmHom_m(X,X)$, whence
\begin{equation*}
\| \chi_n f_k - \chi_n f \|_\infty \leq \left( \beta(\nu_k,\nu,W) - 1 \right)\Gamma^{2m} (n+1)^m \calH^m_\nu(A) \,.
\end{equation*}
Now $\delta(\nu_k|_W,\nu|_W)^{-m} \leq \beta(\nu_k,\nu,W) \leq \delta(\nu_k|_W,\nu|_W)^m$ according to \eqref{eq.fri.1} and $\delta(\nu_k|_W,\nu|_W) \to 1$ as $k \to \infty$ according \ref{b.m}(2) and assumption (2), thus $\| \chi_n f_k - \chi_n f \|_\infty  \to 0$ as $k \to \infty$. This together with \eqref{eq.thu.3} yields classically that
\begin{equation}
\label{eq.thu.4}
\int_{\rmHom_m(X,X)} \chi_n f_k d\mu_k \to \int_{\rmHom_m(X,X)} \chi_n f d\mu \text{ as } k \to \infty \,,
\end{equation}
simply because
\begin{multline*}
\left| \int_{\rmHom_m(X,X)} \chi_n f_k d\mu_k - \int_{\rmHom_m(X,X)} \chi_n f d\mu \right| \\ \leq \| \chi_n f_k - \chi_n f \|_\infty  + \left| \int_{\rmHom_m(X,X)} \chi_n f d\mu_k - \int_{\rmHom_m(X,X)} \chi_n f d\mu \right| \,.
\end{multline*}
Now \eqref{eq.thu.4} and \eqref{eq.thu.1} imply that
\begin{equation*}
\begin{split}
\int_{\rmHom_m(X,X)} \chi_n(\pi) \calH^m_\nu(\pi(A)) d\mu(\pi) & = \lim_k \int_{\rmHom_m(X,X)} \chi_n(\pi) \calH^m_{\nu_k}(\pi(A)) d\mu_k(\pi) \\
&\leq \liminf_k \int_{\rmHom_m(X,X)} \calH^m_{\nu_k}(\pi(A)) d\mu_k(\pi) \\
&\leq \liminf_k \calH^m_{\nu_k}(A) \\
&= \left( \liminf_k \beta(\nu_k,\nu,V) \right) \calH^m_\nu(A) \\
& = \calH^m_\nu(A) \,. 
\end{split}
\end{equation*}
Letting $n \to \infty$ and referring to the monotone convergence theorem we obtain
\begin{equation}
\int_{\rmHom_m(X,X)}  \calH^m_\nu(\pi(A)) d\mu(\pi) \leq \calH^m_\nu(A) \,. 
\end{equation}

The inequality case in \ref{46}(2) is now established and it remains only to show that the above becomes an inequality when $V=W$. This is where assumption (4) turns up. We keep the same notations as above but we reason in the particular case when $V=W$ and $A$ is the set given in assumption (4). Our extra information is that for each $n=1,2,\ldots$ there exists $\veps_n >0$ with
\begin{equation*}
\sup_{k=1,2,\ldots} \int_{\rmHom_m(X,X)} \left( 1 - \chi_n \right) f_k d\mu_k \leq \veps_n
\end{equation*}
and $\veps_n \to 0$ as $n \to \infty$. Thus for every $n,k=1,2,\ldots$,
\begin{multline}
\label{eq.thu.5}
\left| \int_{\rmHom_m(X,X)} f_k d\mu_k - \int_{\rmHom_m(X,X)} \chi_n f d\mu \right| \\
\leq \int_{\rmHom_m(X,X)} \left( 1 - \chi_n \right) f_k d\mu_k + \left| \int_{\rmHom_m(X,X)} \chi_n f_k d\mu_k - \int_{\rmHom_m(X,X)} \chi_n f d\mu \right| \\
\leq \veps_n + \left| \int_{\rmHom_m(X,X)} \chi_n f_k d\mu_k - \int_{\rmHom_m(X,X)} \chi_n f d\mu \right| \,.
\end{multline}
Recalling our assumption that $\mu_k$ is a density contractor on $W$ we infer that
\begin{multline*}
\lim_k \int_{\rmHom_m(X,X)} f_k d\mu_k = \lim_k \int_{\rmHom_m(X,X)} \calH^m_{\nu_k}(\pi(A))d\mu_k(\pi) \\= \lim_k \calH^m_{\nu_k}(A) = \calH^m_\nu(A) \,.
\end{multline*}
Using this together with \eqref{eq.thu.4} and letting $k \to \infty$ in \eqref{eq.thu.5} we now find that
\begin{equation*}
\left| \calH^m_\nu(A) - \int_{\rmHom_m(X,X)} \chi_n(\pi) \calH^m_\nu(\pi(A))d\mu(\pi) \right| \leq \veps_n \,.
\end{equation*}
Letting $n \to \infty$ our conclusion becomes a consequence of the monotone convergence theorem.
\end{proof}

\begin{Theorem}
\label{48}
Let $(X,\|\cdot\|)$ be a finite dimensional normed space. Every $W \in \bG_2(X)$ admits a density contractor with respect to $\calH^2_{\|\cdot\|}$.
\end{Theorem}

We start with the case when $\|\cdot\|$ is crystalline. For use in the proof we introduce the notation $\rmsquare(v_1,v_2) = X \cap \{ t_1v_1 + t_2v_2 : 0 \leq t_i \leq 1, i=1,2 \}$ to refer to the square with two edges $v_1,v_2$, when these constitute an orthonormal family in $X$.

\begin{proof}[Proof of \ref{48} in case $\|\cdot\|$ is crystalline]
In this case we shall prove that the density contractor $\mu$ can be chosen to verify the following two additional requirements. Here $\Gamma = \delta(\|\cdot\|,|\cdot|)$ and 
\begin{equation*}
\tau = \frac{\max_{i=1,\ldots,p}|u_i \wedge u_{i+1}|_2}{\min_{i=1,\ldots,p}|u_i \wedge u_{i+1}|_2} \,.
\end{equation*}
{\it 
\begin{enumerate}
\item[(1)] For every $n=1,2,\ldots$ one has
\begin{equation*}
\mu \left( \rmHom(X,W) \cap \{ \pi : \vvvert \pi \vvvert \geq n \} \right) \leq \frac{4\Gamma^4(2+\tau)}{n^2} \,;
\end{equation*}
\item[(2)] For every $n=1,2,\ldots$ and every orthonormal basis $v_1,v_2$ of $W$ one has
\begin{equation*}
\int_{\rmHom(X,W) \cap \{ \pi : \vvvert \pi \vvvert \geq n \}} \calH^2_{\|\cdot\|}(\pi(\rmsquare(v_1,v_2)))d\mu(\pi)  \leq \frac{\balpha(2)\Gamma^5(2+\tau)}{n^2} \,.
\end{equation*}
\end{enumerate}
}
We let $u_1,\ldots,u_{2p}$, $\lambda_1,\ldots,\lambda_p$ and $\alpha_1,\ldots,\alpha_p$ be defined as in \ref{45}. For use in the definition of $\mu$ we define certain $\tilde{\pi}_{i,j} \in \rmHom(X,W)$, $i,j=1,\ldots,p$ with $i \neq j$, as follows. 
\begin{equation*}
\tilde{\pi}_{i,j}(x) = \alpha_i(x)u_i + \alpha_j(x)u_j \,,
\end{equation*}
$x \in X$. We notice that
\begin{equation}
\label{eq.thu.6}
\vvvert \tilde{\pi}_{i,j} \vvvert \leq 2 \Gamma^2 \,.
\end{equation}
Furthermore if $v_1,v_2$ is an orthonormal family in $X$ then $\tilde{\pi}_{i,j}(\rmsquare(v_1,v_2))$ is a parallelogram in $W$ with sides $\tilde{\pi}_{i,j}(v_1)$ and $\tilde{\pi}_{i,j}(v_2)$, thus
\begin{multline}
\label{eq.thu.7}
\calH^2_{|\cdot|}\left( \tilde{\pi}_{i,j}(\rmsquare(v_1,v_2)) \right) = \left| \tilde{\pi}_{i,j}(v_1) \wedge \tilde{\pi}_{i,j}(v_2) \right|_2 \\ = \left| \det \left( \begin{matrix}  \alpha_i(v_1) & \alpha_j(v_1) \\ \alpha_i(v_2) & \alpha_j(v_2) \end{matrix}\right) \right|. \left| u_i \wedge u_j \right|_2 \,.
\end{multline}
We now normalize the $\tilde{\pi}_{i,j}$. We define
\begin{equation*}
\rho = \sqrt{\frac{\balpha(2)}{2\psi(W)}} \,,
\end{equation*}
and
\begin{equation*}
\pi_{i,j} = \left( \frac{\rho}{\sqrt{|u_i \wedge u_j|_2}}\right) \tilde{\pi}_{i,j}
\end{equation*}
in case $i \neq j$, and $\pi_{i,i} = 0$, $i=1,\ldots,p$. It clearly follows from \eqref{eq.thu.6} and \eqref{eq.thu.7} that
\begin{equation}
\label{eq.thu.8}
\vvvert \pi_{i,j} \vvvert \leq \frac{2 \Gamma^2\rho}{\sqrt{|u_i \wedge u_j|_2}} \,,
\end{equation}
and
\begin{equation}
\label{eq.thu.9}
\calH^2_{|\cdot|}\left( \pi_{i,j}(\rmsquare(v_1,v_2)) \right) = \rho^2\left| \det \left( \begin{matrix}  \alpha_i(v_1) & \alpha_j(v_1) \\ \alpha_i(v_2) & \alpha_j(v_2) \end{matrix}\right) \right| \,.
\end{equation}
We are now ready to define the Borel measure $\mu$ on $\rmHom(X,W)$:
\begin{equation}
\mu = \sum_{i,j=1}^p \lambda_i \lambda_j \bdelta_{\pi_{i,j}} \,.
\end{equation}
From the definition of the $\lambda_i$ we infer that $\sum_{i=1}^p \lambda_i = 1$ and therefore $\mu$ is readily a probability measure.

In order to show that $\mu$ is a density contractor on $W$ with respect to $\calH^2_{\|\cdot\|}$ we will apply \ref{46}(3). Given $V \in \bG_2(X)$ we choose $v_1,v_2$ an orthonormal basis of $V$ and we let $A = \rmsquare(v_1,v_2)$. We observe that
\begin{equation*}
\begin{split}
\int_{\rmHom(X,W)} \calH^2_{|\cdot|}(\pi(A))d\mu(\pi) & = \sum_{i,j=1}^p \lambda_i \lambda_j \calH^2_{|\cdot|}(\pi_{i,j}(A)) \\
& = 2 \sum_{1 \leq i < j \leq p} \lambda_i \lambda_j \rho^2 \left| \det \left( \begin{matrix}  \alpha_i(v_1) & \alpha_j(v_1) \\ \alpha_i(v_2) & \alpha_j(v_2) \end{matrix}\right) \right| \\
& = 2 \rho^2 \sum_{1 \leq i < j \leq p} \lambda_i \lambda_j \left| \la \alpha_i \wedge \alpha_j , v_1 \wedge v_2 \ra \right| \\
& \leq 2 \rho^2 \balpha(2)^{-1} \psi(V) \\
& = \frac{\psi(V)}{\psi(W)} \\
& = \frac{\psi(V)}{\psi(W)} \calH^2_{|\cdot|}(A) \\
\end{split}
\end{equation*}
where the inequality follows from \ref{45} and specializes to an equality in case $V=W$. This completes the proof that $\mu$ is a density contractor on $W$.
\par 
We now turn to establishing the extra properties (1) and (2) stated at the beginning of the proof. Given $n=1,2,\ldots$ we let 
\begin{equation*}
\Pi_n = \rmHom(X,W) \cap \{ \pi : \vvvert \pi \vvvert \geq n \} \,.
\end{equation*}
If $i,j = 1,\ldots,p$ and $\pi_{i,j} \in \Pi_n$ then
\begin{equation*}
| u_i \wedge u_j |_2 \leq \frac{4\Gamma^4 \rho^2}{n^2}
\end{equation*}
according to \eqref{eq.thu.8}. With $i=1,\ldots,p$ we associate $J^+_{n,i} = \{1,\ldots,p\} \cap \{ j : i < j \text{ and } \pi_{i,j} \in \Pi_n \}$. Assuming that $J^+_{n,i} \neq \emptyset$ we define $j^*_i = \max J^+_{n,i}$ and we infer from elementary geometric reasoning that
\begin{multline*}
\bigcup_{j \in J^+_{n,i}} \rmconv(0,u_{j-1},u_j) \subset W \cap B_{\|\cdot\|} \cap \{ s u_i + t u_{j_i^*} : 0 \leq s \text{ and } 0 \leq  t \} \\ \subset X \cap  \{ s u_i + t u_{j_i^*} : 0 \leq s \leq  1 \text{ and } 0 \leq t \leq 1 \}
\end{multline*}
and therefore
\begin{equation*}
\begin{split}
\sum_{j \in J^+_{n,i}} \lambda_j & = \frac{2}{\calH^2_{|\cdot|}\left(B_{\|\cdot\|}\cap W\right)}\sum_{j \in J^+_{n,i}} \calH^2_{|\cdot|}(\rmconv(0,u_j,u_{j+1})) \\
& \leq \frac{1}{\calH^2_{|\cdot|}\left(B_{\|\cdot\|}\cap W\right)}  \left( \left|u_i \wedge u_{j_i^*}\right|_2 +\left|u_{j_i^*} \wedge u_{j_i^*+1}\right|_2\right) \\
& \leq \left( \frac{4\Gamma^4\rho^2(1+\tau)}{\calH^2_{|\cdot|}\left(B_{\|\cdot\|}\cap W\right)} \right) \frac{1}{n^2} \\
& = \frac{4 \Gamma^4(1+\tau)}{n^2} \,.
\end{split}
\end{equation*}
Similarly we let $J^-_{n,i} = \{1,\ldots,p\} \cap \{ j : j < i \text{ and } \pi_{i,j} \in \Pi_n \}$. Assuming that $J^-_{n,i} \neq \emptyset$ we define $j^i_* = \min J^+_{n,i}$ and reasoning analogously we obtain the slightly better
\begin{equation*}
\sum_{j \in J^-_{n,i}} \lambda_j \leq \frac{1}{\calH^2_{|\cdot|}\left(B_{\|\cdot\|}\cap W\right)} \left| u_{j^i_*} \wedge u_i \right|_2 \leq  \frac{4 \Gamma^4}{n^2} \,.
\end{equation*}
Consequently,
\begin{equation}
\label{eq.thu.10}
\mu \left( \Pi_n \right) = \sum_{\substack{i,j=1 \\ \pi_{i,j} \in \Pi_n}}^p \lambda_i \lambda_j = \sum_{i=1}^p \lambda_i \sum_{j \in J^-_{n,i} \cup J^+_{n,i}} \lambda_j \leq \frac{4\Gamma^4(2+\tau)}{n^2}
\end{equation}
and the proof of (1) is complete.
\par 
Finally we observe that for all $i=1,\ldots,p$ and $l=1,2$ one has $|\alpha_i(v_l)| \leq \|\alpha_i\|^*\|v_l\| \leq \Gamma$ and thus $\calH_{|\cdot|}^2(\pi_{i,j}(\rmsquare(v_1,v_2))) \leq 2 \Gamma \rho^2$, according to \eqref{eq.thu.9}. It therefore follows from \eqref{eq.thu.10} that
\begin{multline*}
\int_{\Pi_n} \calH_{\|\cdot\|}^2(\pi(\rmsquare(v_1,v_2))) d\mu(\pi) = \psi(W) \sum_{\substack{i,j=1 \\ \pi_{i,j} \in \Pi_n}}^p \lambda_i \lambda_j \calH_{|\cdot|}^2(\pi_{i,j}(\rmsquare(v_1,v_2))) \\
\leq \psi(W) 2 \Gamma \rho^2 \mu(\Pi_\lambda) = \frac{\balpha(2)\Gamma^5(2+\tau)}{n^2}
\end{multline*}
and the proof of (2) is complete.
\end{proof}

\begin{proof}[Proof of \ref{48} in the general case]
Fix $W \in \bG_2(X)$.
We start by choosing a sequence of crystalline norms $\nu_1,\nu_2,\ldots$ such that $\nu_k \to \|\cdot\|$ as $k \to \infty$. This is done classically by choosing a finite $k^{-1}$-net $F_k \subset X \cap \{ x : \|x\| = 1\}$ and letting $B_{\nu_k} = \rmconv ( F_k \cup (-F_k))$. Now with each $\nu_k$ we will associate a density contractor $\mu_k$ on $W$ with respect to $\calH^2_{\nu_k}$ as in the first part of the proof, choosing the unit vectors $u^k_1,\ldots,u^k_{2p_k}$ (recall the statement of \ref{45}) in order that all the $|u^k_i \wedge u^k_{i+1}|_2$ are nearly the same value. This is where we may have to add unit vectors to the list of vertices $v^k_1,\ldots,v^k_{2q_k}$ of the polygon $W \cap B_{\nu_k}$. It can easily be done for the following reason: The formula $d_k(u,v) = |u \wedge v |_2$ defines a distance of unit vectors $u$, $v$ lying in << between >> $v_1^k$ and $v^k_{q_k}$ such that each << segment >> $\blseg u , v \brseg$ on the corresponding << half unit circle >> can be partitioned into two segments $\blseg u , w \brseg$ and $\blseg w , v \brseg$ of same << length >>; iterating this process we can readily achieve $\tau_k \leq 2$.

Now since $\nu_k \to \|\cdot\|$ as $k \to \infty$ we infer that $\sup_{k=1,2,\ldots} \Gamma_k < \infty$. It therefore follows from the estimates (1) and (2) proved about $\mu_k$ in the first part of this proof that the sequence $\mu_1,\mu_2,\ldots$ verifies the hypotheses of \ref{47}. The conclusion follows.
\end{proof}

\begin{Theorem}
\label{existence.d.c}
Let $(X,\|\cdot\|)$ be a finite dimensional normed space and $m \in \{1,\ldots,\dim X -1 \}$. It follows that every $W \in \bG(X)$ admits a density contractor with respect to $\calH^m_{\|\cdot\|}$ when either $m=1$ or $m=2$ or $m=\dim X -1$.
\end{Theorem}

\begin{proof}
The case $m=1$ follows from Hahn's theorem as in \ref{rem-dc}(2) and the case $m=\dim X -1$ follows from Busemann's theorem \ref{44}. In both cases a density contractor on $W$ is given by $\mu = \bdelta_\pi$ where $\pi : X \to W$ is a projector that contracts $\calH^m_{\|\cdot\|}$. The case $m=2$ is \ref{48} above, a consequence of Burago-Ivanov's Theorem \ref{45} and the density contractor exhibited is not of the simple type $\bdelta_\pi$.
\end{proof}

\begin{Theorem}
\label{411}
Let $(X,\|\cdot\|)$ be a finite dimensional normed space, let $(G,\lno\cdot\rno)$ be a complete normed Abelian group and let $m \in \{1,\ldots,\dim X - 1 \}$. Assume that $W \in \bG_m(X)$, $\mu$ is a density contractor on $W$ and $T \in \calR_m(X,G)$. It follows that
\begin{equation*}
\int_{\rmHom(X,W) \times W} \calM(\la T , \pi , y \ra) d\left( \mu \otimes \calH^m_{\|\cdot\|} \right)(\pi,y) \leq \calM_H(T)
\end{equation*}
and the above integrand is $\mu \otimes \calH^m_{\|\cdot\|}$  measurable. Furthermore
\begin{equation*}
\int_{\rmHom(X,W)} \calM_H(\pi_\# T) d\mu(\pi) \leq \calM_H(T)
\end{equation*}
with equality when $\rmspt T \subset W$, and the above integrand is Borel measurable.
\end{Theorem}

In the above statement we wrote $\calM(\la T , \pi , y \ra)$ without a subscript $H$. Indeed no reference to the norm $\|\cdot\|$ is needed at all, as both $\calR_0(X,G)$ and the mass $\calM$ defined on it are independent of $\|\cdot\|$. Members of $\calR_0(X,G)$ are $G$ valued atomic Borel measures on $X$, i.e. of the form $\sum_{j \in J} g_j \bdelta_{x_j}$ with the $x_j$ all distinct, and
\begin{equation*}
\calM \left( \sum_{j \in J} g_j \bdelta_{x_j}\right) = \sum_{j \in J} \lno g_j \rno \,.
\end{equation*}

\begin{proof}
We give both $X$ and $W$ a fixed orientation. We start with the first inequality in the case when $T = P \in \calP_m(X,G)$ is polyhedral. We choose a decomposition $P = \sum_{k=1}^\kappa \lseg \sigma_k \rseg$ where the $\sigma_1,\ldots,\sigma_\kappa$ are nonoverlapping. Let $\pi \in \rmHom(X,W)$ be of maximal rank. For almost every $y \in W$ the following holds: For every $k=1,\ldots,\kappa$ either $\sigma_k \cap \pi^{-1}\{y\}$ is a singleton or it is empty. Furthermore, letting $F_y = \left( \cup_{k=1}^\kappa \sigma_k \right) \cap \pi^{-1}\{y\}$ one has
\begin{equation*}
\la P , \pi , y \ra = \sum_{x \in F_y} (-1)^{\veps(\pi,x)}\bdelta_x
\end{equation*}
where $\veps(\pi,x)=\pm 1$ according to whether $\pi|_{W_k} : W_k \to W$ preserves or not the orientation, where $W_k$ is the affine $m$ plane containing $\sigma_k$ and is given the orientation of $\lseg \sigma_k \rseg$. In particular, for those $y$,
\begin{equation*}
\calM( \la P , \pi , y \ra ) = \sum_{k=1}^\kappa \lno g_k \rno \ind_{\pi(\sigma_k)}(y) \,.
\end{equation*} 
The same formula holds if $\pi$ is not of maximal rank; in fact in that case both members of the identity above vanish for almost every $y$.
That this function of $(\pi,y)$ be Borel measurable can be established along the lines of \cite[\S 5.4]{BOU.DEP}. It now follows from Fubini's theorem and the property of density contractor that
\begin{multline*}
\int_{\rmHom(X,W) \times W} \calM(\la P , \pi , y \ra) d\left( \mu \otimes \calH^m_{\|\cdot\|} \right)(\pi,y) \\ = \sum_{k=1}^\kappa \lno g_k \rno\int_{\rmHom(X,W)}  \calH_{\|\cdot\|}^m(\pi(\sigma_k)) d\mu(W) \\
\leq \sum_{k=1}^\kappa \lno g_k \rno \calH_{\|\cdot\|}^m(\sigma_k) = \calM_H(P)  \,.
\end{multline*}
We next assume that $T \in \calR_m(X,G)$ and we choose a sequence $P_1,P_2,\ldots$ in $\calP_m(X,G)$ such that $\lim_k \calF(P_k-T) = 0$ and $\limsup_k \calM_H(P_k) \leq \calM_H(T)$, see \cite[Theorem 4.2]{DEP.13.approx}. Taking a subsequence if necessary we may assume that $P_1,P_2,\ldots$ converges rapidly to $T$. Accordingly it follows from \cite[5.2.1(4)]{DEP.HAR.07} that for every $\pi \in \rmHom(X,W)$ one has, for $\calH^m_{\|\cdot\|}$ almost every $y \in W$, $\calF(\la T,\pi,y \ra - \la P_k,\pi,y \ra) \to 0$ as $k \to \infty$. Since $\calM : \calR_0(X,G) \to \R$ is $\calF$ lower semicontinuous, \cite[4.4.1]{DEP.HAR.07}, we infer at once that $(\pi,y) \mapsto \calM(\la T,\pi,y \ra)$ is $\mu \otimes \calH^m_{\|\cdot\|}$ measurable, and that
\begin{equation*}
\calM(\la T,\pi,y \ra) \leq \liminf_k \calM \left( \la P_k , \pi , y \ra \right) \,.
\end{equation*}
It then follows from Fatou's lemma that
\begin{multline*}
\int_{\rmHom(X,W) \times W} \calM(\la T , \pi , y \ra) d\left( \mu \otimes \calH^m_{\|\cdot\|} \right)(\pi,y)  \\ \leq \liminf_k \int_{\rmHom(X,W) \times W} \calM(\la P_k , \pi , y \ra) d\left( \mu \otimes \calH^m_{\|\cdot\|} \right)(\pi,y)\\
\leq \liminf_k \calM_H(P_k) \leq \calM_H(T)\,.
\end{multline*}
\par 
We now turn to proving the second conclusion. 
As before we start with the case when $T = P \in \calP_m(X,G)$ is polyhedral. The function $\rmHom(X,W) \to \R : \pi \mapsto \calM_H(\pi_\#P)$ is continuous. It is indeed the composition of $\calR_m(W,G) \to \R : S \mapsto \calM_H(S)$ and $\rmHom(X,W) \to \calR_m(X,G) : \pi \mapsto \pi_\#P$, the former being continuous since in fact $\calM_H(S) = \psi(W) \calM_{H,|\cdot|}(S)=\psi(W)\calF(S)$ as $m = \dim W$. The latter is continuous as well according to the homotopy formula $\calF(\pi_\#P - \tilde{\pi}_\#P) \leq \max \{ 1+\vvvert \pi \vvvert,  1+\vvvert \tilde{\pi} \vvvert \}^{m+1} \rmdiam ( \{0\} \cup \rmspt P) \calN(P)$, $\pi, \tilde{\pi} \in \rmHom(X,W)$, see for instance \cite[\S 2.6]{DEP.13.approx}.
We next choose a decomposition $P = \sum_{k=1}^\kappa \lseg \sigma_k \rseg$ where the $\sigma_1,\ldots,\sigma_k$ are nonoverlapping, thus $\calM_H(P) = \sum_{k=1}^\kappa \lno g_k \rno \calH^m_{\|\cdot\|}(\sigma_k)$. Given $\pi \in \rmHom(X,W)$ we notice that $\pi_\# P = \sum_{k=1}^\kappa g_k \lseg \pi(\sigma_k) \rseg$, thus $\calM_H(\pi_\#P) \leq \sum_{k=1}^\kappa \lno g_k \rno \calH^m_{\|\cdot\|}(\pi(\sigma_k))$. Integrating over $\pi$ with respect to $\mu$ we obtain
\begin{multline*}
\int_{\rmHom(X,W)} \calM_H(\pi_\#P) d\mu(\pi) \leq \sum_{k=1}^\kappa \lno g_k \rno \int_{\rmHom(X,W)} \calH^m_{\|\cdot\|}(\pi(\sigma_k))d\mu(\pi) \\ \leq \sum_{k=1}^\kappa \lno g_k \rno \calH^m_{\|\cdot\|}(\sigma_k) = \calM_H(P) \,.
\end{multline*}
In case $\rmspt P \subset W$ we reason as follows. Given $\pi \in \rmHom(X,W)$ we note that $\calM_H(\pi_\# P) = \sum_{k=1}^\kappa \lno g_k \rno \calH^m_{\|\cdot\|}(\pi(\sigma_k))$ because either $\pi$ has rank less than or equal to $m-1$ and both sides are clearly zero, or $\pi$ has rank $m$ and then the simplexes $\pi(\sigma_1),\ldots,\pi(\sigma_\kappa)$ are nonoverlapping. Thus the first inequality above becomes an inequality. The second one as well because each $\sigma_k \subset W$.
We now that $T \in \calR_m(X,G)$ and as we did before, we choose a sequence $P_1,P_2,\ldots$ in $\calP_m(X,G)$ such that $\lim_k \calF(P_k-T) = 0$ and $\limsup_k \calM_H(P_k) \leq \calM_H(T)$, see \cite[Theorem 4.2]{DEP.13.approx}.
 Since $\calF(\pi_\# P_k - \pi_\#T) \leq \max \left\{ \vvvert \pi \vvvert^m,\vvvert \pi \vvvert^{m+1}\right\} \calF(P_k-T) \to 0$ as $k \to \infty$ we infer that $\calM_H(\pi_\#P_k-\pi_\#T) \to 0$ as $k \to \infty$ because $\calM_H = \psi(W) \calF$. In particular $\pi \mapsto \calM_H(\pi_\#T)$ is Borel measurable. It then follows from the dominated convergence theorem that
\begin{multline*}
\int_{\rmHom(X,W)} \calM_H(\pi_\#T) d\mu(\pi) = \lim_k \int_{\rmHom(X,W)} \calM_H(\pi_\#P_k) d\mu(\pi) \\ \leq \limsup_k \calM_H(P_k) \leq \calM_H(T) \,.
\end{multline*}
If $\rmspt T \subset W$ then applying \cite[Theorem 4.2]{DEP.13.approx} in $W$ instead of $X$, or applying \cite[Lemma 3.2]{DEP.HAR.14} we can guarantee that each $\rmspt P_k \subset W$. In that case also $\lim_k \calM_H(P_k - T) =0$ thus both inequalities above become equalities.
\end{proof}

\begin{Theorem}
\label{413}
Let $(X,\|\cdot\|)$ be a finite dimensional normed space and let $m \in \{1,\ldots,\dim X - 1 \}$. Assume that $W \in \bG_m(X)$, $\mu$ is a density contractor on $W$ and $A \subset X$ is Borel measurable and countably $(\calH^m_{\|\cdot\|},m)$ rectifiable. It follows that
\begin{equation*}
\int_{\rmHom_m(X,X)} \calH_{\|\cdot\|}^m(\pi(A)) d\mu(\pi) \leq \calH_{\|\cdot\|}^m(A)
\end{equation*}
with equality when $A \subset W$, and the above integrand is Borel measurable.
\end{Theorem}

\begin{proof}
The measurability claim follows from \ref{41}(3). In order to prove the inequality
there is of course no restriction to assume $\calH^m_{\|\cdot\|}(A) < \infty$. We will apply \ref{411} with $G = \bZ$. We start by choosing a Borel measurable orientation of the approximate tangent spaces of $A$, say $\xi : A \to \bigwedge_m X$, see \cite[3.2.25]{GMT}. We let $T \in \calR_m(X,G)$ be associated with the set $A$ and the $\bZ$ orientation $[1,\xi(x)]$ at almost every $x \in A$. Thus $\calM_H(T) = \calH_{\|\cdot\|}^m(A)$. Next we recall that for $\calH_{\|\cdot\|}^m$ almost every $y \in W$ 
\begin{equation*}
\la T , \pi ,y \ra = \sum_{x \in A \cap \pi^{-1}\{y\}} (-1)^{\veps(\pi,x)} \bdelta_x \,,
\end{equation*}
where $\veps(\pi,x) = \pm 1$ according to whether the restriction $ \pi : \rmap\rmTan(\calH^m \hel A,x) \to W$ preserves the orientation or not. In particular for $\calH_{\|\cdot\|}^m$ almost every $y \in \pi(A)$ one has $\calM(\la T , \pi , y \ra ) \geq 1$. Our present conclusion then immediately follows from the first conclusion of \ref{411}. 
\end{proof}

\section{New type of Gross measures}

\begin{Empty}[Choice of density contractors]
\label{51}
Given a finite dimensional normed space $(X,\|\cdot\|)$ and an integer $1 \leq m \leq \dim X -1$ we say that {\em $(X,\|\cdot\|)$ admits density contractors of dimension $m$} if for every every $W \in \bG_m(X)$ there exists a density contractor $\mu$ on $W$. Under this assumption the Axiom of Choice guarantees the existence of a choice map 
\begin{equation*}
\bmu : \bG_m(X) \to \calM_1(\rmHom_m(X,X)) \,,
\end{equation*}
i.e. such that $\bmu(W)$ is a density contractor on $W$, for each $W \in \bG_m(X)$. We call such $\bmu$ a {\em choice of density contractors in dimension $m$} and we write $\bmu_W$ instead of $\bmu(W)$. We shall show at the end of this section, \ref{510} that such $\bmu$ can be chosen to be universally measurable, even though this extra property of a choice of density contractors will play no role in the other results presented here.
\end{Empty}

\begin{Empty}
\label{5.assumption}
{\em In the remaining part of this section we assume that $(X,\|\cdot\|)$ admits density contractors of dimension $m$ and we let $W \mapsto \bmu_W$ be a choice of density contractors in dimension $m$.}
\end{Empty}

\begin{Empty}[Gross measure]
\label{53}
With this data ($m$ and $\bmu$) we shall now associate an outer measure on $X$ denoted as $\calG^m_{\bmu}$. To start with, for each Borel subset $A \subset X$ we define
\begin{equation*}
\zeta_{\bmu}(A) = \underset{W \in \bG_m(X)}\sup \int_{\rmHom_m(X,X)} \calH^m_{\|\cdot\|}(\pi(A)) d\bmu_W(\pi) \,.
\end{equation*} 
Now with $A \subset X$ and $0 < \delta \leq \infty$ we associate
\begin{multline*}
\calG^m_{\bmu,\delta}(A) = \inf \Bigg\{ \sum_{j \in J} \zeta_{\bmu}(B_j) : A \subset \cup_{j \in J} B_j \text{ and } \{B_j : j \in J\} \text{ is a} \\
\text{countable family of Borel subsets of $X$ with } \rmdiam B_j < \delta \,, j \in J \Bigg\} \,, 
\end{multline*}
as well as 
\begin{equation*}
\calG^m_{\bmu}(A) = \underset{\delta > 0}\sup\, \calG^m_{\bmu,\delta}(A) \,.
\end{equation*}
All these are outer measures on $X$. Furthermore Borel sets are $\calG^m_{\bmu}$ measurable and $\calG^m_{\bmu}$ is Borel regular. 
\end{Empty}

Our goal is to establish that $\calG^m_{\bmu}(A)=\calH^m_{\|\cdot\|}(A)$ in case $A$ is countably $(\calH^m_{\|\cdot\|},m)$ rectifiable. We start with a trivial observation about rectangular matrices. Recall that $\Lambda(n,m)$ denotes the set of (strictly) increasing maps $\{1,\ldots,m\} \to \{1,\ldots,n\}$ which we may identify with their image. 

\begin{Lemma}
\label{54}
Let $A \in M_{m \times n}(\R)$ and $B,B' \in M_{n \times m}(\R)$ be such that
\begin{equation*}
B = \begin{pmatrix}
I_m \\
0
\end{pmatrix}
\text{ and }
B' = \begin{pmatrix}
I_m \\
E
\end{pmatrix}
\end{equation*}
for some $E \in M_{(n-m) \times m}(\R)$ with $\|E\|_\infty \leq 1$. It follows that
\begin{equation*}
\left| \det(AB) - \det(AB') \right| \leq \bc_{\theTheorem}(n,m) \|E\|_\infty \sum_{\substack{\lambda \in \Lambda(n,m)\\ \lambda \neq \{1,\ldots,m\}}} \left| \det(A_\lambda) \right|
\end{equation*}
were $A_\lambda$ is the square matrix whose $k^{th}$ column coincides with the $\lambda(k)^{th}$ column of $A$.
\end{Lemma}

\begin{proof}
If $F \subset \{1,\ldots,n\}$ is a nonempty set whose elements are numbered $j_1 < \ldots < j_p$ we let $A_F$ denote the matrix whose $k^{th}$ column is the $j_k^{th}$ column of $A$. We also abbreviate $\lambda_0 = \{1,\ldots,m\}$. It is immediate that $AB = A_{\lambda_0 }$. Furthermore
\begin{equation*}
\begin{split}
AB' & = \begin{pmatrix}
A_{\lambda_0} & \begin{pmatrix} A_{\{n-m+1\}} & \cdots & A_{\{n\}} \end{pmatrix}
\end{pmatrix}
\begin{pmatrix}
I_m \\
E
\end{pmatrix}\\
 & = A_{\lambda_0} + \begin{pmatrix} A_{\{n-m+1\}} & \cdots & A_{\{n\}} \end{pmatrix} E \,.
  \end{split}
\end{equation*}
It follows that the $k^{th}$ column of $AB'$ is
\begin{equation*}
A_{\{k\}} + \sum_{i=n-m+1}^n e_{i,k} A_{\{i\}}
\end{equation*}
where $E=(e_{i,j})$. The conclusion now ensues from the multilinearity of the determinant.
\end{proof}

\begin{Proposition}
\label{55}
Assume \ref{5.assumption}.
Let $M \subset X$ be an $m$ dimensional submanifold of $X$ of class $C^1$, $a \in M$ and $0 < \veps < 1$. There then exists $r_0 > 0$ and $W \in \bG_m(X)$ with the following property. For every $0 < r \leq r_0$ one has
\begin{equation*}
\int_{\rmHom_m(X,X)} \calH^m_{\|\cdot\|}\left( \pi(M \cap \bB_{\|\cdot\|}(a,r)) \right) d\bmu_W(\pi) \geq (1-\veps) \calH^m_{\|\cdot\|}\left(M \cap \bB_{\|\cdot\|}(a,r)\right) \,.
\end{equation*}
\end{Proposition}

\begin{proof}
Abbreviate $W = \rmTan(M,a)$. Owing to the translation invariance of the Hausdorff measure and to the linearity of $\pi$ in the statement, replacing $M$ by $M-a$ we may assume that $a=0$. As usual we consider a Euclidean structure on $X$, and we let $C \geq 1$ be such that $C^{-1}\|x\| \leq |x| \leq C \|x\|$ for every $x \in X$ and $C^{-1} \leq \psi(W) \leq C$ for every $W \in \bG_m(X)$ (recall \ref{25}). We let $0 < \hat{\veps} < 1$ be undefined for now.
\par 
We define
\begin{equation*}
\rmHom^{\sim}(X,W) = \rmHom(X,W) \cap \{ \pi : \rmrank \pi \leq m-1 \}
\end{equation*}
as well as
\begin{equation*}
\rmHom_{\rminv}(X,W) = \rmHom(X,W) \cap \{ \pi : \rmrank \pi = m \}\,,
\end{equation*}
and we notice the first set is relatively closed and second one relatively open.
For each $\kappa=1,2,\ldots$ we define
\begin{equation*}
\rmHom_{\rminv,\kappa}(X,W) = \rmHom_{\rminv}(X,W) \cap \left\{ \pi : \rmLip_{|\cdot|} \pi \leq \kappa \text{ and }  \rmLip_{|\cdot|} \left(\pi|_W^{-1}\right) \leq \kappa \right\} \,.
\end{equation*}
Since $\rmLip_{|\cdot|}\pi = \vvvert \pi \vvvert$, and since $\rmHom_{\rminv}(X,W) \to \rmHom(W,W) : \pi \mapsto \pi|_W^{-1}$ is continuous, it follows that $\rmHom_{\rminv,\kappa}(X,W)$ is relatively closed. Also,
\begin{equation*}
\rmHom_{\rminv}(X,W) = \cup_{\kappa=1}^\infty \rmHom_{\rminv,\kappa}(X,W) \,.
\end{equation*}
We abbreviate $H_\kappa = \rmHom^{\sim}(X,W) \cup \rmHom_{\rminv,\kappa}(X,W)$ and we infer from the monotone convergence theorem that there exists $\kappa$ such that
\begin{multline}
\label{eq.thu.15}
\int_{H_\kappa} \calH^m_{\|\cdot\|}\left( \pi(W \cap \bB_{\|\cdot\|}(0,1)) \right) d\bmu_W(\pi) \\ \geq \left(1 - \hat{\veps} \right) \int_{\rmHom_m(X,X)} \calH^m_{\|\cdot\|}\left( \pi(W \cap \bB_{\|\cdot\|}(0,1)) \right) d\bmu_W(\pi)\\
=  \left(1 - \hat{\veps} \right) \calH^m_{\|\cdot\|}\left( W \cap \bB_{\|\cdot\|}(0,1) \right)
= \left(1 - \hat{\veps} \right) \balpha(m)
\end{multline}
where the last equality is H. Busemann's identity \eqref{eq.busemann.1}.
For a reason that will become clear momentarily we then define
\begin{equation*}
\check{\veps} = \min \left\{ \hat{\veps} , \frac{1}{2\kappa^2} \right\} \,.
\end{equation*}
\par 
There exists $r_1 > 0$ and $f : W \cap \bB_{|\cdot|}(0,r_1) \to W^\perp$ of class $C^1$ such that 
\begin{enumerate}
\item[(a)] $\rmLip_{|\cdot|} f < \check{\veps}$;
\item[(b)] $F(W \cap \bB_{|\cdot|}(0,r_1)) \subset M$,
\end{enumerate}
where $F = \iota_{W} + \iota_{W^\perp} \circ f$, and $\iota_{W}$ and $\iota_{W^\perp}$ are the obvious inclusion maps. We let $\pi_W$ denote the orthogonal projection onto $W$.
It clearly follows that if $0 < r \leq r_1/C$ then
\begin{enumerate}
\item[(c)] $F\left(\pi_W(M\cap\bB_{\|\cdot\|}(0,r))\right) = M \cap \bB_{\|\cdot\|}(0,r)$;
\end{enumerate}
and furthermore, 
\begin{enumerate}
\item[(d)] $W \cap \bB_{\|\cdot\|}\left(0,(1-C^2\hat{\veps})r\right) \subset \pi_W(M\cap\bB_{\|\cdot\|}(0,r)) \subset W \cap \bB_{\|\cdot\|}\left(0,(1+C^2\hat{\veps})r\right)$,
\end{enumerate}
 because for each $x \in W$ one has $\|F(x)\| \leq \|x\| + \|f(x)\| \leq \|x\| + C^2 \left(\rmLip_{|\cdot|} f\right) \|x\|$.
 \par 
 For the remaining part of this proof we fix $0 < r \leq r_1/C$ and we abbreviate $D_r = \pi_W(M\cap\bB_{\|\cdot\|}(0,r))$. For a while we turn to computing Hausdorff measures $\calH^m_{|\cdot|}$ with respect to the Euclidean norm $|\cdot|$. 
 \par 
We claim that if $\pi \in \rmHom^\sim(X,W)$ or $\pi \in \rmHom_{\rminv,\kappa}(X,W)$, applying the (Euclidean) area formula to the mappings $\pi \circ F : W \to W$ and $\pi|_W : W \to W$, \cite[3.2.22]{GMT} we obtain
 \begin{equation}
 \label{eq.thu.17}
 \calH^m_{|\cdot|} \left( \pi(M \cap \bB_{\|\cdot\|}(0,r))\right) = \calH^m_{|\cdot|} \left( \pi(F(D_r)) \right) = \int_{W \cap D_r} J_m(\pi \circ F)(x) d \calH^m_{|\cdot|}(x)
 \end{equation}
 and
 \begin{equation}
 \label{eq.thu.18}
  \calH^m_{|\cdot|} \left( \pi(D_r) \right) = \int_{W \cap D_r} J_m \pi(x) d \calH^m_{|\cdot|}(x) \,.
 \end{equation}
Indeed if $\pi \in \rmHom^\sim(X,W)$ then the two sides of both equations obviously vanish. Furthermore the second equation is in fact valid for every $\pi \in \rmHom_{\rminv}(X,W)$ as well, because in that case $\pi|_W$ is injective.
Therefore it remains only to establish the first equation in case $\pi \in \rmHom_{\rminv,\kappa}(X,W)$. The reason for which it holds true is that $(\pi \circ F)|_{D_r}$ is injective. Indeed if $\pi(F(x))=\pi(F(x'))$ for $x,x' \in D_r$, then $\pi(\iota_{W}(x))-\pi(\iota_{W}(x')) = \pi(\iota_{W^\perp}(f(x'))) - \pi(\iota_{W^\perp}(f(x)))$ and in turn
\begin{multline*}
\kappa^{-1} |x-x'| \leq \left( \rmLip_{|\cdot|} \pi|_W^{-1} \right)^{-1} |x-x'| \leq |\pi(x)-\pi(x')| \\
= \left| \pi(f(x)) - \pi(f(x')) \right| \leq \left(\rmLip_{|\cdot|}  \pi \right) \left( \rmLip_{|\cdot|}  f \right) |x-x'| \leq \kappa \check{\veps} |x-x'|
\end{multline*}
so that $x=x'$ because $\check{\veps} \leq \kappa^{-2}/2$.
\par 
From this ensues that if $\pi \in H_\kappa$ then
\begin{multline}
\label{eq.thu.11}
\left|\calH^m_{|\cdot|} \left( \pi(M \cap \bB_{\|\cdot\|}(0,r))\right) -  \calH^m_{|\cdot|} \left( \pi(D_r) \right) \right| \\ \leq \int_{W \cap D_r} \left|J_m(\pi \circ F)(x) -  J_m \pi(x) \right| d \calH^m_{|\cdot|}(x)\,.
\end{multline} 
We choose an orthonormal basis $e_1,\ldots,e_n$ of $X$ such that $e_1,\ldots,e_m$ is a basis of $W$, and we let $A \in M_{m \times n}(\Rn)$ be the matrix of $\pi$ with respect to these bases. Letting $B \in M_{n \times m}(\Rn)$ be defined as in \ref{54} it is clear that $AB$ is the matrix of $\pi|_W$ with respect to the matrix $e_1,\ldots,e_m$ of $W$, thus $J_m \pi = |\det(AB)|$. Given $x \in D_r$ we notice that $DF(x) = \iota_W + \iota_{W^\perp} \circ Df(x)$ whence the matrix $B_x'$ of $DF(x)$ with respect to $e_1,\ldots,e_m$ and $e_1,\ldots,e_n$ is of the form of $B'$ in \ref{54} with $E_x$ being the matrix of $Df(x)$ with respect to the obvious bases. In particular $\|E_x\|_\infty \leq \rmLip_{|\cdot|} f \leq \hat{\veps}$. Moreover $J_m(\pi \circ F)(x) = |\det(AB_x')|$.  Thus it follows from \ref{54} that
\begin{multline}
\label{eq.thu.12}
\left| J_m(\pi \circ F)(x) - J_m\pi(x)  \right| = \left| |\det(AB'_x)| - |\det(AB)| \right| \\ \leq \bc_{\ref{54}}(n,m)\hat{\veps} \sum_{\substack{\lambda \in \Lambda(n,m)\\ \lambda \neq \{1,\ldots,m\}}} \left| \det(A_\lambda) \right| \,.
\end{multline}
Now given $\lambda \in \Lambda(n,m)$ we let $W_\lambda = \rmspan\{e_{\lambda(1)},\ldots,e_{\lambda(m)}\}$ and we select $O_\lambda : X \to X$ be a isometric linear isomorphism such that $O_\lambda(W)=W_\lambda$. It then becomes clear that $A_\lambda$ is the matrix of $(\pi \circ O_\lambda)|_W$ with respect to the basis $e_1,\ldots,e_m$ of $W$. Thus $|\det(A_\lambda)| = J_m(\pi \circ O_\lambda)|_W$ and in turn
\begin{equation*}
\int_{W \cap D_r} |\det(A_\lambda)| d \calH^m_{|\cdot|} = \calH^m_{|\cdot|}\left(\pi(O_\lambda(D_r))\right)
\end{equation*}
Together with \eqref{eq.thu.11} and \eqref{eq.thu.12} this yields
\begin{equation*}
\begin{split}
\bigg|\calH^m_{|\cdot|} \left( \pi(M \cap \bB_{\|\cdot\|}(0,r))\right) & -  \calH^m_{|\cdot|} \left( \pi(D_r) \right) \bigg| \\
& \leq \bc_{\ref{54}}(n,m) \hat{\veps} \sum_{\substack{\lambda \in \Lambda(n,m)\\ \lambda \neq \{1,\ldots,m\}}}\calH^m_{|\cdot|}\left(\pi(O_\lambda(D_r))\right)
\end{split}
\end{equation*}
Since all sets whose measure appear in the above inequality are subsets of $W$, multiplying both sides by $\psi(W)$ we obtain at once that
\begin{equation*}
\begin{split}
\bigg|\calH^m_{\|\cdot\|} \left( \pi(M \cap \bB_{\|\cdot\|}(0,r))\right) & -  \calH^m_{\|\cdot\|} \left( \pi(D_r) \right) \bigg| \\
& \leq \bc_{\ref{54}}(n,m) \hat{\veps} \sum_{\substack{\lambda \in \Lambda(n,m)\\ \lambda \neq \{1,\ldots,m\}}}\calH^m_{\|\cdot\|}\left(\pi(O_\lambda(D_r))\right)
\end{split}
\end{equation*}
Next we integrate both sides with respect to $\bmu_W$ on the set $H_\kappa =  \rmHom^\sim(X,W) \cup \rmHom_{\rminv,\kappa}(X,W)$:
\begin{multline}
\label{eq.thu.13}
\Bigg|\int_{H_\kappa}\calH^m_{\|\cdot\|}(\pi(D_r))d\bmu_W(\pi) - \int_{H_\kappa} \calH^m_{\|\cdot\|} \left( \pi(M \cap \bB_{\|\cdot\|}(0,r))\right) d\bmu_W(\pi) \Bigg| \\
 \leq \bc_{\ref{54}(n,m)} \hat{\veps} \sum_{\substack{\lambda \in \Lambda(n,m)\\ \lambda \neq \{1,\ldots,m\}}} \int_{\rmHom_m(X,X)}\calH^m_{\|\cdot\|}\left(\pi(O_\lambda(D_r))\right)d\bmu(\pi)\\
 \leq \bc_{\ref{54}(n,m)} \hat{\veps} \sum_{\substack{\lambda \in \Lambda(n,m)\\ \lambda \neq \{1,\ldots,m\}}} \calH^m_{\|\cdot\|}\left(O_\lambda(D_r)\right) \,.
\end{multline}
It follows from (d) above and \eqref{eq.thu.15} that
\begin{multline}
\label{eq.thu.24}
\int_{H_\kappa} \calH^m_{\|\cdot\|} \left( \pi(D_r) \right) d\bmu_W(\pi) \geq \left( 1-C\hat{\veps} \right)^mr^m \int_{H_\kappa} \calH^m_{\|\cdot\|}\left( \pi(W \cap \bB_{\|\cdot\|}(0,1)) \right) d\bmu_W(\pi)\\ \geq \left( 1 - \hat{\veps} \right) \left( 1-C\hat{\veps} \right)^m \balpha(m) r^m \,.
\end{multline}
It further follows from (d) above and H. Busemeann's identity \eqref{eq.busemann.1} that
\begin{equation*}
 \calH^m_{\|\cdot\|}(D_r) \leq  \left( 1 + C^2 \hat{\veps} \right)^m \balpha(m)r ^m \,, 
\end{equation*}
and in turn
\begin{multline*}
\calH^m_{\|\cdot\|}(O_\lambda(D_r)) = \psi(W_\lambda) \calH^m_{|\cdot|}(O_\lambda(D_r)) = \psi(W_\lambda) \calH^m_{|\cdot|}(D_r) \\
= \psi(W_\lambda) \psi(W)^{-1}\calH^m_{\|\cdot\|}(D_r) \leq \balpha(m) C^2 \left( 1 + C^2\right)^m r ^m
\end{multline*}
Using these and \eqref{eq.thu.13} we see that
\begin{multline*}
\int_{\rmHom_m(X,X)} \calH^m_{\|\cdot\|} \left( \pi(M \cap \bB_{\|\cdot\|}(0,r))\right) d\mu(\pi) \\
\geq \bigg( \left( 1 - \hat{\veps} \right)\left( 1 - C^2\hat{\veps}\right)^m - \hat{\veps} \bc_{\ref{54}}(n,m)\left( \rmcard \Lambda(n,m)\right) C^2\left(1+C^2\right)^m \bigg) \balpha(m)r^m \,.
\end{multline*}
Furthermore it follows from B. Kirchheim's area formula \cite{KIR.94} that there exists $r_2 > 0$ such that for every $0 < r \leq r_2$ one has
\begin{equation*}
(1 + \hat{\veps}) \balpha(m)r^m \geq \calH^m_{\|\cdot\|}\left( M \cap \bB_{\|\cdot\|}(0,r) \right) \,.
\end{equation*}
Finally we let $r_0 = \min \{ r_1/C , r_2 \}$ and it should now be obvious how to choose $\hat{\veps}$ according to $\veps$, $n$, $m$ and $C$ so that the conclusion holds.
\end{proof}

\begin{Lemma}
\label{56}
$\zeta_{\bmu}(A) \leq \calG^m_{\bmu}(A)$ whenever $A \subset X$ is Borel. 
\end{Lemma}

\begin{proof}
If $\{ B_j : j \in J \}$ is a countable family of Borel subsets of $X$ so that $A \subset \cup_{j \in J} B_j$ then $\calH^m_{\|\cdot\|}(\pi(A)) \leq \sum_{j \in J} \calH^m_{\|\cdot\|}(\pi(B_j))$ for every $\pi \in \rmHom_m(X,X)$. Thus for each $W \in \bG_m(X)$ one has
\begin{equation*}
\int_{\rmHom_m(X,X)} \calH^m_{\|\cdot\|}(\pi(A)) d\bmu_W(\pi) \leq \sum_{j \in J} \int_{\rmHom_m(X,X)} \calH^m_{\|\cdot\|}(\pi(B_j)) d\bmu_W(\pi) \,.
\end{equation*}
Taking the supremum over $W \in \bG_m(X)$ on both sides yields $\zeta_{\bmu}(A) \leq \sum_{j \in J} \zeta_{\bmu}(B_j)$. The conclusion follows from the arbitrariness of $\{ B_j : j \in J \}$.
\end{proof}

\begin{Theorem}
\label{57}
If $A \subset X$ is Borel and countably $(\calH^m_{\|\cdot\|},m)$ rectifiable then $\calG^m_{\bmu}(A) = \calH^m_{\|\cdot\|}(A)$.
\end{Theorem}

\begin{proof}
We start by showing that $\calG^m_{\bmu}(A) \leq \calH^m_{\|\cdot\|}(A)$.
Given $\delta > 0$ choose a countable Borel partition $\{B_j : j \in J \}$ of $A$ with $\rmdiam B_j < \delta$, $j \in J$ (for instance the $B_j$ are the intersection of $A$ with dyadic semicubes of some generation). Since each $B_j$ is countably $(\calH^m_{\|\cdot\|},m)$ rectifiable it follows from \ref{413} that $\zeta_{\bmu}(B_j) \leq \calH^m_{\|\cdot\|}(B_j)$. Therefore
\begin{equation*}
\calG^m_{\bmu,\delta}(A) \leq \sum_{j \in J} \zeta_{\bmu}(B_j) \leq \sum_{j \in J}  \calH^m_{\|\cdot\|}(B_j) = \calH^m_{\|\cdot\|}(A)
\end{equation*}
and it remains to let $\delta \to 0^+$.
\par 
In order to establish the reverse inequality we start with the case when $A=M$ is an $m$ dimensional submanifold of $X$ of class $C^1$ such that $\calH^m_{\|\cdot\|}(M) < \infty$. Let $\veps > 0$. It follows from \ref{55} that there exists $r(a) > 0$ and $W \in \bG_m(X)$ such that for every $0 < r \leq r(a)$  one has
\begin{equation*}
\int_{\rmHom_m(X,X)} \calH^m_{\|\cdot\|}\left( \pi(M \cap \bB_{\|\cdot\|}(a,r)) \right) d\bmu_W(\pi) \geq (1-\veps) \calH^m_{\|\cdot\|}\left(M \cap \bB_{\|\cdot\|}(a,r)\right) \,.
\end{equation*}
Thus also
\begin{equation*}
\zeta_{\bmu}\left( M \cap \bB_{\|\cdot\|}(a,r) \right) \geq (1-\veps) \calH^m_{\|\cdot\|}\left(M \cap \bB_{\|\cdot\|}(a,r)\right) \,.
\end{equation*}
It follows from \cite[2.8.9 and 2.8.18]{GMT} that $\{\bB_{\|\cdot\|}(a,r) : a \in M \text{ and } 0 < r \leq r(a) \}$ is an $\calH^m_{\|\cdot\|} \hel M$ Vitali relation. Therefore there exists a countable subset $\{a_j : j \in J\}$ of $M$ and corresponding $0 < r_j \leq r(a_j)$, $j \in J$, such that, upon abbreviating $B_j = \bB_{\|\cdot\|}(a_j,r_j)$, the family $\{ B_j : j \in J\}$ is disjointed and $\calH^m_{\|\cdot\|}\left( M \setminus \cup_{j \in J} B_j \right) = 0$. Thus,
\begin{multline*}
(1-\veps) \calH^m_{\|\cdot\|} (M) = \sum_{j \in J} (1-\veps) \calH^m_{\|\cdot\|}(M \cap B_j) \leq \sum_{j \in J} \zeta_{\bmu}(M \cap B_j) \\
\leq \sum_{j \in J} \calG^m_{\bmu}(M \cap B_j) \leq \calG^m_{\bmu}(M) \,,
\end{multline*}
where the second inequality is a consequence of \ref{56}. It follows from the arbitrariness of $\veps > 0$ that the inequality is established in case $A = M$ is an $m$ dimensional $C^1$ submanifold. We consider the Borel finite measures $\phi_\calH = \calH^m_{\|\cdot\|} \hel M$ and $\phi_\calG = \calG^m_{\bmu} \hel M$ (that $\phi_\calG$ be finite follows from the fact that so is $\phi_\calH$ -- by assumption -- and the first part of this proof). Now if $U \subset X$ is open then $M \cap U$ is also an $m$ dimensional $C^1$ submanifold and therefore $\phi_\calH(U) \leq \phi_\calG(U)$. Since $\phi_\calG$ is outer regular (being a finite Borel measure in a metric space) we infer that $\phi_\calH(A) \leq \phi_\calG(A)$ for all Borel subsets $A \subset M$. Assuming now that $A$ is a Borel countable $(\calH^m_{\|\cdot\|},m)$ rectifiable subset of $X$, the conclusion follows from the fact that $A$ admits a partition into Borel sets $A_0,A_1,\ldots$ such that $\calH^m_{\|\cdot\|}(A_0)=0$ and each $A_j$, $j \geq 1$, is a subset of some $m$ dimensional submanifold of $X$ of class $C^1$, \cite[3.2.29]{GMT}.
\end{proof}

We close this section by showing that there exists a universally measurable choice of density contractors.

\begin{Proposition}
\label{58}
Assume that $W,W_1,W_2,\ldots$ are members of $\bG_m(X)$ and that $d(W,W_k) \to 0$ as $k \to \infty$. For every $n=1,2,\ldots$ the following holds:
\begin{equation*}
\lim_k \sup_{\substack{\pi \in \rmHom_m(X,X)\\\vvvert \pi \vvvert \leq n}} \Bigg| \calH^m_{\|\cdot\|} \left( \pi( W_k \cap B_{|\cdot|} )\right) - \calH^m_{\|\cdot\|} \left( \pi( W \cap B_{|\cdot|} )\right) \Bigg| = 0 \,.
\end{equation*}
\end{Proposition}

\begin{proof}
We show how this is a consequence of Steiner's formula. Let $\pi \in \rmHom_m(X,X)$ and choose $W \in \bG_m(X)$ such that $\rmim \pi \subset W$. Through the choice of an orthonormal basis we identify $W$ with $\ell_2^m$. For any convex set $C \subset W$ and $\delta > 0$ we have
\begin{equation}
\label{eq.steiner}
\calH^m_{|\cdot|}(\bB_{|\cdot|}(C,\delta)) = \calH^m_{|\cdot|}(C) + \sum_{k=0}^{m-1} \delta^{m-k} \balpha(m-k) \zeta^k(C) 
\end{equation}
where 
\begin{equation*}
\zeta^k(C) = \bbeta_1(m,k)^{-1} \int_{\bO^*(m,k)} \calL^k(p(C)) d\btheta_{m,k}(p) \,,
\end{equation*}
see for instance \cite[3.2.35]{GMT}. Now if $n$ is integer so that $C \subset W \cap \bB_{|\cdot|}(0,n)$ then $p(C) \subset \R^k \cap \bB(0,n)$ for all $p \in \bO^*(m,k)$ and therefore $\zeta^k(C) \leq \bbeta_1(m,k)^{-1}\balpha(k)n^k$. Thus if $C,C' \subset W \cap \bB_{|\cdot|}(0,n)$ are compact convex and $\delta = \rmdist_\calH(C,C') \leq 1$ then it follows from \eqref{eq.steiner} that
\begin{equation*}
\left| \calH^m_{|\cdot|}(C) - \calH^m_{|\cdot|}(C') \right| \leq \delta \bc(m)
\end{equation*}
where
\begin{equation*}
\bc(m) = \sum_{k=0}^{m-1} \bbeta_1(m,k)^{-1}\balpha(m-k)\balpha(k)n^k \,.
\end{equation*}
Now we let $C_k =  \pi( W_k \cap B_{|\cdot|} )$, $C =  \pi( W \cap B_{|\cdot|} )$, $\delta_k = \rmdist_\calH(C_k,C)$ and we assume that $\vvvert \pi \vvvert \leq n$. Since $d(W_k,W) \to 0$ we have $\delta_k \leq n \rmdist (W_k \cap B_{|\cdot|},W \cap B_{|\cdot|} ) \to 0$ as well. Finally,
\begin{multline*}
 \Bigg| \calH^m_{\|\cdot\|} \left( \pi( W_k \cap B_{|\cdot|} )\right) - \calH^m_{\|\cdot\|} \left( \pi( W \cap B_{|\cdot|} )\right) \Bigg| \\
 =  \psi(W) \Bigg| \calH^m_{|\cdot|} \left( \pi( W_k \cap B_{|\cdot|} )\right) - \calH^m_{|\cdot|} \left( \pi( W \cap B_{|\cdot|} )\right) \Bigg| \\
 \leq C \bc(m)n \rmdist (W_k \cap B_{|\cdot|},W \cap B_{|\cdot|} )
\end{multline*}
if $k$ is large enough (for $n \delta_k \leq 1$), where $C = \max \{ \psi(W) : W \in \bG_m(X) \}$.
\end{proof}

\begin{Empty}[Space of probability measures]
\label{50}
We recall that $\rmHom_m(X,X)$, the collection of linear maps $X \to X$ of rank at most $m$, is a closed subspace of $\rmHom(X,X)$ and therefore a Polish space. The set $\calM_1(\rmHom_m(X,X))$ consisting of those probability Borel measures on $\rmHom_m(X,X)$ is itself a Polish space, equipped with the so called topology of weak convergence of measures, \cite[Chapter II \S 6]{PARTHASARATHY}.
\end{Empty}

\begin{Theorem}
\label{59}
The set
\begin{multline*}
\calE = \bG_m(X) \times \calM_1(\rmHom_m(X,X)) \cap \bigg\{ (W,\mu) : \mu \text{ is a density contractor} \\ \text{on $W$ with respect to $\calH^m_{\|\cdot\|}$} \bigg\}
\end{multline*}
is Borel.
\end{Theorem}

\begin{proof}
We define 
\begin{equation*}
\calE_a = \{ (W,\mu) : \mu \text{ is supported in } \rmHom(X,W) \}
\end{equation*}
and 
\begin{multline*}
\calE_b = \bigg\{ (W,\mu) : \text{for every } V \in \bG_m(X) \text{ one has } \\
\int_{\rmHom_m(X,X)} \calH^m_{\|\cdot\|}\left( \pi(V \cap B_{|\cdot|} )\right) d\mu(\pi) \leq \calH^m_{\|\cdot\|}\left( V \cap B_{|\cdot|} \right) \bigg\}
\end{multline*}
and
\begin{equation*}
\calE_b = \left\{ (W,\mu) :
\int_{\rmHom_m(X,X)} \calH^m_{\|\cdot\|}\left( \pi(W \cap B_{|\cdot|} )\right) d\mu(\pi) = \calH^m_{\|\cdot\|}\left( W \cap B_{|\cdot|} \right) \right\} \,.
\end{equation*}
It follows from \ref{46} that $\calE = \calE_{a} \cap \calE_b  \cap \calE_c$.
\par 
{\it The set $\calE_a$ is closed}. Let $(W_1,\mu_1),(W_2,\mu_2),\ldots$ be members of $\calE_a$ such that $(W_k,\mu_k) \to (W,\mu)$ in $\bG_m(X) \times \calM_1(\rmHom_m(X,X))$ as $k \to \infty$. We observe that $\rmHom_m(X,X) \setminus \rmHom(X,W) = \cup_{j=1}^\infty \frU_j$ where
\begin{equation*}
\frU_j = \rmHom_m(X,X) \cap \left\{ \pi : \rmdist \left( \pi , \rmHom(X,W) \right) > \frac{\vvvert \pi  \vvvert}{j} \right\} \,.
\end{equation*}
Since $\vvvert \pi_{W_k} - \pi_W \vvvert \to 0$ by assumption, it is easily seen that for each fixed $j=1,2,\ldots$ there exists an integer $k(j)$ such that $\rmHom(X,W_k) \cap \frU_j = \emptyset$ whenever $k \geq k(j)$. For such $k \geq k(j)$ it follows that $\mu_k(\frU_j)=0$. Since $\frU_j$ is open we have $\mu(\frU_j) \leq \liminf_k \mu_k(\frU_j) = 0$. It follows that $\rmsupp \mu \subset \rmHom(X,W)$.
\par 
{\it The set $\calE_b$ is closed}. Given $V \in \bG_m(X)$ we define
\begin{equation*}
\calE_{b,V }= \left\{ (W,\mu) : 
\int_{\rmHom_m(X,X)} \calH^m_{\|\cdot\|}\left( \pi(V \cap B_{|\cdot|} )\right) d\mu(\pi) \leq \calH^m_{\|\cdot\|}\left( V \cap B_{|\cdot|} \right) \right\} \,.
\end{equation*}
Since clearly $\calE_b = \cap_{V \in \bG_m(X)} \calE_{b,V}$ it suffices to show that each $\calE_{b,V}$ is closed. Define $\Upsilon_V : \bG_m(X) \times \calM_1(\rmHom_m(X,X)) \to \R$ by the formula
\begin{equation*}
\Upsilon_V(W,\mu) = \int_{\rmHom_m(X,X)} \calH^m_{\|\cdot\|}\left( \pi(V \cap B_{|\cdot|} )\right) d\mu(\pi)
\end{equation*}
(thus $\Upsilon_V$ does depend upon its first variable $W$). As before define $f_{V \cap B_{|\cdot|}} : \rmHom_m(X,X) \to \R$ by $f_{V \cap B_{|\cdot|}}(\pi) = \calH^m_{\|\cdot\|}\left(\pi(V \cap B_{|\cdot|})\right)$. It follows from \ref{41}(2) that for each $n=1,2,\ldots$ the function $\min\{n,f_{V \cap B_{|\cdot|}}\}$ is bounded and continuous. Given $(W_1,\mu_1),(W_2,\mu_2),\ldots$ members of $\bG_m(X) \times \calM_1(\rmHom_m(X,X))$ such that $(W_k,\mu_k) \to (W,\mu)$ we infer that, for each $n=1,2,\ldots$,
\begin{multline*}
\int_{\rmHom_m(X,X)} \min\{n,f_{V \cap B_{|\cdot|}}\} d\mu = \lim_k \int_{\rmHom_m(X,X)} \min\{n,f_{V \cap B_{|\cdot|}}\} d\mu_k \\
\leq \liminf_k \int_{\rmHom_m(X,X)} f_{V \cap B_{|\cdot|}} d\mu_k = \liminf_k \Upsilon_V(W_k,\mu_k) \,.
\end{multline*}
Letting $n\to\infty$ it then follows from the monotone convergence theorem that
\begin{multline*}
\Upsilon_V(W,\mu) = \int_{\rmHom_m(X,X)} f_{V \cap B_{|\cdot|}} d\mu \\
=\lim_n \int_{\rmHom_m(X,X)} \min\{n,f_{V \cap B_{|\cdot|}}\} d\mu \leq \liminf_k \Upsilon_V(W_k,\mu_k)\,.
\end{multline*}
This shows that $\Upsilon_V$ is lower semicontinuous and in turn that $\calE_{b,V}$ is closed.
\par 
{\it The set $\calE_c$ is Borel.} We define $\Upsilon : \bG_m(X) \times \calM_1(\rmHom_m(X,X)) \to \R$ by the formula
\begin{equation*}
\Upsilon(W,\mu) = \int_{\rmHom_m(X,X)} \calH^m_{\|\cdot\|}\left( \pi(W \cap B_{|\cdot|} )\right) d\mu(\pi)
\end{equation*}
so that 
\begin{equation*}
\calE_c = \bG_m(X) \times \calM_1(\rmHom_m(X,X)) \cap \left\{ (W,\mu) : \Upsilon(W,\mu) = \calH^m_{\|\cdot|}\left( W \cap B_{|\cdot|} \right) \right\} \,.
\end{equation*}
Since $\calH^m_{\|\cdot|}\left( W \cap B_{|\cdot|} \right) = \psi(W) \balpha(m)$ is continuous according to \ref{25} the conclusion will follow from the lower semicontinuity of $\Upsilon$ which we now establish. We choose $\chi_1,\chi_2,\ldots$ a nondecreasing sequence of cut-off functions on $\rmHom_m(X,X)$ as in \eqref{eq.cutoff}. With each $n=1,2,\ldots$ we then associate
\begin{equation*}
\Upsilon_n(W,\mu) = \int_{\rmHom_m(X,X)} \chi_n(\pi)\calH^m_{\|\cdot\|}\left( \pi(W \cap B_{|\cdot|} )\right) d\mu(\pi) \,.
\end{equation*}
As $\Upsilon = \sup_{n=1,2,\ldots} \Upsilon_n$ according to the monotone convergence theorem, it is enough to show that each $\Upsilon_n$ is continuous. Fix $n=1,2,\ldots$. Let $(W_1,\mu_1),(W_2,\mu_2),\ldots$ be members of $\bG_m(X) \times \calM_1(\rmHom_m(X,X))$ such that $(W_k,\mu_k) \to (W,\mu)$. It follows from \ref{41}(2) that $\chi_n f_{W_1 \cap B_{|\cdot|}},\chi_n f_{W_2 \cap B_{|\cdot|}},\ldots$ is a sequence in $C_b(\rmHom_m(X,X))$ and it follows from \ref{58} that it converges uniformly to $\chi_n f_{W \cap B_{|\cdot|}}$. The weak* convergence of $\mu_1,\mu_2,\ldots$ clearly implies its uniform convergence on compact subsets of $C_b(\rmHom_m(X,X))$, thus
\begin{multline*}
\Upsilon_n(W,\mu) = \int_{\rmHom_m(X,X)} \chi_n  f_{W \cap B_{|\cdot|}} d\mu \\
= \lim_k \int_{\rmHom_m(X,X)} \chi_n  f_{W_k \cap B_{|\cdot|}} d\mu_k = \lim_k \Upsilon_n(W_k,\mu_k) 
\end{multline*} 
\end{proof}

\begin{Theorem}
\label{510}
Under the assumption \ref{5.assumption} there exists a choice of density contractors $\bmu : \bG_m(X) \to \calM_1(\rmHom_m(X,X))$ which is universally measurable.
\end{Theorem}

\begin{proof}
Since $\bG_m(X)$ and $\calM_1(\rmHom_m(X,X))$ are Polish spaces, and $\calE$ is Borel according to \ref{59}, the result is a consequence of J. von Neumann's selection theorem, \cite[5.5.2]{SRIVASTAVA}.
\end{proof}

\section{New type of Gross mass}

\begin{Empty}[Gross mass]
\label{61}
In this section we work again under the same assumption as \ref{5.assumption}: $\bmu$ is a choice of density contractors of dimension $m$. The corresponding Gross measure $\calG^m_{\bmu}$ is defined in the previous section. Given $T \in \calR_m(X,G)$ we define its {\em Gross mass} as
\begin{equation*}
\calM_G(T) = \int_{\rmset_m \|T\|} \lno \bg(x) \rno d\calG^m_{\bmu}(x) \,.
\end{equation*} 
It follows at once from \ref{57}, approximation by simple functions and the monotone convergence theorem that
\begin{equation}
\label{eq.thu.14}
\calM_G(T) = \calM_H(T) \,.
\end{equation}
We also define
\begin{equation*}
\zeta_{\bmu}(T) = \underset{W \in \bG_m(X)}\sup \int_{\rmHom_m(X,X)} \calM_H(\pi_\#T) d\bmu_W(\pi) \,.
\end{equation*}
We recall that integrand in the above formula is Borel measurable, and that 
\begin{equation}
\label{eq.thu.16}
\zeta_{\bmu}(T) \leq \calM_H(T) \,,
\end{equation}
 both according to \ref{411}. We also notice that $\zeta_{\bmu}(T_1+T_2) \leq \zeta_{\bmu}(T_1) + \zeta_{\bmu}(T_2)$ and $\zeta_{\bmu}(T)=\zeta_{\bmu}(-T)$.
\end{Empty}

\begin{Proposition}
\label{62}
$\zeta_{\bmu} : \calR_m(X,G) \to \R$ is lower semicontinuous with respect to $\calF$ convergence.
\end{Proposition}

\begin{proof}
Let $T,T_1,T_2,\ldots$ be members of $\calR_m(X,G)$ such that $\calF(T-T_j) \to 0$ as $j \to \infty$.
Fix $W \in \bG_m(X)$ and $\pi \in \rmHom(X,W)$. Since $\calF(\pi_\# T - \pi_\# T_j) \to 0$ and $\calF(S) = \calM_H(S)$ for each $S \in \calR_m(X,G)$ we infer that $\calM_H(\pi_\#T) = \lim_j \calM_H(\pi_\# T_j)$. It thus ensues from Fatou's Lemma that
\begin{equation*}
\int_{\rmHom_m(X,X)} \calM_H(\pi_\#T)d\bmu_W(\pi) \leq \liminf_j \int_{\rmHom_m(X,X)} \calM_H(\pi_\#T_j)d\bmu_W(\pi) \,.
\end{equation*}
Taking the supremum over $W \in \bG_m(X)$ on both sides yields the conclusion.
\end{proof}

\begin{Proposition}
\label{63}
Let $T \in \calR_m(X,G)$ and $0 < \veps < 1$. For $\calH^m_{\|\cdot\|}$ almost every $x \in \rmset_m \|T\|$ there then exist $r(x) > 0$ with the following property. For every $0 < r \leq r(x)$ one has
\begin{equation*}
\zeta_{\bmu}\left( T \hel \bB_{\|\cdot\|}(x,r) \right) \geq (1-\veps) \calM_H \left( T \hel \bB_{\|\cdot\|}(x,r) \right) \,.
\end{equation*}
\end{Proposition}

\begin{proof}
Part of the proof is very similar to that of \ref{55}. We indicate how to modify the argument and leave out the details. We start with the case when $\rmsupp T \subset M$ for some $m$ dimensional submanifold $M \subset X$ of class $C^1$. We restrict to $\calH_{\|\cdot\|}^m$ almost every $x \in \rmset_m \|T\|$ by considering only those which are Lebesgue points of $\lno \bg \rno$, \cite[2.9.9]{GMT}, i.e.
\begin{equation*}
0= \lim_{r \to 0^+} \frac{1}{\balpha(m)r^m} \int_{\bB_{\|\cdot\|}(x,r)} \left| \lno \bg(\xi) \rno - \lno \bg(x) \rno \right| d\calH_{\|\cdot\|}^m \hel M (\xi)  \,.
\end{equation*}
Assuming without loss of generality that $x=0$ and that $\lno \bg(0) \rno > 0$ the proof then proceeds as that of \ref{55} with a few changes. We keep the same notation as there. If $\pi \in H_\kappa$ then for each $y \in \pi(M \cap \bB_{\|\cdot\|}(0,r))$ there is a unique $\xi = \pi^{-1}\{y\} \cap M \cap \bB_{\|\cdot\|}(0,r)$ and one has (here $\bchi : \calR_0(X,G) \to G$ denote the augmentation map)
\begin{equation*}
\begin{split}
\bg_{\pi_\# (T \shel \bB_{\|\cdot\|}(0,r))}(y) & = \bchi \left( \la T \hel \bB_{\|\cdot\|}(0,r) , \pi , y \ra \right) \\
& = \bchi \left( \bdelta_\xi\, \bg_{T \shel \bB_{\|\cdot\|}(0,r)}(\xi) \right) \\
& = \bg_{T \shel \bB_{\|\cdot\|}(0,r)}(\xi) \,.
\end{split}
\end{equation*}
Now \eqref{eq.thu.17} becomes
\begin{multline}
\label{eq.thu.19}
\calM_{H,|\cdot|}\left( \pi_\# \left(T \hel\bB_{\|\cdot\|}(0,r)\right) \right) = \int_{\pi(F(D_r))} \lno \bg_{\pi_\# (T \shel \bB_{\|\cdot\|}(0,r))}(y) \rno d\calH^m_{|\cdot|}(y) \\
= \int_{W \cap D_r} \lno \bg_{T \shel \bB_{\|\cdot\|}(0,r)}(F(x)) \rno J_m(\pi \circ F)(x)d\calH^m_{|\cdot|}(x) \,.
\end{multline}
We henceforth abbreviate $\bg = \bg_{T \shel \bB_{\|\cdot\|}(0,r)}$.
Combining \eqref{eq.thu.19} and \eqref{eq.thu.18}, we obtain the following replacement for \eqref{eq.thu.11}
\begin{equation}
\label{eq.thu.20}
\begin{split}
\bigg| & \calM_{H,|\cdot|}\left( \pi_\# \left(T \hel\bB_{\|\cdot\|}(0,r)\right) \right) - \lno \bg(0) \rno \calH^m_{|\cdot|}\left( \pi(D_r) \right) \bigg| \\
& \leq \int_{W \cap D_r} \bigg|  \lno \bg(F(x)) \rno J_m(\pi \circ F)(x) - \lno \bg(0) \rno J_m\pi(x) \bigg| d\calH^m_{|\cdot|}(x) \\
& \leq \int_{W \cap D_r} \bigg| \lno \bg(F(x)) \rno - \lno \bg(0) \rno\bigg| J_m(\pi \circ F)(x)  d\calH^m_{|\cdot|}(x) \\
& \quad \quad + \lno \bg(0)\rno \int_{W \cap D_r} \left| J_m(\pi \circ F)(x) - J_m\pi(x) \right|d\calH^m_{|\cdot|}(x) \\
& = \rmI + \rmII \,.
\end{split}
\end{equation}
The second term $\rmII$ is bounded above in the exact same way as in the proof of \ref{55}, see \eqref{eq.thu.12}, \eqref{eq.thu.13} and the lines thereafter:
\begin{multline}
\label{eq.thu.21}
\int_{H_\kappa} d\bmu(\pi) \int_{W \cap D_r} \left| J_m(\pi \circ F)(x) - J_m\pi(x) \right|d\calH^m_{|\cdot|}(x) \\ \leq \hat{\veps} \bc_{ \ref{54}}(n,m)C^2(1+C^2)^m (\rmcard \Lambda(n,m)) \balpha(m)r^m\,.
\end{multline}
In order to bound $\rmI$ from above we use our restriction to $0$ being a Lebesgue point of $\lno g(\cdot) \rno$. Choose $r_3 > 0$ small enough (according to $\hat{\veps}$ and $\kappa$) for 
\begin{equation}
\label{eq.thu.22}
\int_{\bB_{\|\cdot\|}(0,r)} \bigg| \lno \bg(\xi) \rno - \lno \bg(0) \rno \bigg| d\calH^m_{|\cdot|} \hel M (\xi ) \leq \frac{\hat{\veps}}{2^m \kappa^m} \balpha(m)r^m
\end{equation}
whenever $0 < r \leq r_3$. We notice that for $x \in D_r$,
\begin{equation*}
J_m(\pi \circ F)(x) \leq \left( \rmLip \pi \circ F \right)^m \leq \kappa^m \left( 1 + \hat{\veps} \right)^m \leq \kappa^m \left( 1 + \hat{\veps} \right)^m J_mF(x) \,.
\end{equation*}
We now apply the area formula \cite[3.2.22]{GMT},
\begin{equation}
\label{eq.thu.23}
\begin{split}
\int_{W \cap D_r}& \bigg| \lno \bg(F(x)) \rno - \lno \bg(0) \rno\bigg| J_m(\pi \circ F)(x)  d\calH^m_{|\cdot|}(x)\\
& \leq \kappa^m\left(1+\hat{\veps}\right)^m\int_{W \cap D_r}\bigg| \lno \bg(F(x)) \rno - \lno \bg(0) \rno\bigg| J_m F(x)  d\calH^m_{|\cdot|}(x)\\
& = \kappa^m\left(1+\hat{\veps}\right)^m\int_{M \cap \bB_{\|\cdot\|}(0,r)} \bigg| \lno \bg(\xi) \rno - \lno \bg(0) \rno \bigg| d\calH^m_{|\cdot|} (\xi ) \\
& \leq \hat{\veps} \balpha(m)r^m \,.
\end{split}
\end{equation}
Now multiplying both members of \eqref{eq.thu.20} by $\psi(W)$ it follows from \eqref{eq.thu.21} and \eqref{eq.thu.23} that
\begin{multline}
\label{eq.thu.25}
\Bigg|\int_{H_\kappa} \calM_{H,\|\cdot\|}\left( \pi_\# \left(T \hel\bB_{\|\cdot\|}(0,r)\right) \right) d\bmu_W(\pi) - \lno \bg(0) \rno \int_{H_\kappa} \calH^m_{\|\cdot\|}\left( \pi(D_r) \right) d\bmu_W(\pi)\Bigg| \\
\leq \hat{\veps} \bC \balpha(m)r^m
\end{multline}
where
\begin{equation*}
\bC = C \bigg( \bc_{\ref{54}}(n,m) C^2(1+c^2)^m (\rmcard \Lambda(n,m)) + \lno \bg(0) \rno C (1+C^2)^m +1 \bigg)\,.
\end{equation*}
Recalling \eqref{eq.thu.24} it ensues from \eqref{eq.thu.25} that
\begin{multline*}
\int_{\rmHom_m(X,X)} \calM_{H,\|\cdot\|}\left( \pi_\# \left(T \hel\bB_{\|\cdot\|}(0,r)\right) \right) d\bmu_W(\pi)\\ \geq \bigg( \left( 1-\hat{\veps}\right) \left( 1- C \hat{\veps} \right)^m - \bC \lno \bg(0) \rno^{-1} \hat{\veps} \bigg) \balpha(m)r^m \lno \bg(0)  \rno \,.
\end{multline*}
Choosing $r_2 > 0$ according to B. Kirchheim's area formula as at the end of the proof of \ref{55} we infer that if also $0 < r \leq r_2$ then
\begin{multline*}
\calM_H\left(T \hel \bB_{\|\cdot\|}(0,r)\right)= \int_{\bB_{\|\cdot\|}(0,r)} \lno \bg(\xi)\rno d\left(\calH^m_{\|\cdot\|} \hel M \right)(\xi) \\
\leq \lno \bg(0) \rno \calH^m_{\|\cdot\|} \left( M \cap \bB_{\|\cdot\|}(0,r) \right) +  \int_{\bB_{\|\cdot\|}(0,r)} \bigg| \lno \bg(\xi) \rno - \lno \bg(0) \rno \bigg| d\calH^m_{|\cdot|} \hel M (\xi ) \\
\leq \left( 1 + \hat{\veps} \right) \lno \bg(0) \rno \balpha(m)r^m + \hat{\veps}C^m \balpha(m)r^m \,.
\end{multline*}
Letting $r_0 = \min \{ r_1/C,r_2,r_3\}$ it becomes clear that
\begin{equation*}
\int_{\rmHom_m(X,X)} \calM_{H}\left( \pi_\# \left(T \hel\bB_{\|\cdot\|}(0,r)\right) \right) d\bmu_W(\pi) \geq (1-\veps) \calM_H\left(T \hel \bB_{\|\cdot\|}(0,r)\right)
\end{equation*}
whenever $0 < r \leq r_0$, provided $\hat{\veps}$ has been chosen small enough according to $\veps$, $n$, $m$ and $\lno \bg(0)\rno$. This readily implies that
\begin{equation*}
\zeta_{\bmu}\left( T \hel \bB_{\|\cdot\|}(0,r) \right) \geq (1-\veps) \calM_H \left( T \hel \bB_{\|\cdot\|}(0,r) \right)
\end{equation*}
for such $0 < r \leq r_0$.
\par 
We now turn to the general case. Since $T$ is concentrated on a countably $(\calH^m_{\|\cdot\|},m)$ rectifiable Borel set it follows from \cite[3.2.29]{GMT} that there exists a disjointed sequence $A_1,A_2,\ldots$ of Borel sets such that $T = T \hel \cup A_k$ and each $A_k$ is contained in some $m$ dimensional submanifold $M_k \subset X$ of class $C^1$. It suffices to show that conclusion holds for $\calH^m_{\|\cdot\|}$ almost every $x \in A_k$, with $k$ fixed. Defining $T_k = T \hel A_k)$ and $R_k = T-T_k = T \hel E_k$ where $Ek = \cup_{j \neq k} A_j$ we claim that for almost every $x \in A_k$ the following holds. Given $0 < \hat{\veps} < 1$ there exists $r(x,\hat{\veps}) > 0$ such that the following three conditions are satisfied for all $0 < r \leq r(x,\hat{\veps})$.
\begin{enumerate}
\item[(A)] $\zeta_{\bmu} \left( T_k \hel \bB_{\|\cdot\|}(x,r) \right) \geq \left( 1 - \hat{\veps} \right) \calM_H \left( T_k \hel \bB_{\|\cdot\|}(x,r) \right)$;
\item[(B)] $\calM_H \left(T \hel \bB_{\|\cdot\|}(x,r) \right) \geq \left( 1 - \hat{\veps} \right) \lno \bg_k(x) \rno \balpha(m) r^m$;
\item[(C)] $\calM_H \left(R_k \cap \bB_{\|\cdot\|}(x,r) \right) \leq \hat{\veps} \balpha(m)r^m$.
\end{enumerate}
Indeed (A) follows from the first part of this proof and (B) follows from the fact that $\calM_H \left(T \hel \bB_{\|\cdot\|}(x,r) \right) \geq \calM_H \left(T_k \hel \bB_{\|\cdot\|}(x,r) \right)$ and that the inequality with $T$ replaced by $T_k$ holds true at small scales according to B. Kirchheim's area formula as in the first part of this proof. Finally we leave it to the reader to check that (C) can be deduced from \cite[2.10.19(4)]{GMT}. Finally, recalling \eqref{eq.thu.16} we infer that
\begin{equation*}
\begin{split}
\zeta_{\bmu} \left( T \hel \bB_{\|\cdot\|}(x,r) \right) & \geq \zeta_{\bmu} \left( T_k \hel \bB_{\|\cdot\|}(x,r) \right) - \zeta_{\bmu} \left( R_k \hel \bB_{\|\cdot\|}(x,r) \right) \\
& \geq \left( 1 - \hat{\veps} \right) \calM_H \left( T_k \hel \bB_{\|\cdot\|}(x,r) \right) - \calM_H \left(R_k \cap \bB_{\|\cdot\|}(x,r) \right) \\
& \geq \left( 1 - \hat{\veps} \right) \calM_H \left( T \hel \bB_{\|\cdot\|}(x,r) \right) - \left( 2 + \hat{\veps} \right) \calM_H \left(R_k \cap \bB_{\|\cdot\|}(x,r) \right) \\
& \geq \left( 1 - \hat{\veps} \left( 1 + \frac{2+\hat{\veps}}{1-\hat{\veps}} \right) \lno \bg_k(x) \rno^{-1} \right) \calM_H \left( T \hel \bB_{\|\cdot\|}(x,r) \right) \,.
\end{split}
\end{equation*}
For such good $x \in A_k$ it should now be clear how to initially choose $\hat{\veps}$ according to $\veps$ and $x$ so that the conclusion holds.
\end{proof}

\begin{Theorem}
\label{64}
Assume that $(X,\|\cdot\|)$ admits density contractors of dimension $m$ and let $\bmu$ be a choice of density contractors. It follows that the Gross mass $\calG^m_{\bmu}$ is lower semicontinuous with respect to $\calF$ convergence on $\calR_m(X,G)$.
\end{Theorem}

\begin{proof}
Let $T,T_1,T_2,\ldots$ be members of $\calR_m(X,G)$ such that $\calF(T-T_j) \to 0$ as $j \to \infty$. Given $\veps > 0$ we apply \ref{63} and \cite[5.2.3(2)]{DEP.HAR.07} to $T$. It follows that for $\calH^m_{\|\cdot\|}$ almost every $x \in \rmset_m \|T\|$ there exists $r(x) > 0$ with the the property that for $\calL^1$ almost every $0 < r \leq r(x)$ one has
\begin{enumerate}
\item[(A)] $(1-\veps) \calM_H \left( T \hel \bB_{\|\cdot\|}(x,r)\right) \leq \zeta_{\bmu}\left((T \hel \bB_{\|\cdot\|}(x,r)\right)$;
\item[(B)] $\calF \left(  T \hel \bB_{\|\cdot\|}(x,r) -  T_j \hel \bB_{\|\cdot\|}(x,r) \right) \to 0$ as $j \to \infty$.
\end{enumerate}
The collection of those balls $\bB_{\|\cdot\|}(x,r)$ for which both conditions (A) and (B) hold thus constitutes a fine covering of a set on which $\|T\|$ concentrates. It follows from \cite[2.8.9 and 2.8.15]{GMT} that it contains a disjointed countable subcollection $B_k = \bB_{\|\cdot\|}(x_k,r_k)$, $k=1,2,\ldots$ such that $\calM_G(T) = \sum_k \calM_G(T \hel B_k)$. Therefore
\begin{equation*}
\begin{split}
(1-\veps) \calM_G(T) & = (1-\veps)\sum_k \calM_G(T \hel B_k) \\
& \leq \sum_k \zeta_{\bmu} (T  \hel B_k) \\
\intertext{(by (A))}
& \leq \sum_k \liminf_j \zeta_{\bmu}(T_j \hel B_k) \\
\intertext{(by (B) and \ref{62})}
& \leq \liminf_j \sum_k \zeta_{\bmu}(T_j \hel B_k) \\
&\leq \liminf_j \sum_k \calM_G(T_j \hel B_k) \\
\intertext{(by \eqref{eq.thu.16} and \eqref{eq.thu.14})}
& \leq \liminf_j \calM_G(T_j) \,.
\end{split}
\end{equation*}
Since $\veps > 0$ is arbitrary the proof is complete.
\end{proof}

\section{Mass minimizing chains of dimension 2 or of codimension 1}

We are now able to dispense with hypothesis (C) of \ref{general.existence} in case of chains of dimension 2 or codimension 1.

\begin{Theorem}
\label{special.existence}
Assume that 
\begin{enumerate}
\item[(A)] $(X,\|\cdot\|)$ is a finite dimensional normed space and $(G,\lno \cdot \rno)$ is an Abelian normed locally compact White group;
\item[(B)] $m=2$ or $m=\dim X - 1$;
\item[(D)] $B \in \calR_{m-1}(X,G)$ and $\partial B = 0$.
\end{enumerate}
It follows that the Plateau problem
\begin{equation*}
(\calP) \begin{cases}
\text{minimize } \calM_H(T) \\
\text{among } T \in \calR_{m}^{\rmloc}(X,G) \text{ such that } \partial T = B
\end{cases}
\end{equation*}
admits a solution.
\end{Theorem}

\begin{proof}
In view of \ref{general.existence} it suffices to show that the Hausdorff mass is lower semicontinuous with respect to $\calF$ convergence, under assumption (B). We have in fact given two distinct proofs of this. First recall that $(X,\|\cdot\|)$ admits density contractors of dimension $m$ with respect to $\calH_{\|\cdot\|}^m$, \ref{existence.d.c}, when $m=2$ or $m=\dim X - 1$. The first proof then goes as follows: The triangle inequality for cycles is established in \ref{tic.d.c} and  this in turn implies the sought for lower semicontinuity according to \ref{theorem.2}. The second proof refers to the lower semicontinuity of a Gross mass, \ref{64}   and the fact that Hausdorff mass and Gross mass coincide, \ref{61}.
\end{proof}

\section{Convex hulls}

In this final section we restrict to the case when $m = \dim X -1$.

\begin{Proposition}
\label{good.proj}
Assume $W_1,\ldots,W_Q$ are affine hyperplanes and $W_1^+,\ldots,W_Q^+$ are half spaces determined by these hyperplanes. Define $C = \cap_{q=1}^Q W_q^+$. There then exists a Lipschitz map $f : X \to X$ with the following properties.
\begin{enumerate}
\item[(1)] $f |_C = \rmid_C$;
\item[(2)] For every $T \in \calR_{\dim X -1}(X,G)$ one has $\calM_H(f_\# T) \leq \calM_H(T)$. Furthermore if $\rmsupp \partial T \subset C$ then $\rmsupp f_\# T \subset C$ and $\partial f_\# T = \partial T$.
\end{enumerate}
\end{Proposition}

\begin{proof}
We each $q=1,\ldots,Q$ we associate a measure contracting projector $\pi_q : X \to X$ on $W_q$ as in \ref{44} (in fact, a translation of those) and we define 
\begin{equation*}
\rho_q : X \to X : x \mapsto \begin{cases} 
x &\text{ if } x \in W_q^+ \\
\pi_q(x) &\text{ if } x \not \in W_q^+ 
\end{cases} \,.
\end{equation*}
It is obvious that each $\rho_q$ is Lipschitz, and so is $f = \rho_Q \circ \rho_{Q-1} \circ \ldots \circ \rho_1$. Readily, $f|_C = \rmid_C$.
Given $S \in \calR_{\dim X -1}(X,G)$ and $q=1,\ldots,Q$ we observe that
\begin{equation*}
\begin{split}
\rho_{q\,\#}S & = \rho_{q\,\#} \bigg( S \hel \left( \rmInt W_q^+ \right)\bigg) +  \rho_{q\,\#}\bigg( S \hel \left( \rmInt W_q^+ \right)^c \bigg)\\
& = S \hel \left( \rmInt W_q^+ \right) + \pi_{q\,\#}\bigg( S \hel \left( \rmInt W_q^+ \right)^c \bigg)
\end{split}
\end{equation*}
according to \cite[5.5.2(1)]{DEP.HAR.07}. It further follows from \ref{411} (applied with the obvious density contractor $\mu = \delta_{\pi_q}$ on $W_q$) that
\begin{equation*}
\calM_H \bigg(\pi_{q\,\#}\bigg( S \hel \left( \rmInt W_q^+ \right)^c \bigg)\bigg) \leq \calM_H\bigg( S \hel \left( \rmInt W_q^+ \right)^c \bigg) \,.
\end{equation*}
We conclude that $\calM_H(\rho_{q\,\#}S) \leq \calM_H(S)$.
\par 
Let $T \in \calR_{\dim X - 1}(X,G)$.
We define inductively $T_q = \rho_{q \, \#}T_{q-1}$, $q=1,\ldots,Q$, where $T_0=T$. Thus $T_Q = f_\# T$. It follows from the preceding paragraph that $\calM_H(T_q) \leq \calM_H(T_{q-1})$ for each $q$, whence $\calM_H(f_\# T ) \leq \calM_H(T)$. In the remaining part of this proof we assume that $\rmsupp \partial T \subset C$. We will establish by induction that for every $q=0,1,\ldots,Q$ the following hold:
\begin{enumerate}
\item[$(A)_q$] $\partial T_q = \partial T$;
\item[$(B)_q$] $\rmsupp T_q \subset \cap_{j=1}^q W_j^+$.
\end{enumerate}
Condition $(A)_0$ is trivially true. Assuming $(A)_{q-1}$ holds we notice that $\partial T_{q} = \partial \rho_{q\,\#}T_{q-1} = \rho_{q\,\#} \partial T_{q-1} = \rho_{q\,\#} \partial T$. Since $\rmsupp \partial T \subset C \subset W_j^+$ and $\rho_{q}|_{W_q^+} = \rmid_{W_q^+}$ it follows that $\partial T_q  = \rho_{q\,\#} \partial T = \partial T$ according to \cite[5.5.2(1)]{DEP.HAR.07}.
\par 
Condition $(B)_1$ is verified because
 $\rmsupp T_1 = \rmsupp \rho_{1\,\#}T \subset \rho_1(X) = W_1^+$. We now assume $(B)_{q-1}$ holds true and we prove $(B)_q$. Abbreviate $E_q = \cap_{j=1}^q W_j^+$. One easily checks that
\begin{equation*}
\rho_{q}(E_{q-1}) \subset W_{q} \cup E_{q} \,.
\end{equation*}
It then ensues from $(B)_{q-1}$ that $\rmsupp T_{q} \subset W_{q} \cup E_{q}$ and it remains to show that $T'_q := T_q \hel (W_q \setminus E_q)=0$. Let $U$ be a component of $W_q \setminus E_q$. It is open, and unbounded (if $x \in U$ and $L \subset W_q$ is a line through $x$, then $L \cap E_q$ is convex, hence an interval and consequently one of the lines in $L$ starting at $x$ is included in $U$). Since $\rmsupp \partial T_q \subset C \subset E_q$ according to $(A)_q$ we infer that $(\partial T'_q) \hel U = 0$. It follows from the constancy theorem \cite[Theorem 6.4]{DEP.HAR.14} that $T'_q = g \lseg U \rseg$ for some $g \in G$ and where an orientation of $W_q$ has been chosen. Since $\rmsupp T'_q \subset \rmsupp T_q$ is compact and since $U$ is unbounded it follows that $g=0$. Accordingly $T'_q$ and the proposition is proved.
\end{proof}

In the following $\rmconv( \rmsupp B)$ denotes the convex hull of $\rmsupp B$.

\begin{Theorem}
\label{existence.convex}
Assume that 
\begin{enumerate}
\item[(A)] $(X,\|\cdot\|)$ is a finite dimensional normed space and $(G,\lno \cdot \rno)$ is an Abelian normed locally compact White group;
\item[(D)] $B \in \calR_{m-1}(X,G)$ and $\partial B = 0$.
\end{enumerate}
It follows that the Plateau problem
\begin{equation*}
(\calP) \begin{cases}
\text{minimize } \calM_H(T) \\
\text{among } T \in \calR_{\dim X - 1}(X,G) \text{ such that } \partial T = B
\end{cases}
\end{equation*}
admits a minimizer $T$ such that $\rmsupp T \subset \rmconv( \rmsupp B)$.
\end{Theorem}

\begin{proof}
Let $C = \rmconv( \rmsupp B)$. For each $x \in \rmBdry C$ there exists a half space $W_x^+ \supset C$ according to Hahn's theorem. The separability of $X^*$ guarantees that $C = \cap_{x \in D} W_x^+$ for some countable subset $D \subset \rmBdry C$. Choose a numbering $D = \{x_1,x_2,\ldots\}$. Choose also a compact convex polytope $C_0$ such that $C \subset C_0$. For each $n=1,2,\ldots$ define $C_n = C_0 \cap \left( \cap_{j=1}^n W^+_{x_j} \right)$. Given a minimizing sequence $T_1,2_2,\ldots$ for problem $(\calP)$ we apply \ref{good.proj} to $C_n$ and $T_k$. We obtain $T_{n,k} = f_{n \, \#} T_k$ and we notice that $T_{n,1},T_{n,2},\ldots$ is another minimizing sequence with $\rmsupp T_{n,k} \subset C_n$ for each $k=1,2,\ldots$. As $C_n \subset C_0$ is compact, the compactness theorem guarantees the existence of a subsequence converging in the flat norm $\calF$ to some $\hat{T}_n \in \calR_{\dim X - 1}(X,G)$ with $\rmsupp \hat{T }_n\subset C_n$. Since $\calM_H$ is $\calF$ lower semicontinuous in codimension 1 we also infer that $\calM_H(\hat{T}_n) = \inf (\calP)$. Applying the compactness theorem to the sequence $\hat{T}_1,\hat{T}_2,\ldots$ we obtain a subsequence $\hat{T}_{\vphi(1)}, \hat{T}_{\vphi(2)},\ldots$ converging in the flat norm to some $T \in \calR_{\dim X - 1}(X,G)$ such that $\partial T = B$ and $\calM_H(T) = \inf(\calP)$. Finally $\rmsupp T \subset C_{\vphi(n)}$ for every $n=1,2,\ldots$ and since the sequence $C_1,C_2,\ldots$ is decreasing we conclude that $\rmsupp T \subset C$.
\end{proof}

\begin{Example}
We close this paper by observing that under the assumption of \ref{existence.convex} there may also exist minimizers $T$ of problem $(\calP)$ such that $\rmsupp T \not \subset \rmconv(\rmsupp B)$. Consider $X = \ell^N_\infty$, $A = [-1,1]^{N-1} \subset \R^{N-1}$ and any $f : A \to \R$ with $|f(x)-f(x')| \leq \|x-x'\|_\infty$ for every $x,x' \in A$ and $f(x)=0$ whenever $x \in \rmBdry A$. Define $F : \R^{N-1} \to \R^N$ by $F(x)=(x,f(x))$. Let also $f_0=0$ and $F_0 = 0$. The key point is that since $\rmLip f \leq 1$ one has $\calH^{N-1}_{\|\cdot\|_\infty}(F(A)) = \calH^{N-1}_{\|\cdot\|_\infty}(F_0(A)) = \balpha(N-1)$ independently of $f$, as the happy reader will easily check. Then for any $G$, any $g \in G \setminus \{0\}$, letting $T_0 = g \lseg A \rseg \in \calR_{N-1}(X,G)$ and $T = F_\#T_0$ we see that $\calM_H(T)=\calM_H(T_0)$ and $\partial T = \partial T_0$. Among those $T$ only one, namely $T_0$, is supported in the convex hull of the support of its boundary.
\end{Example}

%=======================
% BIBLIOGRAPHY AND INDEX
%=======================
\renewcommand{\em}{\it}
\bibliographystyle{siam}
\bibliography{../../../Bibliography/thdp}

\begin{thebibliography}{10}

\bibitem{ALV.THO.04}
{\sc J.~C. \'{A}lvarez Paiva and A.~C. Thompson}, {\em Volumes on normed and
  {F}insler spaces}, in A sampler of {R}iemann-{F}insler geometry, vol.~50 of
  Math. Sci. Res. Inst. Publ., Cambridge Univ. Press, Cambridge, 2004,
  pp.~1--48.

\bibitem{AMB.SCH}
{\sc L.~Ambrosio and T.~Schmidt}, {\em Compactness results for normal currents
  and the {P}lateau problem in dual {B}anach spaces}, Proc. Lond. Math. Soc.
  (3), 106 (2013), pp.~1121--1142.

\bibitem{BOU.DEP}
{\sc P.~Bouafia and T.~{De Pauw}}, {\em Integral geometric measure in separable
  {B}anach space}, Math. Ann., 363 (2015), pp.~269--304.

\bibitem{BUR.IVA.12}
{\sc D.~Burago and S.~Ivanov}, {\em Minimality of planes in normed spaces},
  Geom. Funct. Anal., 22 (2012), pp.~627--638.

\bibitem{BUS.47}
{\sc H.~Busemann}, {\em Intrinsic area}, Ann. of Math. (2), 48 (1947),
  pp.~234--267.

\bibitem{BUS.49}
\leavevmode\vrule height 2pt depth -1.6pt width 23pt, {\em A theorem on convex
  bodies of the {B}runn-{M}inkowski type}, Proc. Nat. Acad. Sci. U. S. A., 35
  (1949), pp.~27--31.

\bibitem{DEP.13.approx}
{\sc T.~{De Pauw}}, {\em Approximation by polyhedral $g$ chains in banach
  spaces}, Z. Anal. Anwend., 33 (2014), pp.~311--334.

\bibitem{DEP.HAR.07}
{\sc T.~{De Pauw} and R.~Hardt}, {\em Rectifiable and flat {$G$} chains in a
  metric space}, Amer. J. Math., 134 (2012), pp.~1--69.

\bibitem{DEP.HAR.14}
\leavevmode\vrule height 2pt depth -1.6pt width 23pt, {\em Some basic theorems
  of flat $g$ chains}, J. Math. Anal. Appl., 418 (2014), pp.~1047--1061.

\bibitem{GMT}
{\sc H.~Federer}, {\em Geometric measure theory}, Die Grundlehren der
  mathematischen Wissenschaften, Band 153, Springer-Verlag New York Inc., New
  York, 1969.

\bibitem{FED.FLE.60}
{\sc H.~Federer and W.~H. Fleming}, {\em Normal and integral currents}, Ann. of
  Math. (2), 72 (1960), pp.~458--520.

\bibitem{FLE.66}
{\sc W.~H. Fleming}, {\em Flat chains over a finite coefficient group}, Trans.
  Amer. Math. Soc., 121 (1966), pp.~160--186.

\bibitem{KIR.94}
{\sc B.~Kirchheim}, {\em Rectifiable metric spaces: local structure and
  regularity of the {H}ausdorff measure}, Proc. Amer. Math. Soc., 121 (1994),
  pp.~113--123.

\bibitem{PARTHASARATHY}
{\sc K.~Parthasarathy}, {\em Probability measures in metric spaces}, Academic
  Press, 1967.

\bibitem{SCH.01}
{\sc R.~Schneider}, {\em On the {B}usemann area in {M}inkowski spaces},
  Beitr\"{a}ge Algebra Geom., 42 (2001), pp.~263--273.

\bibitem{SRIVASTAVA}
{\sc S.~M. Srivastava}, {\em A course on {B}orel sets}, vol.~180 of Graduate
  Texts in Mathematics, Springer-Verlag, New York, 1998.

\bibitem{THOMPSON}
{\sc A.~C. Thompson}, {\em Minkowski geometry}, vol.~63 of Encyclopedia of
  Mathematics and its Applications, Cambridge University Press, Cambridge,
  1996.

\bibitem{WHI.99.deformation}
{\sc B.~White}, {\em The deformation theorem for flat chains}, Acta Math., 183
  (1999), pp.~255--271.

\bibitem{WHI.99.rectifiability}
\leavevmode\vrule height 2pt depth -1.6pt width 23pt, {\em Rectifiability of
  flat chains}, Ann. of Math. (2), 150 (1999), pp.~165--184.

\end{thebibliography}

%\printindex

\end{document}